\documentclass[reqno]{amsart}
\usepackage[T1]{fontenc}
\usepackage[utf8]{inputenc}
\usepackage{amssymb}
\usepackage[mathscr]{eucal}
\usepackage{mathtools}
\usepackage{microtype}
\usepackage{enumerate}
\usepackage{xcolor}
\usepackage{graphicx}
\usepackage[percent]{overpic}
\usepackage{subcaption}
\usepackage[colorlinks=true]{hyperref}
\hypersetup{urlcolor=blue,citecolor=red,linkcolor=blue}
\usepackage[initials]{amsrefs}
\graphicspath{{nls-scenarios-pictures/}}
\numberwithin{equation}{section}
\numberwithin{figure}{section}
\numberwithin{table}{section}
\theoremstyle{plain}
\newtheorem{theorem}{Theorem}[section]
\newtheorem{proposition}[theorem]{Proposition}
\newtheorem{lemma}[theorem]{Lemma}
\newtheorem*{proposition*}{Proposition}
\theoremstyle{remark}
\newtheorem{example}[theorem]{Example}
\newtheorem{remark}[theorem]{Remark}
\newtheorem*{comment*}{Comment}
\newtheorem*{remark*}{Remark}
\theoremstyle{definition}
\newtheorem*{rh-pb*}{Basic RH-problem}
\newtheorem*{assumption*}{Assumption}
\newtheorem*{acknowledgements*}{Acknowledgements}
\providecommand{\B}[1]{\mathbf{#1}}
\providecommand{\BS}[1]{\boldsymbol{#1}}
\providecommand{\C}[1]{\mathcal{#1}}
\providecommand{\D}[1]{\mathbb{#1}}
\providecommand{\R}[1]{\mathrm{#1}}
\newcommand{\dd}{\mathrm{d}}
\newcommand{\eul}{\mathrm{e}}
\newcommand{\ii}{\mathrm{i}}
\providecommand{\abs}[1]{\lvert#1\rvert}
\providecommand{\accol}[1]{\lbrace#1\rbrace}
\providecommand{\croch}[1]{\lbrack#1\rbrack}
\providecommand{\norm}[1]{\lVert#1\rVert}
\renewcommand{\Im}{\operatorname{Im}}
\newcommand{\model}{\operatorname{mod}}
\newcommand{\modell}{\operatorname{mod-}}

\newcommand{\cross}{\mathrm{cr}}
\newcommand{\dist}{\mathrm{dist}}
\newcommand{\merge}{\mathrm{merge}}
\newcommand{\new}{\mathrm{new}}
\newcommand{\ord}{\mathrm{O}}
\newcommand{\osmall}{\mathrm{o}}
\renewcommand{\Re}{\operatorname{Re}}

\DeclareMathOperator{\Res}{Res}
\begin{document}
%---------------------------------------------------------%
\title[NLS with step-like oscillating background: long-time asymptotics]{The focusing NLS equation with step-like oscillating background: scenarios of long-time asymptotics}
%---------------------------------------------------------%
\author[A. Boutet de Monvel]{Anne Boutet de Monvel}
\address{AB: Institut de Math\'ematiques de Jussieu-Paris Rive Gauche, Universit\'e de Paris, 75205 Paris Cedex 13, France.}
\email{anne.boutet-de-monvel@imj-prg.fr}
%---------------------------------------------------------%
\author[J. Lenells]{Jonatan Lenells}
\address{JL: Department of Mathematics, KTH Royal Institute of Technology, 100 44 Stockholm, Sweden.}
\email{jlenells@kth.se}
%---------------------------------------------------------%
\author[D. Shepelsky]{Dmitry Shepelsky}
\address{DS: B.~Verkin Institute for Low Temperature Physics and Engineering, 47 Nauky Avenue, 61103 Kharkiv, Ukraine.} 
\email{shepelsky@yahoo.com}
%---------------------------------------------------------%
\date{}
%---------------------------------------------------------%
\begin{abstract}
We consider the Cauchy problem for the focusing nonlinear Schr\"odinger equation with initial data approaching two different plane waves $A_j\eul^{\ii\phi_j}\eul^{-2\ii B_jx}$, $j=1,2$ as $x\to\pm\infty$. Using Riemann--Hilbert techniques and Deift--Zhou steepest descent arguments, we study the long-time asymptotics of the solution. We detect that each of the cases $B_1<B_2$, $B_1>B_2$, and $B_1=B_2$ deserves a separate analysis. Focusing mainly on the first case, the so-called shock case, we show that there is a wide range of possible asymptotic scenarios. We also propose a method for rigorously establishing the existence of certain higher-genus asymptotic sectors. 
\end{abstract}
%---------------------------------------------------------%
\maketitle
\tableofcontents
%---------------------------------------------------------%
%:s.1
%---------------------------------------------------------%
\section{Introduction}  \label{sec:intro}

We consider the Cauchy problem for the focusing nonlinear Schr\"odinger (NLS) equation
\begin{subequations}\label{nlsic}
\begin{alignat}{2}\label{nls}
&\ii q_t+q_{xx}+2\abs{q}^2q=0,&\qquad&x\in\D{R},\quad t\geq 0,\\
\label{ic}
&q(x,0)=q_0(x),&&x\in\D{R},
\end{alignat}
\end{subequations}
with initial data approaching oscillatory waves at plus and minus infinity:
\begin{equation} \label{q0-limits}
q_0(x)\sim 
\begin{cases}
A_1\eul^{\ii\phi_1}\eul^{-2\ii B_1x},&x\to -\infty,\\
A_2\eul^{\ii\phi_2}\eul^{-2\ii B_2x},&x\to +\infty,
\end{cases}
\end{equation}
where $\accol{A_j,B_j,\phi_j}_1^2$ are real constants such that $A_j>0$. 
Our goal is to describe the long-time behavior of the solution $q(x,t)$ for different choices of the parameters $\accol{A_j,B_j,\phi_j}_1^2$. The tools we use are Riemann--Hilbert (RH) techniques and Deift--Zhou steepest descent arguments.

In order for the formulation of the Cauchy problem \eqref{nlsic}-\eqref{q0-limits} to be complete, it has to be supplemented with boundary conditions for $t>0$. These boundary conditions are the natural extensions of \eqref{q0-limits} to $t>0$ and are given by
\begin{subequations}\label{cauchy-asbg}
\begin{equation}   \label{cauchy-as}
\int_0^{(-1)^j\infty}\abs{q(x,t)-q_{0j}(x,t)}\dd x<\infty\quad\text{for all }t\geq 0,\qquad j=1,2,
\end{equation}
where $q_{0j}(x,t)$, $j=1,2$ are the plane wave solutions of the NLS equation satisfying the initial conditions $q_{0j}(x,0)=A_j\eul^{\ii\phi_j}\eul^{-2\ii B_jx}$, that is,
\begin{equation}  \label{bg}
q_{0j}(x,t)=A_j\eul^{\ii\phi_j}\eul^{-2\ii B_jx+2\ii\omega_jt},\quad\omega_j\coloneqq A_j^2-2B_j^2.
\end{equation}
\end{subequations}

The RH formalism, which can be viewed as a version of the inverse scattering transform (IST) method, is well-developed for problems with ``zero boundary conditions'', that is, for problems where the solution is assumed to decay to $0$ as $x\to\pm\infty$ for each $t\geq 0$. In particular, detailed asymptotic formulas can be derived by employing the steepest descent method for RH problems introduced by Deift and Zhou \cite{DZ93}. The adaptation of the RH formalism and the Deift-Zhou approach to problems with  ``nonzero boundary conditions'' has been the subject of more recent works. 
%---------------------------------------------------------%
%:s.1.1
%---------------------------------------------------------%
\subsection{Previous work on the focusing NLS with nonzero boundary conditions}

The first studies of the focusing NLS equation with nonzero boundary conditions by the IST method were presented in \cites{KI78,Ma79}, where initial profiles satisfying \eqref{q0-limits} with $A_1=A_2$, $\phi_1=\phi_2$, and $B_1=B_2=0$ were considered. In particular, the Ma soliton \cite{Ma79} (also discovered in \cite{KI78}) was introduced. It was also mentioned in \cite{Ma79} that a plane wave solution corresponds to a one-band potential in the spectrum of the Zakharov--Shabat scattering equations, whereas the cnoidal wave (elliptic function) and the multicnoidal wave (hyperelliptic function) solutions correspond to two-band and $N$-band potentials, respectively. A perturbation theory for the NLS equation with non-vanishing boundary conditions was put forward in \cite{GK12}, where particular attention was paid to the stability of the Ma soliton. Whitham theory results for the focusing NLS with step-like data can be found in \cite{B1995}.

An IST approach for initial data satisfying \eqref{q0-limits} with $A_1=A_2$, $\phi_1,\phi_2\in\D{R}$, and $B_1=B_2=0$ was presented in \cite{BK14}, and was further developed in \cites{BM16,BM17}. In particular, it was shown in \cites{BM16,BM17} that for such initial data, the long-time behavior is described by three asymptotic sectors in the $(x,t)$ half-plane $t>0$: two sectors adjacent to the half-axes $x<0$, $t=0$ and $x>0$, $t=0$ in which the solution asymptotes to modulated plane waves, and a middle sector in which the solution asymptotes to an elliptic (genus~$1$) modulated wave. An IST formalism for the case of \emph{asymmetric} nonzero boundary conditions ($A_1\neq A_2$, $\phi_1,\phi_2\in\D{R}$, $B_1=B_2=0$) was presented in \cite{D14}. 

In \cite{BV07}, the long-time asymptotics was studied for the \emph{symmetric shock case} of $A_1=A_2$, $\phi_1=\phi_2$, $B_1=-B_2 < 0$. In this case, the asymptotic picture is symmetric under $x\mapsto -x$. Five asymptotic sectors were described in \cite{BV07}: a central sector containing the half-axis $x=0$, $t>0$ in which the solution $q(x,t)$ asymptotes to a modulated elliptic (genus~$1$) wave \cite{BV07}*{Theorem 1.2}, two contiguous sectors (the transition regions) in which the leading asymptotics is described by modulated hyperelliptic (genus~$2$) waves \cite{BV07}*{Theorem 1.3}, and two sectors adjacent to the $x$-axis in which $q(x,t)$ asymptotes to modulated plane (genus~$0$) waves.

The long-time asymptotics in the case when the left background is zero (i.e., when $A_1=0$ and $A_2\neq 0$) was analyzed in \cite{BKS11}. It was shown that the asymptotic picture involves three sectors in this case: a slow decay sector (adjacent to the negative $x$-axis), a modulated plane wave sector (adjacent to the positive $x$-axis), and a modulated elliptic wave sector (between the first two).

%-------------------%
\begin{remark*}
Although we only consider the focusing version of the NLS equation in this paper, it is worth mentioning that the solution of the defocusing NLS equation with asymmetric nonzero boundary conditions was studied by IST methods in \cite{BP1982} and that extensive results on its long-time behavior were presented in \cite{J2015}.
\end{remark*}
%-------------------%

%---------------------------------------------------------%
%:s.1.2
%---------------------------------------------------------%
\subsection{Summary of results}
The main takeaways of the present paper can be summarized as follows: 

(a) 
Whereas earlier studies focused on specific choices of the parameters $A_j,B_j$, and $\phi_j$, we introduce a RH approach for the solution of \eqref{nlsic} with (solitonless) initial data satisfying \eqref{q0-limits} for general values of $\accol{A_j,B_j,\phi_j}_1^2$ with $B_1\neq B_2$. 

(b) 
We show that the panorama of asymptotic scenarios arising from \eqref{nlsic}-\eqref{q0-limits} is surprisingly rich (some of them can be qualitatively caught using the Whitham modulated equations \cite{Bio18}). In fact, we detect several new scenarios even in the symmetric shock case studied in \cite{BV07}. More precisely, our analysis in Section \ref{sec:shock} shows that the scenario presented in \cite{BV07} is only one of five different possible scenarios in this case. Whereas the long-time behavior along the $t$-axis is always described by a genus~$1$ wave, the asymptotics along the lines $x/t=c$, for small values of $c$, can be either a genus~$1$ (as in \cite{BV07}), a genus~$2$, or a genus~$3$ wave depending on the value of $A_j/(B_2-B_1)$. Asymmetric parameter choices may give rise to an even wider range of possibilities.

(c) 
For each scenario we associate to each asymptotic sector a corresponding $g$-function, which is the basic ingredient of a rigorous asymptotic analysis: it determines a sequence of transformations (``deformations'') of the original RH problem leading to an exactly solvable ``model RH problem'', in terms of which the main asymptotic term can be expressed, through the (now standard) procedures of (i) ``making lenses'' and (ii) estimating the solutions of associated local RH problems (``parametrices''). In the present paper, we give some details of the realization of this approach for the ``rarefaction case'' (with $B_1>B_2$) and we give references to the existing literature where particular cases arising within the ``shock wave case'' (with $B_2>B_1$) were treated.

The asymptotics obtained in this way (in particular, \cites{BLS20b,BLS20c} for the case of $B_2>B_1$), similarly to other cases treated in the literature (e.g., \cites{BM16,BM17} for the case of $B_2=B_1$) do not depend on details of the corresponding initial data and thus manifest the \emph{universality} of the asymptotics.

(d) 
We propose an approach for rigorously establishing the existence of certain higher-genus asymptotic sectors. A sector in which the leading asymptotics of the solution can be expressed in terms of theta functions associated with a genus $g$ Riemann surface is referred to as a genus $g$ sector. At a technical level, such sectors arise when the definition of the so-called $g$-function involves a Riemann surface \cite{DVZ94}. 
In order for the $g$-function to be suitable for the asymptotic analysis, certain parameters appearing in its definition need to satisfy a nonlinear system of equations. The relevant asymptotic sector exists only if this system has a solution. For example, the asymptotic analysis for the genus $2$ sector carried out in \cite{BV07} implicitly assumes that the system of equations \cite{BV07}*{Eqs.~(3.29)} has a solution. In Section \ref{implicitfunctiontheorem}, we establish the existence of this genus $2$ sector rigorously. Although we only provide details for this particular genus $2$ sector, we expect that our approach can be used to show existence also of other genus $g$ sectors appearing in this paper and elsewhere. The approach can be described very briefly as follows. We first show that the existence of a solution of the above-mentioned nonlinear system is equivalent to the existence of a branch of the zero set $F=0$ of a certain mapping $\B{x}\mapsto F(\B{x})$ emanating from a point $\B{x}_0$. The existence of such a branch cannot be immediately deduced from the implicit function theorem because some of entries of the Jacobian matrix of $F$ have singularities at $\B{x}_0$. The central idea of the approach is to introduce a suitably renormalized version $\tilde{F}$ of $F$ which is more amenable to analysis. The construction of $\tilde{F}$ can be illustrated by the following simple one-dimensional example. Consider the function $f\colon (0,1)\to\D{R}$ defined by $f(x)=x\ln x$. This function extends continuously to $x_0=0$, but its first derivative $f'(x)=1+\ln x$ does not. However, the function $\tilde{f}\colon(0,1)\to\D{R}$ defined by
$\tilde{f}(x)=f(x/|\ln x|)$ is such that both $\tilde{f}(x)$ and $\tilde{f}'(x)$ extend continuously to $x_0=0$.
%---------------------------------------------------------%
%:s.1.3
%---------------------------------------------------------%
\subsection{Organization of the paper}

Our analysis is based on a RH formalism which is developed in Section~\ref{sec:rhp}. In Sections~\ref{sec:plane}-\ref{sec:shock}, we analyze the long-time behavior of the solution $q(x,t)$ of \eqref{nlsic}--\eqref{q0-limits}. In Section~\ref{sec:plane}, we show, for any choice of the parameters $A_j$, $\phi_j$, and $B_1\neq B_2$, that the leading behavior of $q$ near the negative and positive halves of the $x$-axis is described by the plane waves $q_{01}$ and $q_{02}$, respectively.

Away from the $x$-axis, the asymptotic analysis turns out to be very different in the two cases $B_1>B_2$ and $B_1<B_2$. Section~\ref{sec:rarefaction} is devoted to the case $B_1>B_2$, called the \emph{rarefaction} case. In this case, the asymptotic picture resembles two copies of that found in \cite{BKS11}, namely, the solution is slowly decaying near the $t$-axis and in two transition sectors the asymptotics has the form of elliptic waves. Section~\ref{sec:shock} is devoted to the case $B_1<B_2$, called the \emph{shock} case. Restricting ourselves to the symmetric case of $A_1=A_2$, $\phi_1,\phi_2\in\D{R}$, and $B_1=-B_2$ (the latter actually being no loss of generality), we describe all the possible asymptotic scenarios that can occur. Finally, in Section~\ref{implicitfunctiontheorem}, we establish the existence of the genus $2$ asymptotic sectors featured in \cite{BV07}. Forthcoming papers will be devoted to a detailed analysis of the asymptotics in a genus~$3$ sector \cites{BLS20b,BLS20c}.
%---------------------------------------------------------%
%:s.1.4
%---------------------------------------------------------%
\subsection{Assumptions}   \label{sec:ass}
Our results are subject to a few assumptions. These assumptions will be stated whenever they are introduced, but are also summarized here for convenience.
\begin{enumerate}[(a)]
\item 
Throughout the paper, we assume that the initial data is such that no solitons are present. 
\item 
The case of $B_1=B_2$ has already been studied extensively in the literature, see \cites{BK14,BM16,BM17,D14}. Thus, from Section \ref{sec:jumps-sigma12} and onwards, we will assume that $B_1\neq B_2$ for conciseness.
\item 
From Section \ref{sec:analytic-continuation} and onwards, we will assume that the initial data $q_0(x)$ is identically equal to the backgrounds outside a compact set, i.e., that there exists a $C>0$ such that $q_0(x)=q_{01}(x,0)$ for $x<-C$ and $q_0(x)=q_{02}(x,0)$ for $x>C$. This allows us to avoid the technical work associated with the introduction of analytic approximations or $\bar{\partial}$ extensions of the jump matrices to perform the steepest descent analysis. This assumption is made purely for convenience and can be relaxed without affecting the structure of the final asymptotic formulas. 
\item 
As already mentioned, when treating the shock case in Section \ref{sec:shock}, we will restrict ourselves to the \emph{symmetric} case of $A_1=A_2$ and $B_2=-B_1 >0$. Asymmetric cases in which $A_1\neq A_2$ and/or $B_2\neq-B_1$ can be analyzed by similar methods, but since the symmetric case is already very rich, we restrict ourselves to this case for definiteness. 
\end{enumerate}
%---------------------------------------------------------%
%:s.2
%---------------------------------------------------------%
\section{The Riemann--Hilbert formalism} \label{sec:rhp}
%---------------------------------------------------------%
%:s.2.1
%---------------------------------------------------------%
\subsection{Notation}
As above, we let $\accol{A_j,B_j,\phi_j}_1^2$ denote real constants such that $A_j>0$. We let $\Sigma_j=\croch{\bar E_j,E_j}$, where $E_j\coloneqq B_j+\ii A_j$, denote the vertical segment $\Sigma_j=\accol{B_j+\ii s\mid\abs{s}\leq A_j}$ oriented upwards; see Figure~\ref{fig:basic-contour} in the cases $B_2<B_1$ (rarefaction) and $B_1<B_2$ (shock). 

We let $\D{C}^+=\accol{\Im k>0}$ and $\D{C}^-=\accol{\Im k<0}$ denote the open upper and lower halves of the complex plane. The Riemann sphere will be denoted by $\bar{\D{C}}=\D{C}\cup\accol{\infty}$. We write $\ln k$ for the logarithm with the principal branch, that is, $\ln k=\ln\abs{k}+\ii\arg k$ where $\arg k\in(-\pi,\pi\rbrack$. Unless specified otherwise, all complex powers will be defined using the principal branch, i.e., $z^{\alpha}=\eul^{\alpha\ln z}$. We let $f^*(k)\coloneqq\overline{f(\bar k)}$ denote the Schwarz conjugate of a function $f(k)$.

Given an open subset $D\subset\bar{\D{C}}$ bounded by a piecewise smooth contour $\Sigma$, we let $\dot E^2(D)$ denote the Smirnoff class consisting of all functions $f(k)$ analytic in $D$ with the property that for each connected component $D_j$ of $D$ there exist curves $\accol{C_n}_1^{\infty}$ in $D_j$ such that the $C_n$ eventually surround each compact subset of $D_j$ and $\sup_{n\geq 1}\norm{f}_{L^2(C_n)}<\infty$. All RH problems in the paper are $2\times 2$ matrix-valued and are formulated in the $L^2$-sense as follows (see \cites{Le17,Le18}): 
\begin{equation}   \label{rhp}
\begin{cases}
m\in I+\dot E^2(\D{C}\setminus\Sigma),&\\
m_+(k)=m_-(k)J(k)&\text{for a.e. }k\in\Sigma,
\end{cases}
\end{equation}
where $m_+$ and $m_-$ denote the boundary values of the solution $m$ from the left and right sides of the contour $\Sigma$. All contours will be invariant under complex conjugation and the jump matrix $J\equiv J(k)$ will always satisfy
\begin{equation}  \label{jump-symm}
J=
\begin{cases}
\sigma_2J^*\sigma_2,&k\in\Sigma\setminus\D{R},\\
\sigma_2(J^*)^{-1}\sigma_2,&k\in\Sigma\cap\D{R},
\end{cases}
\qquad\text{where}\quad\sigma_2\coloneqq \begin{pmatrix}0&-\ii\\\ii&0\end{pmatrix}.
\end{equation}
Together with uniqueness of the solution of the RH problem \eqref{rhp}, this implies the symmetry
\begin{equation}  \label{m-symm}
m=\sigma_2m^*\sigma_2,\quad k\in\D{C}\setminus\Sigma.
\end{equation}

%-------------------%
\begin{remark*}
Smirnoff classes were first introduced in the 1930s \cite{S1932} (see also \cite{KL1937}) as generalizations of the Hardy spaces $H^p$, $p>0$. Whereas Hardy spaces consist of functions analytic in the open unit disk, Smirnoff classes involve functions analytic in a more general open subset $D$. Typically, the Smirnoff class $E^p(D)$, $p>0$, is defined whenever $D$ is a simply connected domain $D\subset\D{C}$ with rectifiable Jordan boundary, see \cite{D1970}. In the context of RH problems, the subset $D$ is often unbounded because the contour passes through infinity. The definition of $E^p(D)$ can be naturally extended to include unbounded domains $D$ by imposing invariance under linear fractional transformations. Moreover, in the context of RH problems involving functions normalized at infinity, it is convenient to use a slight modification $\dot E^p(D)$ of the Smirnoff class $E^p(D)$, where $\dot E^p(D)$ consists of all functions $f$ such that both $f(z)$ and $zf(z)$ belong to $E^p(D)$. We think of $\dot E^p(D)$ as the subspace of $E^p(D)$ of functions that vanish at infinity. If $D$ is bounded, then $\dot{E}^p(D)=E^p(D)$. We refer to \cite{Le18} for further information on Smirnoff classes in the context of RH problems.
\end{remark*}
%-------------------%

%---------------------------------------------------------%
%:s.2.2
%---------------------------------------------------------%
\subsection{Reduction}  \label{sec:reduction}

The study of \eqref{nlsic}-\eqref{q0-limits} can be reduced to one of the following three cases, depending on whether $B_1<B_2$, $B_1>B_2$, or $B_1=B_2$:
\begin{enumerate}[(i)]
\item
$B_1=-1$, $B_2=1$, and $\phi_2=0$;
\item
$B_1=1$, $B_2=-1$, and $\phi_2=0$;
\item
$B_1=B_2=\phi_2=0$.
\end{enumerate}
To see this, note that if $q(x,t)$ satisfies \eqref{nls}, then so does the function $\tilde q(x,t)$ defined by
\[
\tilde q(x,t)\coloneqq Aq(A(x+4Bt),A^2t)\eul^{-2\ii B(x+2Bt)},
\]
for any choice of $A>0$ and $B\in\D{R}$. If $q_0$ satisfies \eqref{q0-limits}, then $\tilde q_0$ satisfies
\[
\tilde q_0(x)\sim\begin{cases}
A_1'\eul^{\ii\phi_1}\eul^{-2\ii B_1'x},&x\to -\infty,\\
A_2'\eul^{\ii\phi_2}\eul^{-2\ii B_2'x},&x\to +\infty,
\end{cases}
\]
where
\[
A_1'=AA_1,\qquad A_2'=AA_2,\qquad B_1'=B_1A+B,\qquad B_2'=B_2A+B.
\]
If $B_2>B_1$, then, by choosing 
\[
A=\frac{2}{B_2-B_1}>0,\qquad B=\frac{B_1+B_2}{B_1-B_2},
\]
we can arrange so that $B_2'=-B_1'=1$. Similarly, if $B_2<B_1$, then, by choosing 
\[
A=\frac{2}{B_1-B_2}>0,\qquad B=\frac{B_1+B_2}{B_2-B_1},
\]
we can arrange so that $B_1'=-B_2'=1$. On the other hand, if $B_1=B_2$, then by choosing $A=1$ and $B=-B_1=-B_2$, we can arrange so that $B_1'=B_2'=0$. Furthermore, in either of these cases, due to the invariance of \eqref{nls} under the global symmetry $q\mapsto q\eul^{\ii\phi}$, we may also assume that $\phi_2=0$ (and thus denote $\phi_1=\phi$). Therefore we may, without loss of generality, restrict our attention to solutions whose initial data satisfy one of the following conditions:
\begin{enumerate}[\textbullet]
\item 
If $B_1<B_2$, then
\begin{equation}
\label{b2b1}
q_0(x)\sim \begin{cases}
A_1\eul^{\ii\phi}\eul^{2\ii x},&x\to -\infty,\\
A_2\eul^{-2\ii x},&x\to +\infty.
\end{cases}
\end{equation}
\item
If $B_1>B_2$, then
\begin{equation}
\label{b1b2}
q_0(x)\sim
\begin{cases}
A_1\eul^{\ii\phi}\eul^{-2\ii x},&x\to -\infty,\\
A_2\eul^{2\ii x},&x\to +\infty.
\end{cases}
\end{equation}
\item
If $B_1=B_2$, then
\begin{equation}  \label{b12}
q_0(x)\sim
\begin{cases}
A_1\eul^{\ii\phi_1},&x\to -\infty,\\
A_2\eul^{\ii\phi_2},&x\to +\infty.
\end{cases}
\end{equation}
\end{enumerate}
However, in what follows we often prefer to keep the setting with arbitrary $B_j$ and $\phi_j$.
%---------------------------------------------------------%
%:s.2.3
%---------------------------------------------------------%
\subsection{Background solutions}

The IST formalism in the form of a RH problem requires that the solution $q(x,t)$ can be represented in terms of the solution of a $2\times 2$-matrix RH problem whose formulation (jump conditions and possible residue conditions) involves only spectral functions which are defined in terms of the initial data. In the adaptation of the IST to case of ``nonzero backgrounds'', the first step is to find a convenient description of the \emph{background solutions} of the Lax pair equations (see, e.g., \cite{BKS11}*{Eqs.~(1.4)-(1.5)}), i.e., the solutions $\Phi_{0j}(x,t,k)$, $j=1,2$ of the equations
\begin{subequations} \label{lax}
\begin{alignat}{3}\label{laxx}
\Phi_x(x,t,k)&=U(x,t,k)\Phi(x,t,k),&\ \ &\text{with}&\ \ &U=-\ii k\sigma_3+\begin{pmatrix}
0 & q\\ -\bar q & 0
\end{pmatrix},\\
\label{laxt}
\Phi_t(x,t,k)&=V(x,t,k)\Phi(x,t,k),&&\text{with}&&V=-2\ii k^2\sigma_3 +2k\begin{pmatrix}
0 & q\\ -\bar q & 0
\end{pmatrix}+\ii\begin{pmatrix}\abs{q}^2&q_x\\\bar q_x&-\abs{q}^2\end{pmatrix},
\end{alignat}
\end{subequations}
where $\sigma_3\coloneqq\left(\begin{smallmatrix}1&0\\0&-1\end{smallmatrix}\right)$ and $q(x,t)=q_{0j}(x,t)$ with $q_{0j}$ as in \eqref{bg}. These solutions $\Phi_{0j}(x,t,k)$ of \eqref{lax} will play the role that $\eul^{(-\ii zx-2\ii z^2t)\sigma_3}$ plays in the case of decaying initial data.

In view of the central role of the RH problem in the IST method, it is natural to try to characterize the background solutions in terms of the solutions of appropriate RH problems.

For $j=1,2$, we introduce the functions
\begin{alignat}{2}  \label{x-om}
X_j(k)&=\sqrt{(k-E_j)(k-\bar E_j)},&\qquad&\Omega_j(k)=2(k+B_j)X_j(k),\\
\label{nu-om}
\nu_j(k)&=\left(\frac{k-E_j}{k-\bar E_j}\right)^{\frac{1}{4}},&&\C{E}_j(k)=\frac{1}{2}\begin{pmatrix}
	\nu_j+\nu_j^{-1} & \nu_j-\nu_j^{-1} \\
	\nu_j-\nu_j^{-1} & \nu_j+\nu_j^{-1}
\end{pmatrix}.
\end{alignat}
We choose the branches of the square and fourth roots so that these functions are analytic in $\D{C}\setminus\Sigma_j$ and satisfy the large $k$ asymptotics
\begin{alignat*}{2}
X_j(k)&=k-B_j+\ord(k^{-1}),&\qquad&\Omega_j(k)=2k^2+\omega_j+\ord(k^{-1}),\\
\nu_j(k)&=1+\ord(k^{-1}),&&\C{E}_j(k)=I+\ord(k^{-1}).
\end{alignat*}
We denote by $X_{j\pm}$, $\Omega_{j\pm}$, $\nu_{j\pm}$, and $\C{E}_{j\pm}$ their boundary values from the left and right sides of $\Sigma_j$. Note that $X_j^*=X_j$, $\Omega_j^*=\Omega_j$, $\nu_j^*=\nu_j^{-1}$, and $\C{E}_j^*=\sigma_2\C{E}_j\sigma_2$. The background solutions $\Phi_{0j}(x,t,k)$, $j=1,2$ are defined as follows:
\begin{subequations}  \label{Phi0-N}
\begin{align}  \label{Phi0}
\Phi_{0j}(x,t,k)&\coloneqq\eul^{(-\ii B_jx+\ii\omega_jt)\sigma_3}
N_j(k)\eul^{(-\ii X_j(k)x-\ii\Omega_j(k)t)\sigma_3},\\
\label{N12}
N_j(k)&\coloneqq\eul^{\frac{\ii\phi_j}{2}\sigma_3}\C{E}_j(k)\eul^{-\frac{\ii\phi_j}{2}\sigma_3}.
\end{align}
\end{subequations}
The functions $N_j$ and $\Phi_{0j}$ are analytic in $k\in\D{C}\setminus\Sigma_j$. They satisfy the relations $\det\Phi_{0j}=\det N_j=\det\C{E}_j\equiv 1$ and the symmetry \eqref{m-symm}. Since $\Sigma_j$ is oriented upwards (see Figure~\ref{fig:basic-contour}), $\nu_{j+}(k)=\ii\nu_{j-}(k)$ for $k\in\Sigma_j$ and thus $N_j$, $j=1,2$ satisfies the RH problem
\begin{equation}  \label{Nj}
\begin{cases}
N_j\in I+\dot E^2(\D{C}\setminus\Sigma_j),&\\[1mm]
N_{j+}(k)=N_{j-}(k)\begin{pmatrix}
	0 & \ii\eul^{\ii\phi_j}\\
	\ii\eul^{-\ii\phi_j} & 0
\end{pmatrix},&k\in\Sigma_j.
\end{cases}
\end{equation}
%---------------------------------------------------------%
%:s.2.4
%---------------------------------------------------------%
\subsection{Jost solutions and spectral functions}\label{sec:jost}

Assuming that $q(x,t)$ satisfies the Cauchy problem defined by \eqref{nlsic} and \eqref{cauchy-asbg}, define the Jost solutions $\Phi_j\equiv\Phi_j(x,t,k)$, $j=1,2$ of the Lax pair equations \eqref{lax} by
\begin{equation}  \label{phij}
\Phi_j(x,t,k)\coloneqq\mu_j(x,t,k)\eul^{(-\ii X_j(k)x-\ii\Omega_j(k)t)\sigma_3},
\end{equation}
where $X_j$, $\Omega_j$ are as in \eqref{x-om}, and $\mu_j$, $j=1,2$ solve the Volterra integral equations
\begin{align}\label{mu}
\mu_j(x,t,k)&=\eul^{(-\ii B_jx +\ii\omega_jt)\sigma_3}N_j(k)\notag\\
&\quad+\int_{(-1)^j\infty}^x\!\!\Phi_{0j}(x,t,k)\Phi_{0j}^{-1}(y,t,k)\croch{(Q-Q_{0j})(y,t)}\mu_j(y,t,k)\eul^{-\ii X_j(k)(y-x)\sigma_3}\,\dd y
\end{align}
with $N_j$ as in \eqref{N12} and
\[
Q=\begin{pmatrix}
0 & q \\ -\bar q & 0  
\end{pmatrix},\qquad
Q_{0j}=\begin{pmatrix}
0 & q_{0j} \\ -\bar q_{0j} & 0
\end{pmatrix}.
\]
The symmetry properties of $\Phi_{0j}$, $N_j$, $X_j$, and $\Omega_j$ imply that $\mu_j$ and $\Phi_j$ satisfy the symmetry \eqref{m-symm}. Observe that $\Phi_j$, $j=1,2$ solve the Volterra integral equations
\begin{equation}  \label{Phi}
\Phi_j(x,t,k)=\Phi_{0j}(x,t,k)+\int_{(-1)^j\infty}^x\!\!\Phi_{0j}(x,t,k)\Phi_{0j}^{-1}(y,t,k)\croch{(Q-Q_{0j})(y,t)}\Phi_j(y,t,k)\,\dd y.
\end{equation}
In what follows $\mu^{(i)}$ denotes the $i$-th column of a matrix $\mu$.

%-------------------%
%:prop 2.1
%-------------------%
\begin{proposition}[analyticity]   \label{analyticity}
The column $\mu_1^{(1)}$ is analytic in $\D{C}^+\setminus\Sigma_1$ with a jump across $\Sigma_1$. The column $\mu_2^{(2)}$ is analytic in $\D{C}^+\setminus\Sigma_2$ with a jump across $\Sigma_2$. The column $\mu_1^{(2)}$ is analytic in $\D{C}^-\setminus\Sigma_1$ with a jump across $\Sigma_1$. The column $\mu_2^{(1)}$ is analytic in $\D{C}^-\setminus\Sigma_2$ with a jump across $\Sigma_2$.
\end{proposition}
%-------------------%

%-------------------%
\begin{proof}
The first and second columns of \eqref{mu} involve the exponentials $\eul^{-\ii X_j(k)(y-x)}$ and $\eul^{\ii X_j(k)(y-x)}$, respectively. Hence the domains of definition of the columns $\mu_j^{(i)}$ are determined by the sign of $\Im X_j$. For example, since the Volterra equation of $\mu_1^{(1)}$ involves the exponential $\eul^{-2\ii X_1(k)(y-x)}$, $\mu_1^{(1)}$ is defined and analytic in the domain $\D{C}^+\setminus\Sigma_1$ where $\Im X_1(k)>0$.
\end{proof}
%-------------------%

For $k\in\Sigma_1\cup\Sigma_2$, one can define the $2\times 2$ matrices $\mu_{j\pm}$ as solutions of \eqref{mu} with $N_j$, $X_j$, $\Omega_j$, and $\Phi_{0j}$ replaced by $N_{j\pm}$, $X_{j\pm}$, $\Omega_{j\pm}$, and $\Phi_{0j\pm}$, respectively. We also define
\[
\Phi_{j\pm}(x,t,k)\coloneqq\mu_{j\pm}(x,t,k)\eul^{(-\ii X_{j\pm}(k)x-\ii\Omega_{j\pm}(k)t)\sigma_3}.
\]
The symmetry properties of $N_{j\pm}$, $X_{j\pm}$, $\Omega_{j\pm}$, and $\Phi_{0j\pm}$ imply that $\mu_j$ and $\Phi_j$ also satisfy \eqref{m-symm}.

For $k\in\D{R}$, $\Phi_2(x,t,k)$ and $\Phi_1(x,t,k)$ are related by a scattering matrix $S(k)$, which is independent of $(x,t)$ and has determinant $1$. The symmetry \eqref{m-symm} implies that $S(k)$ has the same matrix structure as in the case of zero background:
\begin{equation}  \label{scattering}
\begin{split}
\Phi_2(x,t,k)&=\Phi_1(x,t,k)S(k),\quad k\in\D{R},\quad k\neq B_1,B_2,\\
S(k)&=\begin{pmatrix}
a^*(k) & b(k) \\
-b^*(k) & a(k)
\end{pmatrix}.
\end{split}
\end{equation}
By Proposition~\ref{analyticity}, $a(k)$ and $a^*(k)$ are analytic in $\D{C}^+\setminus(\Sigma_1\cup\Sigma_2)$ and $\D{C}^-\setminus(\Sigma_1\cup\Sigma_2)$, respectively, with jumps across $\Sigma_1\cup\Sigma_2$. Moreover, $a(k)=1+\ord(1/k)$ as $k\to\infty$ in $\D{C}^+$ and $b(k)=\ord(1/k)$ as $k\to\infty$ for $k\in\D{R}$. Setting $t=0$ in \eqref{scattering}, it follows that $a(k)$ and $b(k)$ are determined by $q_0(x)$.
%---------------------------------------------------------%
%:s.2.5
%---------------------------------------------------------%
\subsection{The basic RH problem}  \label{sec:rhp-basic}

As in the case of zero background, the analytic and asymptotic properties of $\Phi_{j\pm}$ suggest that we introduce the $2\times 2$ matrix-valued function $m(x,t,k)$ by
\begin{equation}  \label{m}
m(x,t,k)\coloneqq\begin{cases}
\begin{pmatrix}\frac{\Phi_1^{(1)}}{a}&\Phi_2^{(2)}\end{pmatrix}\eul^{(\ii kx+2\ii k^2t)\sigma_3},& k\in\D{C}^+,\\
\begin{pmatrix}\Phi_2^{(1)}&\frac{\Phi_1^{(2)}}{a^*}\end{pmatrix}\eul^{(\ii kx+2\ii k^2t)\sigma_3}, 
&k\in\D{C}^-,
\end{cases}
\end{equation}
and that we characterize $m(x,t,k)$ as the solution of a RH problem whose data are uniquely determined by $q_0(x)$. Since $\Phi_1$ and $\Phi_2$ satisfy \eqref{m-symm}, so does $m$. 
%---------------------------------------%
%:fig 2.1
%---------------------------------------%
\begin{figure}[ht]
\centering\includegraphics[scale=.68]{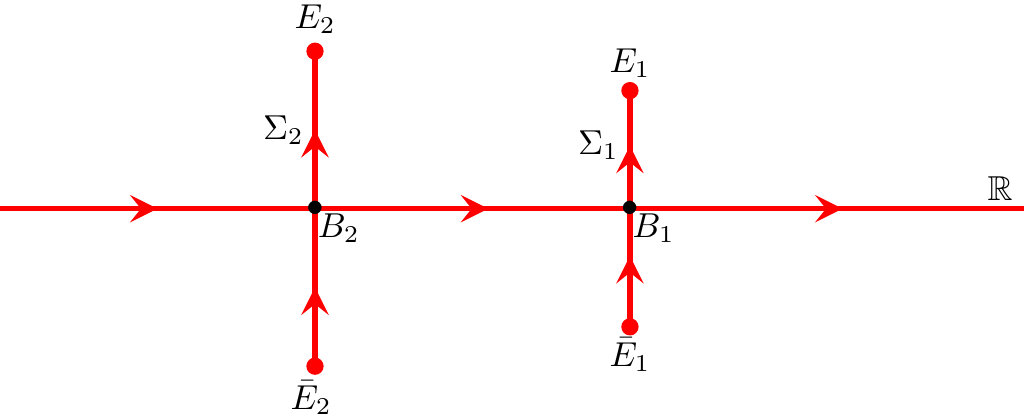}\hspace{2mm}\includegraphics[scale=.68]{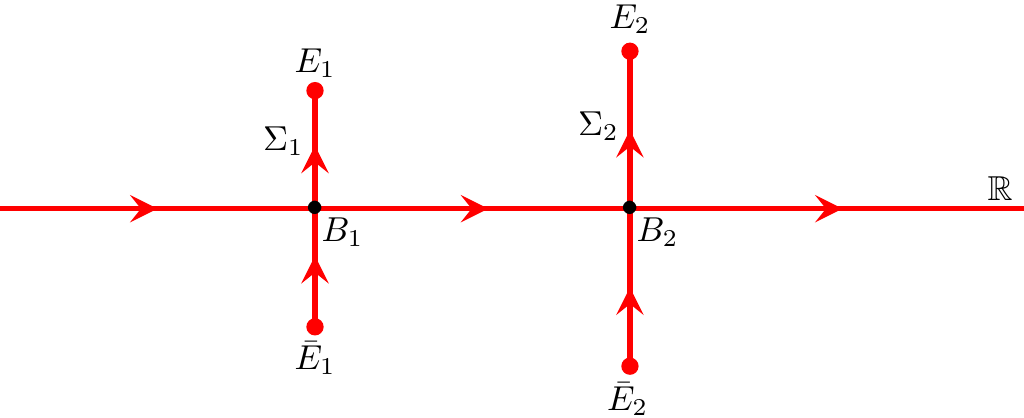}
\caption{The contour $\Sigma=\D{R}\cup\Sigma_1\cup\Sigma_2$ for the basic RH problem in the rarefaction case (left) and shock case (right).} 
\label{fig:basic-contour}
\end{figure}
%---------------------------------------%

For simplicity, we make the following ``no soliton'' assumption:

%-------------------%
\begin{assumption*}
We assume that $a(k)\neq 0$ for $k\in\D{C}^+\cup\D{R}$, $k\neq E_1,E_2$.
\end{assumption*}
%-------------------%

For the behavior of $a(k)$ at the end points of $\Sigma_1$ and $\Sigma_2$, see Section \ref{sec:ends-sigma-12} below.

The function $m$ satisfies the following conditions which will be part of the basic RH problem:
\begin{subequations} \label{rhp-basic-jump}
\begin{equation}   \label{rhp-basic}
\begin{cases}
m(x,t,\,\cdot\,)\in I+\dot E^2(\D{C}\setminus\Sigma),&\\
m_+(x,t,k)=m_-(x,t,k)J(x,t,k)&\text{for a.e. }k\in\Sigma,
\end{cases}
\end{equation}
where $\Sigma\coloneqq\D{R}\cup\Sigma_1\cup\Sigma_2$ and 
\begin{equation}  \label{J-J0}
J(x,t,k)=\eul^{-(\ii kx+2\ii k^2t)\sigma_3}J_0(k)\eul^{(\ii kx+2\ii k^2t)\sigma_3}
\end{equation}
\end{subequations}
for some matrix $J_0(k)$ yet to be specified. Since $m$ obeys \eqref{m-symm}, the matrices $J$ and $J_0$ satisfy the symmetries \eqref{jump-symm}. Our next goal is to determine $J_0(k)$ on each part of the contour $\Sigma$.
%---------------------------------------------------------%
%:s.2.5.1
%---------------------------------------------------------%
\subsubsection{Jump across $\D{R}$}  \label{sec:jump-real-line}

Introduce the reflection coefficient $r(k)$ by
\begin{equation}  \label{reflection}
r(k)\coloneqq\frac{b^*(k)}{a(k)},\quad k\in\D{R},\ \ k\neq B_1,B_2.
\end{equation}
The scattering relation \eqref{scattering} can be rewritten as a jump condition.

%-------------------%
%:lem 2.2
%-------------------%
\begin{lemma}   \label{jump-real-line}
For $k\in\D{R}$, $J_0\equiv J_0(k)$ is given by
\begin{equation}\label{jump-r}
J_0=\begin{pmatrix}
	1+rr^*&r^* \\
	r & 1
\end{pmatrix} =
\begin{pmatrix}
	1 &r^* \\
	0 & 1
\end{pmatrix}
\begin{pmatrix}
	1 & 0 \\
	r & 1
\end{pmatrix},\quad k\in\D{R},\ \ k\neq B_1,B_2.
\end{equation}
\end{lemma}
%-------------------%

%---------------------------------------------------------%
%:s.2.5.2
%---------------------------------------------------------%
\subsubsection{Jumps across $\Sigma_1$ and $\Sigma_2$}\label{sec:jumps-sigma12}

When determining the jump of $m(x,t,k)$ across $\Sigma_1$ and $\Sigma_2$, two cases are to be distinguished.
\begin{enumerate}[1.]
\item 
$\Sigma_1\cap\Sigma_2\neq\emptyset$, i.e.\ $B_1=B_2$.
\item
$\Sigma_1\cap\Sigma_2=\emptyset$, i.e.\ $B_1\neq B_2$.
\end{enumerate}
As already noticed, the first case has attracted more attention in the literature, see \cites{BK14,BM16,BM17,D14}. Henceforth, we therefore only consider the second case, that is, the case $B_1\neq B_2$.

%-------------------%
%:lem 2.3
%-------------------%
\begin{lemma}   \label{jump-sigma-12}
Suppose $B_1\neq B_2$. Then
\begin{equation} \label{jumps-sig}
J_0=\begin{cases}
\begin{pmatrix}
	1 & 0 \\
\frac{\ii\eul^{-\ii\phi_1}}{a_+a_-} & 1
\end{pmatrix},&k\in\Sigma_1\cap\D{C}^+,\\
\begin{pmatrix}
\frac{a_-}{a_+} & \ii\eul^{\ii\phi_2}\\
0 & \frac{a_+}{a_-}
\end{pmatrix}, & k\in\Sigma_2\cap\D{C}^+,
\end{cases}
\qquad
J_0=\begin{cases}
\begin{pmatrix}1&\frac{\ii\eul^{\ii\phi_1}}{a_+^*a_-^*}\\0&1\end{pmatrix},&k\in\Sigma_1\cap\D{C}^-,\\
\begin{pmatrix}\frac{a_+^*}{a_-^*}&0\\\ii\eul^{-\ii\phi_2}&\frac{a_-^*}{a_+^*}\end{pmatrix},&k\in\Sigma_2\cap\D{C}^-.
\end{cases}
\end{equation}
\end{lemma}
%-------------------%

%-------------------%
\begin{proof}
For $k\in\Sigma_1\cup\Sigma_2$, introduce the solutions $\Xi_j(x,t,k)$, $j=1,2$ of the integral equations 
\begin{align*}
&\Xi_j(x,t,k)=\\
&\quad=I+\int_{(-1)^j\infty}^x\Phi_{0j}(x,t,k)\Phi_{0j}^{-1}(y,t,k)\left\lbrack(Q-Q_{0j})(y,t)\right\rbrack\Xi_j(y,t,k)\Phi_{0j}(y,t,k)\Phi_{0j}^{-1}(x,t,k)\dd y.
\end{align*}
For each fixed $(y,t)$, the function $\Phi_{0j}(x,t,k)\Phi_{0j}^{-1}(y,t,k)$ is a solution of the $x$-part \eqref{laxx} with $q$ replaced by $q_{0j}$. Since this solution equals the identity matrix at $x=y$ and the matrix $U$ in \eqref{laxx} is a polynomial in $k$, we conclude that $\Phi_{0j}(x,t,k)\Phi_{0j}^{-1}(y,t,k)$ is an entire function of $k$, well defined for $k\in\Sigma_1\cup\Sigma_2$. Thus, $\Phi_{j\pm}$ and $\Xi_j\Phi_{0j\pm}$ solve the same integral equation for $k\in\Sigma_j$, and $\Phi_j$ and $\Xi_j\Phi_{0j}$ solve the same integral equation for $k\in\Sigma_{j'}$, $j'\neq j$. Hence, $\Phi_{1\pm}(x,t,k)$ and $\Phi_{2\pm}(x,t,k)$ can be written as follows for $k\in\Sigma_1\cup\Sigma_2$:
\begin{subequations}  \label{Phi-12pm}
\begin{alignat}{4}  \label{Phi-1pm}
&\Phi_{1\pm}=\Xi_1\Phi_{01\pm}&\quad&\text{and}&\quad&\Phi_2=\Xi_2\Phi_{02},&\qquad&k\in\Sigma_1,\\
\label{Phi-2pm}
&\Phi_{2\pm}=\Xi_2\Phi_{02\pm}&&\text{and}&&\Phi_1=\Xi_1\Phi_{01},&&k\in\Sigma_2.
\end{alignat}
\end{subequations}
Next, introduce the scattering matrices $S_\pm(k)$ on $\Sigma_1\cup\Sigma_2$:
\begin{subequations}  \label{scat-12}
\begin{alignat}{2} \label{scat-2}
\Phi_{2\pm}(x,t,k)&=\Phi_1(x,t,k)S_\pm(k),&\qquad&k\in\Sigma_2,\\
\label{scat-1}
\Phi_2(x,t,k)&=\Phi_{1\pm}(x,t,k)S_\pm(k),&&k\in\Sigma_1.
\end{alignat}
\end{subequations}
Notice that $\det S_\pm(k)=1$. Let us consider the two cases $k\in\Sigma_2$ and $k\in\Sigma_1$ separately.
\begin{enumerate}[(1)]
\item
For $k\in\Sigma_2$, we use \eqref{scat-2} and \eqref{Phi-2pm} to write $S_\pm(k)=\Phi_1^{-1}(x,t,k)\Xi_2(x,t,k)\Phi_{02\pm}(x,t,k)$.
\end{enumerate}
Setting $x=t=0$ we have $S_\pm(k)=P_2(k)N_{2\pm}(k)$, with $P_2(k)\coloneqq\Phi_1^{-1}(0,0,k)\Xi_2(0,0,k)$. Hence, using \eqref{Nj},
\begin{equation}
S_+(k)=S_-(k)\begin{pmatrix}
0&\ii\eul^{\ii\phi_2}\\
\ii\eul^{-\ii\phi_2}&0
\end{pmatrix},\qquad k\in\Sigma_2.
\label{scat-si-2}
\end{equation}
In particular, 
\begin{equation}
S_{12+}=\ii\eul^{\ii\phi_2}S_{11-},\qquad S_{22+}=\ii\eul^{\ii\phi_2}S_{21-}.
\label{s-2-pm}
\end{equation}
By \eqref{m} the jump relation across $\Sigma_2\cap\D{C}^+$ reads as follows for $x=t=0$:
\[
\begin{pmatrix}
\frac{\Phi_1^{(1)}}{a_+}&\Phi_{2+}^{(2)}
\end{pmatrix}= 
\begin{pmatrix}
\frac{\Phi_1^{(1)}}{a_-}&\Phi_{2-}^{(2)}\end{pmatrix}
\begin{pmatrix}
\frac{a_-}{a_+}&c_2\\0&\frac{a_+}{a_-}
\end{pmatrix}
\]
for some function $c_2\equiv c_2(k)$. Thus 
\[
\frac{\Phi_{2+}^{(2)}}{a_+}-\frac{\Phi_{2-}^{(2)}}{a_-}=\frac{c_2}{a_+ a_-}\Phi_1^{(1)}.
\]
Let us calculate $c_2$. From the scattering relation \eqref{scat-2} we have
\begin{equation}  \label{scat-phi-s}
\Phi_{2\pm}^{(2)}=S_{12\pm}\Phi_1^{(1)}+S_{22\pm}\Phi_1^{(2)}.
\end{equation}
Since $\det\Phi_1=1$ we thus have $\det\left(\Phi_1^{(1)}\ \ \Phi_{2\pm}^{(2)}\right)=S_{22\pm}$. Since $a\coloneqq(\Phi_1^{-1}\Phi_2)_{22}$ (see \eqref{scattering}), we also have $\det\bigl(\Phi_1^{(1)}\ \ \Phi_{2\pm}^{(2)}\bigr)=a_\pm$. Therefore,
\begin{equation}  \label{apm}
S_{22\pm}=a_{\pm},\quad k\in\Sigma_2\cap\D{C}^+.
\end{equation}
From \eqref{scat-phi-s} and \eqref{apm} we obtain
\[
\frac{\Phi_{2+}^{(2)}}{a_+}-\frac{\Phi_{2-}^{(2)}}{a_-}=\left(\frac{S_{12+}}{S_{22+}}-\frac{S_{12-}}{S_{22-}}\right)\Phi_1^{(1)}.
\]
Using \eqref{s-2-pm} and the fact that $\det S_-=1$ we have
\[
\frac{S_{12+}}{S_{22+}}-\frac{S_{12-}}{S_{22-}}=\ii\eul^{\ii\phi_2}\frac{S_{11-}S_{22-}-S_{12-}S_{21-}}{S_{22+}S_{22-}}=\frac{\ii\eul^{\ii\phi_2}}{a_+ a_-}
\]
and thus $c_2=\ii\eul^{\ii\phi_2}$.
\begin{enumerate}[(1)]
\addtocounter{enumi}{1}
\item
For $k\in\Sigma_1$, we use \eqref{scat-1} and \eqref{Phi-1pm} to write $S_\pm(k)=\Phi_{01\pm}^{-1}(x,t,k)\Xi_1^{-1}(x,t,k)\Phi_2(x,t,k)$.
\end{enumerate}
Setting $x=t=0$, this relation reads $S_\pm(k)=\left(N_{1\pm}(k)\right)^{-1}P_1(k)$ with $P_1(k)\coloneqq\Xi_1^{-1}(0,0,k)\Phi_2(0,0,k)$. Hence, by \eqref{Nj}, 
\[
S_-S_+^{-1}=\left(N_{1-}\right)^{-1}N_{1+}=
\begin{pmatrix}
0 & \ii\eul^{\ii\phi_1}\\\ii\eul^{-\ii\phi_1} & 0
\end{pmatrix},
\]
so we have
\begin{equation} \label{scat-si-1}
S_-(k)=\begin{pmatrix}
0 & \ii\eul^{\ii\phi_1} \\ \ii\eul^{-\ii\phi_1} & 0
\end{pmatrix} S_+(k),\qquad k\in\Sigma_1.
\end{equation}
In particular,
\begin{equation}  \label{sss}
S_{21-} =\ii\eul^{-\ii\phi_1}S_{11+},\qquad S_{22-}=\ii\eul^{-\ii\phi_1}S_{12+},\quad k\in\Sigma_1.
\end{equation}
By \eqref{m} the jump relation across $\Sigma_1\cap\D{C}^+$ has the form 
\[
\begin{pmatrix}\frac{\Phi_{1+}^{(1)}}{a_+}&\Phi_2^{(2)}\end{pmatrix}=\begin{pmatrix}\frac{\Phi_{1-}^{(1)}}{a_-}&\Phi_2^{(2)}\end{pmatrix}
\begin{pmatrix}
1 & 0 \\ c_1 & 1
\end{pmatrix} 
\]
for some function $c_1\equiv c_1(k)$. Thus,
\[
\frac{\Phi_{1+}^{(1)}}{a_+}-\frac{\Phi_{1-}^{(1)}}{a_-}=c_1\Phi_2^{(2)}.
\]
On the other hand, from the scattering relation \eqref{scat-1} and $\det S_\pm=1$, we get
\begin{equation} \label{scat-phi-s1}
\Phi_{1\pm}^{(1)}=S_{22\pm}\Phi_2^{(1)}-S_{21\pm}\Phi_2^{(2)}.
\end{equation}
Since $\det\Phi_2=1$, this relation gives $\det\left(\Phi_{1\pm}^{(1)}\ \ \Phi_2^{(2)}\right)=S_{22\pm}$. Since $\det\left(\Phi_{1\pm}^{(1)}\ \ \Phi_2^{(2)}\right)=a_\pm$ we get
\begin{equation}  \label{apm1}
S_{22\pm}=a_{\pm},\quad k\in\Sigma_1\cap\D{C}^+.
\end{equation}
By \eqref{apm1} and \eqref{scat-phi-s1},
\[
\frac{\Phi_{1+}^{(1)}}{a_+}-\frac{\Phi_{1-}^{(1)}}{a_-}=\left(\frac{S_{21-}}{S_{22-}}-\frac{S_{21+}}{S_{22+}}\right)\Phi_2^{(2)}.
\]
As above, using \eqref{sss} and the fact that $\det S_+\equiv 1$,
we arrive at 
\[
\frac{S_{21-}}{S_{22-}}-\frac{S_{21+}}{S_{22+}}=\frac{\ii\eul^{-\ii\phi_1}}{a_+ a_-}
\]	
and thus $c_1=\frac{\ii\eul^{-\ii\phi_1}}{a_+ a_-}$. The expressions for $k\in\Sigma_j\cap\D{C}^-$ follow from the symmetry \eqref{m-symm}.
\end{proof}
%-------------------%

%---------------------------------------------------------%
%:s.2.5.3
%---------------------------------------------------------%
\subsubsection{Jumps across $\Sigma_1$ and $\Sigma_2$ when $a$ and $b$ have analytic continuation}  \label{sec:analytic-continuation}

The analytic properties of the eigenfunctions and spectral functions
discussed in sections \ref{sec:jost}, \ref{sec:jump-real-line}, and \ref{sec:jumps-sigma12} are satisfied if the initial data $q_0(x)$ approach the backgrounds in such a way that the difference is integrable (in $L^1(\pm\infty,0)$), see \eqref{q0-limits} and \eqref{cauchy-asbg}. However, in the remainder of the paper, we make the following assumption on $q_0$ for simplicity.

%-------------------%
\begin{assumption*}[on $q_0$ and $B_j$]
Henceforth, we will assume that $B_1\neq B_2$, that $q_0$ is smooth and that
\begin{equation}  \label{qnot}
q_0(x)=\begin{cases}
A_1\eul^{\ii\phi_1}\eul^{-2\ii B_1x},&x<-C,\\
A_2\eul^{\ii\phi_2}\eul^{-2\ii B_2x},&x>C,
\end{cases}
\end{equation}
for some $C>0$, i.e., that $q_0(x)=q_{01}(x,0)$ for $x<-C$ and $q_0(x)=q_{02}(x,0)$ for $x>C$.
\end{assumption*}
%-------------------%

Then, $a(k)$ and $b(k)$ are both analytic in $\D{C}\setminus(\Sigma_1\cup\Sigma_2)$, and the scattering matrices $S_\pm\equiv S_\pm(k)$ on $\Sigma_1\cup\Sigma_2$ can be written as
\begin{equation}\label{s-ab}
S_\pm=
\begin{pmatrix}
a_\pm^*&b_\pm\\
-b_\pm^*&a_\pm
\end{pmatrix}.
\end{equation}
Accordingly, the relations \eqref{scat-si-2} and \eqref{scat-si-1} between $S_+$ and $S_-$ imply relations amongst $a_\pm(k)$ and $b_\pm(k)$:
\begin{equation}   \label{ab12}
\begin{cases}
a_+=-\ii\eul^{-\ii\phi_1}b_-,&\\
b_+=-\ii\eul^{\ii\phi_1}a_-,&
\end{cases}\ k\in\Sigma_1,\qquad\quad
\begin{cases}
a_+=-\ii\eul^{\ii\phi_2}b_-^*,&\\
b_+=\ii\eul^{\ii\phi_2}a_-^*,&
\end{cases}\ k\in\Sigma_2.
\end{equation}
Moreover, in this case, using that $\det S_-=1$,
\begin{alignat}{2} \label{r+r-sigma1}
r_+(k)-r_-(k)&=\frac{\ii\eul^{-\ii\phi_1}}{a_+(k) a_-(k)},&\qquad&k\in\Sigma_1,\\
\label{tr+tr-sigma2}
\tilde r_+(k)-\tilde r_-(k)&=\frac{\ii\eul^{\ii\phi_2}}{a_+(k)a_-(k)},&&k\in\Sigma_2,
\end{alignat}
where
\begin{equation}  \label{reflection-anal}
r(k)\coloneqq\frac{b^*(k)}{a(k)},\qquad\tilde r(k)\coloneqq\frac{b(k)}{a(k)},
\end{equation}
so that the jump matrix $J_0\equiv J_0(k)$ can be written as follows for $k\in\Sigma_1\cup\Sigma_2$:
\begin{equation}\label{jump-si-anal+}
J_0=\begin{cases}
\begin{pmatrix}
1&0\\r_+-r_-&1\end{pmatrix},&k\in\Sigma_1\cap\D{C}^+,\\[5mm]
\begin{pmatrix}
\frac{a_-}{a_+}&(\tilde r_+-\tilde r_-)a_+a_-\\
0&\frac{a_+}{a_-}
\end{pmatrix},& k\in\Sigma_2\cap\D{C}^+,
\end{cases}
\end{equation}
and
\begin{equation}\label{jump-si-anal-}
J_0=\begin{cases}
\begin{pmatrix}
1&r_-^*-r_+^*\\0&1\end{pmatrix},&k\in\Sigma_1\cap\D{C}^-,\\[5mm]
\begin{pmatrix}
\frac{a_+^*}{a_-^*}&0\\
(\tilde r_-^*-\tilde r_+^*)a_+^*a_-^*&\frac{a_-^*}{a_+^*}
\end{pmatrix},& k\in\Sigma_2\cap\D{C}^-.
\end{cases}
\end{equation}

%---------------------------------------------------------%
%:s.2.5.4
%---------------------------------------------------------%
\subsubsection{Behavior at infinity}  \label{sec:reflec-infty}
Since $q_0$ is smooth, then, as for the problem with zero background \cite{FT87}*{Part one, Chapter I, \S6},
\begin{alignat}{3}  \label{scatt-at-infty}
a(k)&=1+\ord(k^{-1}),&\qquad&k\in\D{C}^+\cup\D{R},&\quad&k\to\infty,\\
b(k)&=\ord(k^{-1}),&&k\in\D{R},&&k\to\infty.\notag
\intertext{Thus,} 
r(k)&=\ord(k^{-1}),&&k\in\D{R},&&k\to\infty.\notag
\end{alignat}

%-------------------%
%:lem 2.4
%-------------------%
\begin{lemma}  \label{lem:reflec-infty}
Under assumption \eqref{qnot} on $q_0$,
\begin{equation}  \label{ab-infty}
a(k)=1+\ord\left(\tfrac{\eul^{4C\abs{\Im k}}}{k}\right)\ \text{ and }\ \ b(k)=\ord\left(\tfrac{\eul^{4C\abs{\Im k}}}{k}\right),\quad k\in\D{C},\ \ k\to\infty.
\end{equation}
Moreover,
\begin{equation}  \label{reflec-infty}
r(k)=\ord\left(\tfrac{\eul^{4C\Im k}}{k}\right),\quad k\in\D{C}^+\cup\D{R},\ \ k\to\infty.
\end{equation}
\end{lemma}
%-------------------%

%-------------------%
\begin{proof}
We first estimate $\Phi_1(x,0,k)$. Introduce 
\begin{alignat*}{2}
&\hat\Phi_1(x,k)\coloneqq\eul^{\ii B_1x\sigma_3}\Phi_1(x,0,k),&\qquad& 
\hat\Phi_{01}(x,k)\coloneqq\eul^{\ii B_1x\sigma_3}\Phi_{01}(x,0,k),\\
&G_1(\tau,k)\coloneqq N_1(k)\eul^{-\ii X_1(k)\tau\sigma_3}N_1^{-1}(k),&&
\hat Q_1(x)\coloneqq\eul^{\ii B_1x\sigma_3}(Q(x,0)-Q_{01}(x,0))\eul^{-\ii B_1x\sigma_3}.
\end{alignat*}
Then, under assumption \eqref{qnot}, the integral equation \eqref{mu} can be written for $t=0$ as a Volterra integral equation for $\hat\Phi_1$:
\[
\hat\Phi_1(x,k)=\hat\Phi_{01}(x,k)+\int_{-C}^xG_1(x-y,k)\hat Q_1(y) \hat\Phi_1(y,k)\dd y,
\]
or, in operator form,
\begin{equation}\label{phi-ie}
\hat\Phi_1=\hat\Phi_{01}+K_1\hat\Phi_1,
\end{equation}
where $K_1$ is an integral operator acting on $\C{C}(\D{R})$ as follows:
\[
(K_1f)(x)=\begin{cases}
\int_{-C}^xG_1(x-y,k)\hat Q_1(y)f(y)\dd y,&x\geq -C,\\
0,&\text{otherwise}.
\end{cases}
\] 
Let $\norm{\;}$ denote some $2\times2$ matrix norm and $\mu_1\coloneqq\abs{\Im X_1(k)}$. We have the estimate
\[
\norm{G_1(\tau,k)}\leq D\eul^{\mu_1\tau},\qquad\tau\geq 0
\]
for some positive constant $D$. Moreover, from \eqref{Phi0-N}, enlarging $D$ if necessary, we get
\[
\norm{\hat\Phi_{01}(x,k)}\leq D\eul^{\mu_1\abs{x}}\leq\begin{cases}D\eul^{\mu_1 C},&-C\leq x<0,\\
D\eul^{\mu_1x},&x\geq 0,\end{cases}
\]
provided $k$ is far from $E_1$ and $\bar E_1$. Equation \eqref{phi-ie} can be solved by the Neumann series
\begin{equation}\label{neumann}
\hat\Phi_1=\sum_{n=0}^{\infty}K_1^n\hat\Phi_{01}.
\end{equation}
We will now prove the estimate
\begin{equation}  \label{KnPhi}
\norm{K_1^n\hat\Phi_{01}(x,k)}\leq D^{n+1}\eul^{\mu_1(x+2C)}\frac{p_1^n(x)}{n!}\,,\quad x\geq-C,
\end{equation}
where $p_1(x)\coloneqq\int_{-C}^x\norm{\hat Q_1(y)}\dd y$. For $n=1$ and $-C\leq x<0$ we indeed have 
\[
\norm{K_1\hat\Phi_{01}(x,k)}\leq\int_{-C}^xD\eul^{\mu_1(x-y)}\norm{\hat Q_1(y)}D\eul^{\mu_1C}\leq D^2\eul^{\mu_1(x+2C)}p_1(x).
\]
Moreover, for $x\geq 0$,
\begin{align*}
\norm{K_1\hat\Phi_{01}(x,k)}
&\leq\int_{-C}^0D\eul^{\mu_1(x-y)}\norm{\hat Q_1(y)}D\eul^{\mu_1C}\dd y+\int_0^xD\eul^{\mu_1(x-y)}\norm{\hat Q_1(y)}D\eul^{\mu_1y}\dd y\\
&\leq D^2\eul^{\mu_1(x+2C)}p_1(x).
\end{align*}
Thus, we are done for $n=1$. Then, using \eqref{KnPhi} for $n-1$ we get the estimate for $n$:
\begin{align*}
\left\lVert K_1\left(K_1^{n-1}\hat\Phi_{01}(x,k)\right)\right\rVert&\leq\int_{-C}^x\norm{G_1(x-y,k)}\norm{\hat Q_1(y)}\norm{K_1^{n-1}\hat\Phi_{01}(x,k)}\dd y\\
&\leq\int_{-C}^xD\eul^{\mu_1(x-y)}p_1'(y)D^n\eul^{\mu_1(y+2C)}\frac{p_1^{n-1}(y)}{(n-1)!}\dd y\\
&=D^{n+1}\eul^{\mu_1(x+2C)}\frac{p_1^n(x)}{n!}.
\end{align*}
Hence, the solution $\hat\Phi_1$ of \eqref{phi-ie} satisfies $\norm{\hat\Phi_1(x,k)}\leq D\eul^{Dp_1(x)}\eul^{\mu_1(x+2C)}$ for $x>-C$, and thus
\begin{subequations}  \label{phi12-est}
\begin{equation}\label{phi1-est}
\norm{\Phi_1(x,0,k)}\leq D\eul^{Dp_1(x)}\eul^{\mu_1(x+2C)},\quad x>-C.
\end{equation}
Since $\det\Phi_1=1$ we have the same estimate for $\norm{\Phi_1^{-1}(x,0,k)}$. Similarly, we get the estimate
\begin{equation}\label{phi2-est}
\norm{\Phi_2(x,0,k)}\leq D\eul^{Dp_2(x)}\eul^{\mu_2(2C-x)},\quad x<C,
\end{equation}
\end{subequations}
where $p_2(x)\coloneqq\int_x^C\norm{\hat Q_2(y)}\dd y$ and $\mu_2\coloneqq\abs{\Im X_2(k)}$. Now, setting $x=t=0$ in \eqref{scattering} and using \eqref{phi12-est} we arrive at the estimates
\begin{equation}\label{ab-est}
\abs{a(k)}\leq\hat D\eul^{4C\mu},\qquad\abs{b(k)}\leq\hat D\eul^{4C\mu},\quad k\in\D{C},
\end{equation}
where $\mu\coloneqq\max(\mu_1,\mu_2)=\abs{\Im k}+\ord\left(\frac{1}{k}\right)$. Further, taking into account the estimates
\begin{alignat*}{2}
G_1(\tau,k)&=\eul^{-\ii X_1(k)\tau\sigma_3}+\ord\left(\tfrac{\eul^{\mu_1\tau}}{k}\right),&\quad&k\to\infty,\\
\hat\Phi_{01}(x,k)&=\eul^{-\ii X_1(k)x\sigma_3}+\ord\left(\tfrac{\eul^{\mu_1\abs{x}}}{k}\right),&&k\to\infty,
\end{alignat*}
one can estimate $(K_1\hat\Phi_{01})(x,k)$ for $x>-C$ as follows:
\[
(K_1\hat\Phi_{01})(x,k)=\int_{-C}^x\eul^{\ii X_1(k)(y-x)\sigma_3}\hat Q_1(y)\eul^{-\ii X_1(k)y\sigma_3}\dd y+\ord\left(\tfrac{\eul^{\mu(x+2C)}}{k}\right),\quad k\to\infty.
\]
Since $\hat Q_1$ is off-diagonal, integrating by parts in the integral produces a factor $\frac{1}{X_1(k)}\sim\frac{1}{k}$, then, the total estimate for $(K_1\hat\Phi_{01})(x,k)$ takes the form $\ord\left(\frac{\eul^{\mu_1(x+2C)}}{k}\right)$. Hence, writing the series \eqref{neumann} as $\hat\Phi_1=\hat\Phi_{01}+\sum_{n=1}^{\infty}K_1^n\hat\Phi_{01}$, we get 
\[
\Phi_1(x,0,k)=\Phi_{01}(x,0,k)+\ord\left(\tfrac{\eul^{\mu_1(x+2C)}}{k}\right),\quad k\to\infty.
\]
By similar arguments,
\[
\Phi_2(x,0,k)=\Phi_{02}(x,0,k)+\ord\left(\tfrac{\eul^{\mu_2(2C-x)}}{k}\right),\quad k\to\infty.
\]
Using these estimates at $x=0$ we get
\[
\Phi_1^{-1}(0,0,k)\Phi_2(0,0,k)=\left(I+\ord\left(\tfrac{1}{k}\right)+\ord\left(\tfrac{\eul^{2C\mu}}{k}\right)\right)\left(I+\ord\left(\tfrac{1}{k}\right)+\ord\left(\tfrac{\eul^{2C\mu}}{k}\right)\right)=I+\ord\left(\tfrac{\eul^{4C\mu}+1}{k}\right).
\]
Thus, estimates \eqref{ab-est} can be improved to 
\[
a(k)=1+\ord\left(\tfrac{\eul^{4C\mu}+1}{k}\right),\qquad b(k)=\ord\left(\tfrac{\eul^{4C\mu}+1}{k}\right),\quad k\to\infty.
\]
This proves \eqref{ab-infty}. Using \eqref{scatt-at-infty}, the estimate \eqref{reflec-infty} follows.
\end{proof}
%-------------------%

%---------------------------------------------------------%
%:s.2.5.5
%---------------------------------------------------------%
\subsubsection{Behavior at the ends of $\Sigma_1$ and $\Sigma_2$}  \label{sec:ends-sigma-12}

We have shown (see \eqref{Phi-12pm} and \eqref{scat-12} in the proof of Lemma~\ref{jump-sigma-12}) that the scattering matrices on $\Sigma_1$ and $\Sigma_2$ can be represented as follows:
\begin{enumerate}[\textbullet]
\item 
for $k\in\Sigma_2$, $S_\pm(k)=P_2(k)\eul^{\frac{\ii\phi_2}{2}\sigma_3}\C{E}_{2\pm}(k)\eul^{-\frac{\ii\phi_2}{2}\sigma_3}$, where $P_2(k)\coloneqq\Phi_1^{-1}(0,0,k)\Xi_2(0,0,k)$ is non-singular at $k=E_2$ and $k=\bar E_2$ with $\det P_2(k)\equiv 1$;
\item
for $k\in\Sigma_1$, $S_\pm(k)=\eul^{\frac{\ii\phi_1}{2}\sigma_3}\C{E}_{1\pm}^{-1}(k)\eul^{-\frac{\ii\phi_1}{2}\sigma_3}P_1(k)$, where $P_1(k)\coloneqq\Xi_1^{-1}(0,0,k)\Phi_2(0,0,k)$ is non-singular at $k=E_1$ and $k=\bar E_1$ with $\det P_1(k)\equiv 1$.
\end{enumerate}
Under assumption \eqref{qnot}, the integral equations determining $\Phi_j$ and $\Xi_j$, $j=1,2$ involve integration over finite intervals and thus the functions $P_j(k)$ are analytic in $\D{C}\setminus(\Sigma_1\cup\Sigma_2)$ whereas the $\Xi_j$ are entire functions. Moreover, $\Phi_1$, and thus $P_2$, is analytic in a vicinity of $E_2$, and $\Phi_2$, and thus $P_1$, is analytic in a vicinity of $E_1$. Consequently, $S(k)=\left(\begin{smallmatrix}a^*(k) & b(k) \\ -b^*(k) & a(k)\end{smallmatrix}\right)$ is analytic in $\D{C}\setminus(\Sigma_1\cup\Sigma_2)$, and the behavior of its entries near $E_j$ and $\bar E_j$ is determined by the behavior of $\nu_j(k)$ involved in $\C{E}_j(k)$. Namely, for $k$ in a vicinity of $E_2$, the representation $S(k)=P_2(k)\eul^{\frac{\ii\phi_2}{2}\sigma_3}\C{E}_2(k)\eul^{-\frac{\ii\phi_2}{2}\sigma_3}$ implies
\[
a(k)=\frac{1}{2\nu_2(k)}\left(-(P_2(k))_{21}\eul^{\ii\phi_2}+(P_2(k))_{22}\right)+\frac{\nu_2(k)}{2}\left((P_2(k))_{21}\eul^{\ii\phi_2}+(P_2(k))_{22}\right).
\]
Thus we have two possibilities:
\begin{enumerate}[(i)]
\item
(generic case) if $(P_2(E_2))_{22}-(P_2(E_2))_{21}\eul^{\ii\phi_2}\neq 0$ then
\[
a(k)=c_{\R{g}}(k-E_2)^{-\frac{1}{4}}+\ord\bigl((k-E_2)^{\frac{1}{4}}\bigr),
\]
where $c_{\R{g}}=\frac{1}{2}(E_2-\bar E_2)^{\frac{1}{4}}\left((P_2(E_2))_{22}-(P_2(E_2))_{21}\eul^{\ii\phi_2}\right)\neq 0$;
\item
(virtual level case) if $(P_2(E_2))_{22}-(P_2(E_2))_{21}\eul^{\ii\phi_2}=0$ then
\[
a(k)=c_{\R{v}}(k-E_2)^{\frac{1}{4}}+\ord\bigl((k-E_2)^{\frac{3}{4}}\bigr),
\]
where $c_{\R{v}}=(E_2-\bar E_2)^{-\frac{1}{4}}(P_2(E_2))_{22}\neq 0$ (the latter inequality is due to $\det P_2\equiv 1$).
\end{enumerate}
Similarly for $k$ near $E_1$.

In the same way, the Jost solutions $\Phi_j$, $j=1,2$ also inherit from $\nu_j$ their singularities at $k=E_j$ and $k=\bar E_j$, see Proposition~\ref{nearbranchpoints} below. Consequently, the singularities (if any) of the entries of $m(x,t,k)$, defined by \eqref{m}, at $k=E_j$ or $\bar E_j$ are, generically, all of order at most $\abs{k-E_j}^{-\frac{1}{4}}$ or $\abs{k-\bar E_j}^{-\frac{1}{4}}$.
 
In the case with virtual level at $k=E_1$, $m^{(1)}$ can have a stronger singularity, of order $\abs{k-E_1}^{-\frac{1}{2}}$ at $k=E_1$ (then $m^{(2)}$ has a singularity of order $\abs{k-\bar E_1}^{-\frac{1}{2}}$ at $k=\bar E_1$). If this is the case, then introducing $\tilde m\coloneqq m\nu_1^{\sigma_3}$, where $\nu_1(k)$ is defined by \eqref{nu-om}, reduces the order of singularities to $-\frac{1}{4}$ and also makes the jump matrix (for $\tilde m$) bounded at $k=E_1$ and $k=\bar E_1$. Indeed, by \eqref{jumps-sig} the (21) entry of the jump matrix for $\tilde m$ near $k=E_1$ involves $\ii\eul^{-\ii\phi_1}\frac{\nu_{1+}\nu_{1-}}{a_+ a_-}$, which is bounded at $k=E_1$.

%-------------------%
\begin{remark*}
Under our assumptions, the possible singularities of $m(x,t,k)$ (or $\tilde m$, in the virtual level case), constructed from the Jost solutions, at the end points of $\Sigma_1$ and $\Sigma_2$, are sufficiently weak to make it possible to proceed with the $L^2$ setting for the RH problem. This is in contrast with other settings of problems with nonzero boundary conditions, e.g., with the case considered in \cite{BMi19}, where $B_1=B_2\eqqcolon B$ (and $A_1=A_2\eqqcolon A$), where a stronger singularity at $E\coloneqq B+\ii A$ (taking the RH problem out of the $L^2$ setting) may correspond to soliton-like structures like rogue waves.
\end{remark*}
%-------------------%

%-------------------%
\begin{remark*}
It is possible to control the behavior of the Jost solutions and the spectral functions at the end points of $\Sigma_1$ and $\Sigma_2$ under much weaker assumptions on the behavior of $q_0(x)$ than \eqref{qnot}, with the same results concerning the singularities. Actually, this can be done assuming that $x(q_0(x)-q_{0j}(x))$ is in $L^1(0,(-1)^j\infty)$. More precisely, we have the following result whose proof is an easy adaptation to the focusing NLS equation of an argument presented in \cite{FLQ20} for the defocusing NLS equation.
\end{remark*}
%-------------------%

%-------------------%
%:prop 2.5
%-------------------%
\begin{proposition}   \label{nearbranchpoints}
Suppose that $(1+|x|)[(q_0-q_{01})(x)]\in L^1((-\infty,0])$. Fix $x\in\D{R}$ and $\epsilon\in(0,\Im E_1)$. Let $B_\epsilon(E_1)$ and $B_\epsilon(\bar{E}_1)$ be the disks of radius $\epsilon$ centered at $E_1$ and $\bar E_1$, respectively. The Jost function $\Phi_1$ satisfies the following estimates for $k$ near the branch points $E_1$ and $\bar{E}_1$:
\begin{subequations}
\begin{alignat}{2}\label{mu1nearE1}
&|\Phi_1^{(1)}(x,0,k)|\leq C|k-E_1|^{-1/4},&\qquad&k\in B_\epsilon(E_1) \setminus\Sigma_1,\\ 
\label{mu1nearE1bar}
&|\Phi_1^{(1)}(x,0,k)|\leq C|k-\bar{E}_1|^{-1/4},&&k\in B_\epsilon(\bar{E}_1)\setminus\Sigma_1.
\end{alignat}
\end{subequations}
\end{proposition}
%-------------------%

%-------------------%
\begin{proof}
Let $x\in\D{R}$ be arbitrary and $\eta(x,k)\coloneqq\mu_1^{(1)}(x,0,k)$. The first column of the Volterra equation \eqref{mu} for $\mu_1$ evaluated at $t = 0$ reads
\begin{equation}\label{etaVolterra}
\eta(x,k)=\eta_0(x,k)+\int_{-\infty}^xE(x,y,k)[(Q-Q_{01})(y,0)]\eta(y,k) \dd y,
\end{equation}
where $\eta_0(x,k)\coloneqq\eul^{-\ii B_1x\sigma_3}N_1^{(1)}(k)$ and
\[
E(x,y,k)\coloneqq\eul^{\ii X_1(k)(x-y)}\Phi_{01}(x,0,k)\Phi_{01}^{-1}(y,0,k).
\]
Using \eqref{Phi0-N}, we get that
\[
E(x,y,k)=\eul^{\frac{\ii\phi_1}{2}\sigma_3}\eul^{-\ii B_1x\sigma_3}F(x,y,k)\eul^{\ii B_1y\sigma_3}\eul^{-\frac{\ii\phi_1}{2}\sigma_3}
\]
where
\[
F(x,y,k)\coloneqq\eul^{\ii(x-y)X_1(k)}\C{E}_1(k)\eul^{-\ii(x-y)X_1(k)\sigma_3}\C{E}_1(k)^{-1}.
\]
Fix $\epsilon\in(0,\Im E_1)$. We will show that
\begin{equation}\label{EboundonB}
|E(x,y,k)|\leq C(1+|x-y|),\quad -\infty<y\leq x,\quad k\in B_\epsilon(E_1)\setminus\Sigma_1.
\end{equation}
Since we have $F(x,y,k)=f(x-y,k)$ with
\[
f(x,k)\coloneqq\C{E}_1(k)\eul^{\ii xX_1(k)(I-\sigma_3)}\C{E}_1(k)^{-1},
\]
the estimate \eqref{EboundonB} will follow if we can show that
\begin{equation}\label{fboundonB}
|f(x,k)|\leq C(1+x),\qquad x\geq 0,\quad k\in B_\epsilon(E_1)\setminus\Sigma_1.
\end{equation}
Differentiating $f$ with respect to $x$, we obtain
\begin{equation}\label{fx}
f_x(x,k)=\ii\eul^{2\ii xX_1(k)}\begin{pmatrix}X_1(k)-k+B_1&-\ii A_1\\
\ii A_1&X_1(k)+k-B_1
\end{pmatrix}.
\end{equation}
Hence, using that $|\eul^{2\ii xX_1(k)}|\leq 1$ for $x\geq 0$ and $k\in \D{C}^+\setminus\Sigma_1$, we get
\begin{align}\label{fxbound}
|f_x(x,k)|\leq C,\quad x\geq 0,\quad k\in B_\epsilon(E_1)\setminus\Sigma_1.
\end{align}
Since $f(x,k)=I+\int_0^xf_x(y,k)\dd y$, the estimate \eqref{fboundonB} follows.

Using the estimate \eqref{EboundonB} of $E$, the solution of the Volterra equation \eqref{etaVolterra} can be constructed in the standard way. Let $K(x,y,k)\coloneqq E(x,y,k)[(Q-Q_{01})(y,0)]$ and define $\eta_n(x,k)$ for any integer $n\geq 1$ by
\begin{equation}\label{Psil}
\eta_n(x,k)\coloneqq\int_{-\infty<x_1\leq\dots\leq x_n\leq x_{n+1}=x}\Bigl(\prod_{i=1}^nK(x_{i+1},x_i,k)\Bigr)\eta_0(x_1,k)\dd x_1\cdots\dd x_n.
\end{equation}

Let $x\in\D{R}$ be fixed. Since $\eta_0(x_1,k)=\eul^{-\ii\frac{\phi_1}{2}(I-\sigma_3)}\eul^{-\ii x_1B_1\sigma_3}\C{E}_1^{(1)}(k)=\ord((k-E_1)^{-1/4})$ as $k\to E_1$ uniformly for $x_1\in\D{R}$, we find, using \eqref{EboundonB},
\begin{align*}
|\eta_n(x,k)| 
&\leq\int_{-\infty<x_1\leq\dots\leq x_n\leq x_{n+1}=x}\Bigl(\prod_{i=1}^n|K(x_{i+1},x_i,k)|\Bigr)|\eta_0(x_1,k)|\dd x_1\cdots\dd x_n\\
&\leq\frac{C}{|k-E_1|^{1/4}}\int_{-\infty<x_1\leq\dots\leq x_n\leq x_{n+1}=x}\Bigl(\prod_{i=1}^nC(1+|x_{i+1}-x_i|)|(Q-Q_{01})(x_i,0)|\Bigr)\dd x_1\cdots\dd x_n\\
&\leq\frac{C}{|k-E_1|^{1/4}}\int_{-\infty<x_1\leq\dots\leq x_n\leq x_{n+1}=x}\Bigl(\prod_{i=1}^nC(1+|x-x_i|)|(Q-Q_{01})(x_i,0)|\Bigr)\dd x_1\cdots\dd x_n\\
&\leq\frac{C}{|k-E_1|^{1/4}}\frac{C^n\|(1+|\,\cdot\,|)(Q-Q_{01})(\,\cdot\,,0)\|_{L^1((-\infty,x])}^n}{n!},\qquad k\in B_\epsilon(E_1)\setminus\Sigma_1.
\end{align*}
Hence the Neumann series
\[
\eta(x,k)=\sum_{n=0}^{\infty}\eta_n(x,k)
\]
converges, and its sum, which solves the Volterra equation \eqref{etaVolterra}, can be estimated as follows:
\begin{align*}
|\eta(x,k)| 
&\leq\sum_{n=0}^\infty|\eta_n(x,k)|\leq\frac{C}{|k-E_1|^{1/4}}\sum_{n=0}^\infty\frac{C^n\|(1+|\,\cdot\,|)(Q-Q_{01})(\,\cdot\,,0)\|_{L^1((-\infty,x])}^n}{n!}\\
&= \frac{C}{|k-E_1|^{1/4}}\eul^{C\|(1+|\,\cdot\,|)(Q-Q_{01})(\,\cdot\,,0)\|_{L^1((-\infty,x])}}\leq\frac{C}{|k-E_1|^{1/4}}
\end{align*}
uniformly for $k\in B_\epsilon(E_1)\setminus\Sigma_1$. Recalling that $\eta=\mu_1^{(1)}$, where $\mu_1$ is related to $\Phi_1$ via \eqref{phij}, we have $|\Phi_1^{(1)}(x,0,k)|=\eul^{x\Im X_1(k)}\abs{\eta(x,k)}$, and this proves \eqref{mu1nearE1}. The estimate \eqref{mu1nearE1bar} follows in the same way using that $\eta_0(x_1,k)=\ord((k-\bar{E}_1)^{-1/4})$ as $k\to\bar{E}_1$.
\end{proof}
%-------------------%

%---------------------------------------------------------%
%:s.2.5.6
%---------------------------------------------------------%
\subsubsection{Spectral functions for pure step initial conditions.}  \label{sec:pure-step}

For pure step initial conditions, i.e.,
\begin{equation} \label{ic-exp}
q_0(x)\coloneqq\begin{cases}
A_1\eul^{\ii\phi_1}\eul^{-2\ii B_1 x},&x<0,\\
A_2\eul^{\ii\phi_2}\eul^{-2\ii B_2 x},&x>0,
\end{cases}
\end{equation}
the spectral functions can be calculated explicitly. In this case, \eqref{scattering} evaluated at $x=t=0$ gives
\[
S(k)\coloneqq\begin{pmatrix}
a^*(k) & b(k) \\ 
-b^*(k) & a(k)\end{pmatrix} 
=N_1^{-1}(k)N_2(k)=\eul^{\frac{\ii\phi_1}{2}\sigma_3}\C{E}_1^{-1}(k)\eul^{-\frac{\ii\phi}{2}\sigma_3}\C{E}_2(k)\eul^{-\frac{\ii\phi_2}{2}\sigma_3},
\]
where $\phi\coloneqq\phi_1-\phi_2$. Thus $a\equiv a(k)$ and $b\equiv b(k)$ are explicitly given by
\begin{align*}
a&=\frac{1}{4}\left\lbrack-\eul^{-\ii\phi}\left(\nu_1-\nu_1^{-1}\right)\left(\nu_2-\nu_2^{-1}\right)+\left(\nu_1+\nu_1^{-1}\right)\left(\nu_2+\nu_2^{-1}\right)\right\rbrack,\\
b&=\frac{1}{4}\left\lbrack\eul^{\ii\phi_2}\left(\nu_1+\nu_1^{-1}\right)\left(\nu_2-\nu_2^{-1}\right)-\eul^{\ii\phi_1}\left(\nu_1-\nu_1^{-1}\right)\left(\nu_2+\nu_2^{-1}\right)\right\rbrack,
\end{align*}
where $\nu_j\equiv\nu_j(k)$, $j=1,2$ are given by \eqref{nu-om}.
%---------------------------------------------------------%
%:s.2.5.7
%---------------------------------------------------------%
\subsubsection{Summary.}  \label{sec:rh-summary}

The basic RH problem is the RH problem defined by \eqref{rhp-basic-jump} with jump $J_0$ given by \eqref{jump-r} and \eqref{jumps-sig}, and complemented by the condition that the possible singularities at the end points of $\Sigma_1$ and $\Sigma_2$ are of order at most $\abs{k-E_j}^{-\frac{1}{4}}$ or $\abs{k-\bar E_j}^{-\frac{1}{4}}$. The latter condition implies that the $L^2$-theory is applicable for the underlying RH problem. In particular, since we assumed that $a(k)\neq 0$ for all $k\in\D{C}^+$ (except, possibly, for $k=E_j$, see Section~\ref{sec:ends-sigma-12}), the solution of this problem is unique.

Recall that the scattering data $a(k)$, $b(k)$, and $r(k)$ are uniquely determined by $q_0(x)$.
 
%-------------------%
\begin{rh-pb*}    \label{basic-rhp}
Given $r(k)$ for $k\in\D{R}$, $a_+(k)$ and $a_-(k)$ for $k\in(\Sigma_1\cup\Sigma_2)\cap\D{C}^+$, find $m(x,t,k)$ analytic in $k\in\D{C}\setminus\Sigma$ that satisfies 
\begin{enumerate}[(i)]
\item
the jump condition \eqref{rhp-basic-jump} completed by \eqref{jump-r} and \eqref{jumps-sig},
\item
the normalization condition $m(x,t,k)\to I$ as $k\to\infty$,
\item
the condition that possible singularities at the end points of $\Sigma_1$ and $\Sigma_2$ are of order at most $\abs{k-E_j}^{-\frac{1}{4}}$ or $\abs{k-\bar E_j}^{-\frac{1}{4}}$.
\end{enumerate}
\end{rh-pb*}
%-------------------%

%-------------------%
%:prop 2.6
%-------------------%
\begin{proposition}  \label{prop:basic-rhp}
Let $m(x,t,k)$ be the solution of the basic RH problem. Then, the solution $q(x,t)$ of the Cauchy problem \eqref{nlsic}-\eqref{q0-limits} is given by
\[
q(x,t)=2\ii\lim_{k\to\infty}km_{12}(x,t,k).
\]
\end{proposition}
%-------------------%

%---------------------------------------------------------%
%:s.3
%---------------------------------------------------------%
\section{Asymptotics: the plane wave region} \label{sec:plane}
%---------------------------------------------------------%
%:s.3.1
%---------------------------------------------------------%
\subsection{Preliminaries}  

The representation of the solution of the Cauchy problem for a nonlinear integrable equation in terms of the solution of an associated RH problem makes it possible to analyze the long-time asymptotics via the Deift--Zhou steepest descent method. Originally, this method was proposed for problems with zero background \cite{DZ93}. Its adaptation to problems with nonzero background has required the development of the so-called $g$-function mechanism \cite{DVZ94}. This mechanism is relevant when some entries of the jump matrix grow exponentially or oscillate as $t\to+\infty$. 
%---------------------------------------%
%:fig 3.1
%---------------------------------------%
\begin{figure}[ht]
\centering\includegraphics[scale=.73]{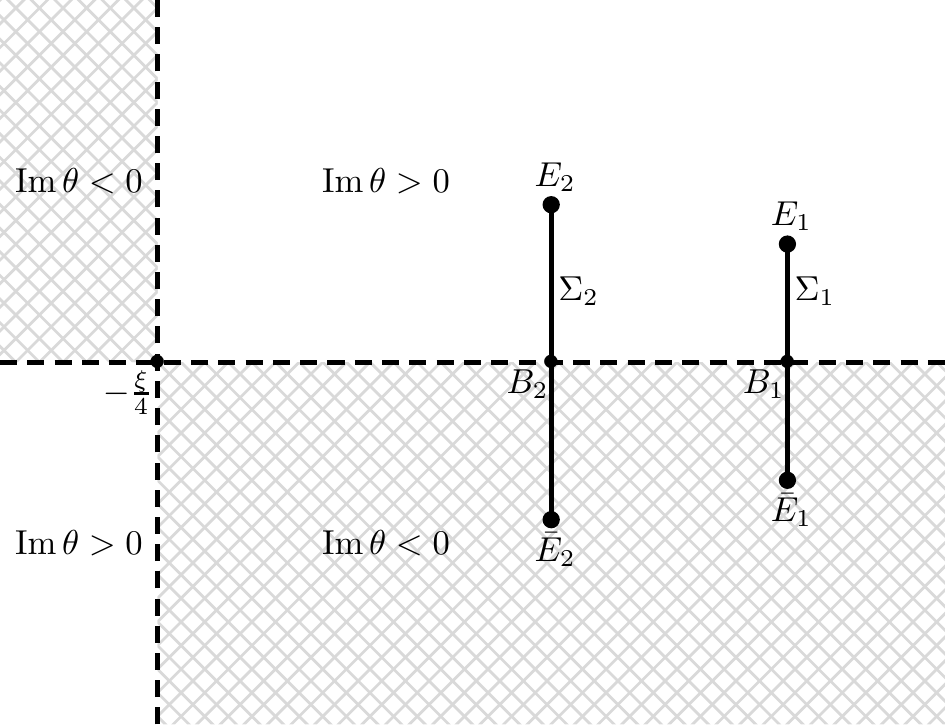}\hspace{4mm}\includegraphics[scale=.73]{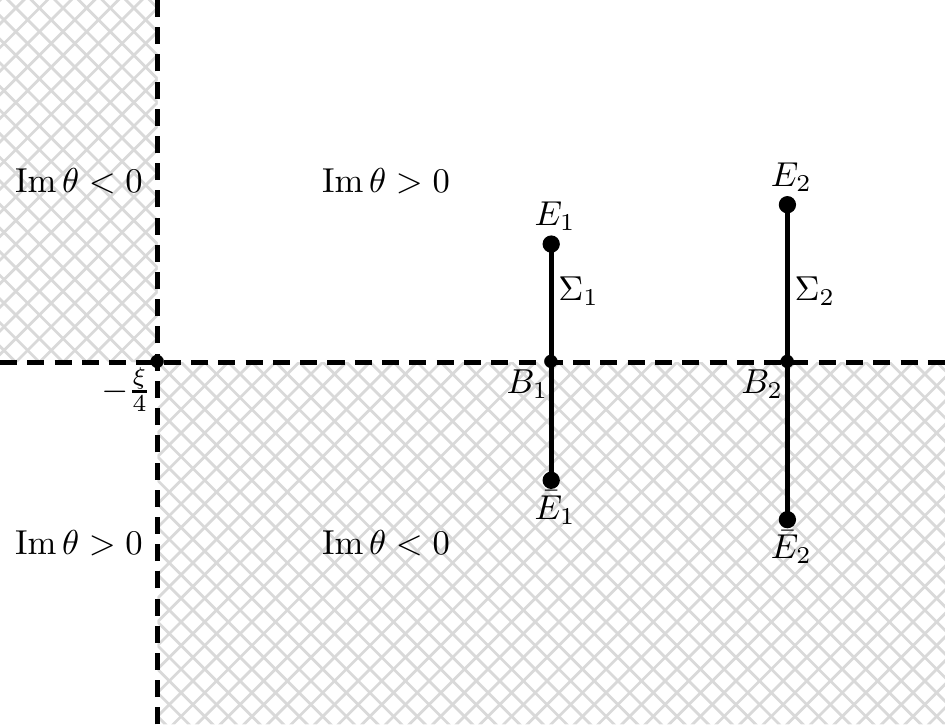}
\caption{Signature table of $\Im\theta(\xi,k)$ for $\xi\gg 0$: rarefaction (left), shock (right)} 
\label{fig:theta-ab}
\end{figure}
%---------------------------------------%

The general idea consists in replacing the original ``phase function''
\begin{equation}   \label{phase}
\theta(\xi,k)\coloneqq 2k^2+\xi k,\qquad\xi\coloneqq\frac{x}{t}
\end{equation}
in the jump matrix (see~\eqref{J-J0})
\[
J(x,t,k)=\eul^{-\ii t\theta(\xi,k)\sigma_3}J_0(k)\eul^{\ii t\theta(\xi,k)\sigma_3}
\]
by another analytic (up to jumps across certain arcs) function $g(\xi,k)$ chosen in such a way that, after appropriate triangular factorizations of the jump matrices and associated redefinitions (``deformations'') of the original RH problem, the jumps containing, originally, exponentially growing entries, become (piecewise) constant matrices (independent of $k$, but dependent, in general, on $x$ and $t$) of special structure whereas the other jumps decay exponentially to the identity matrix. The structure of the ``limiting'' RH problem is such that the problem can be solved explicitly in terms of Riemann theta functions and Abel integrals on Riemann surfaces associated with the limiting RH problem. For different ranges of the parameter $\xi=x/t$, different Riemann surfaces (with different genera) may appear \cites{BM17,BKS11,BV07}.

According to the values of the parameters $A_j$, $B_j$, there are different scenarios. Each of them is characterized by the set of appropriate $g$-functions that we are led to introduce to perform the asymptotic analysis. All these $g$-functions have two properties in common:
\begin{enumerate}[(i)]
\item
the symmetry $g=g^*$,
\item 
the asymptotics
\begin{equation} \label{g-deriv-asy}
g'(\xi,k)=\theta'(\xi,k)+\ord(k^{-2})=4k+\xi+\ord(k^{-2}),\quad k\to\infty,
\end{equation}
where $g'$ and $\theta'$ denote the derivatives of $g$ and $\theta$ with respect to~$k$.
\end{enumerate}
These properties imply that the level set $\Im g(\xi,k)=0$ has two infinite branches: the real axis and another branch which asymptotes to the vertical line $\Re k=-\xi/4$. In what follows the term ``infinite branch'' always refers to this last branch and we call the intersection points of the real axis with the other branches of the level set $\Im g=0$ ``real zeros'' of $\Im g$.

%-------------------%
\begin{remark*}
There are different conventions in the literature for the definition of a $g$-function. In many references, it is the function $\tilde{g}= \frac{1}{2}(\theta-g)$ that is referred to as the $g$-function.
\end{remark*}
%-------------------%

%---------------------------------------------------------%
%:s.3.2
%---------------------------------------------------------%
\subsection{Asymptotics for large $\BS{\abs{\xi}}$: Plane waves} \label{sec:large-xi}

A common fact concerning the long-time asymptotics (that holds for any relationships amongst $B_j$ and $A_j$) for problems with backgrounds satisfying \eqref{q0-limits} is that for $\xi<C_1$ and for $\xi>C_2$, with some $C_j$ that can be expressed in terms of $B_j$ and $A_j$, the solution asymptotes to the corresponding plane waves, with additional phase factors depending on $\xi$. See Figure~\ref{fig:large-xi}.

%---------------------------------------%
%:fig 3.2
%---------------------------------------%
\begin{figure}[ht]
\centering\includegraphics[scale=.73]{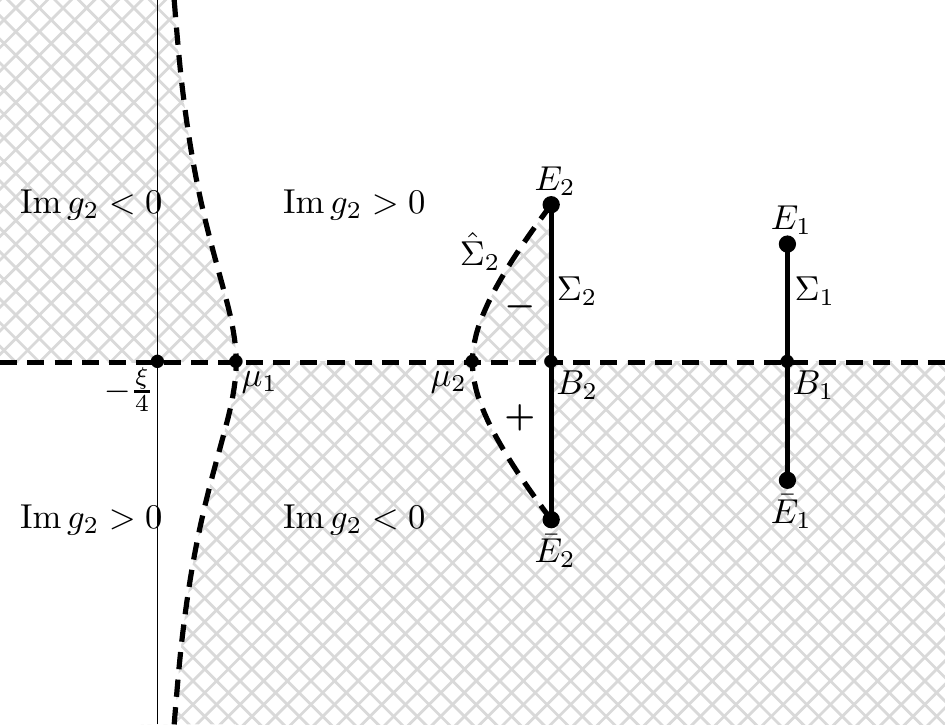}\hspace{4mm}\includegraphics[scale=.73]{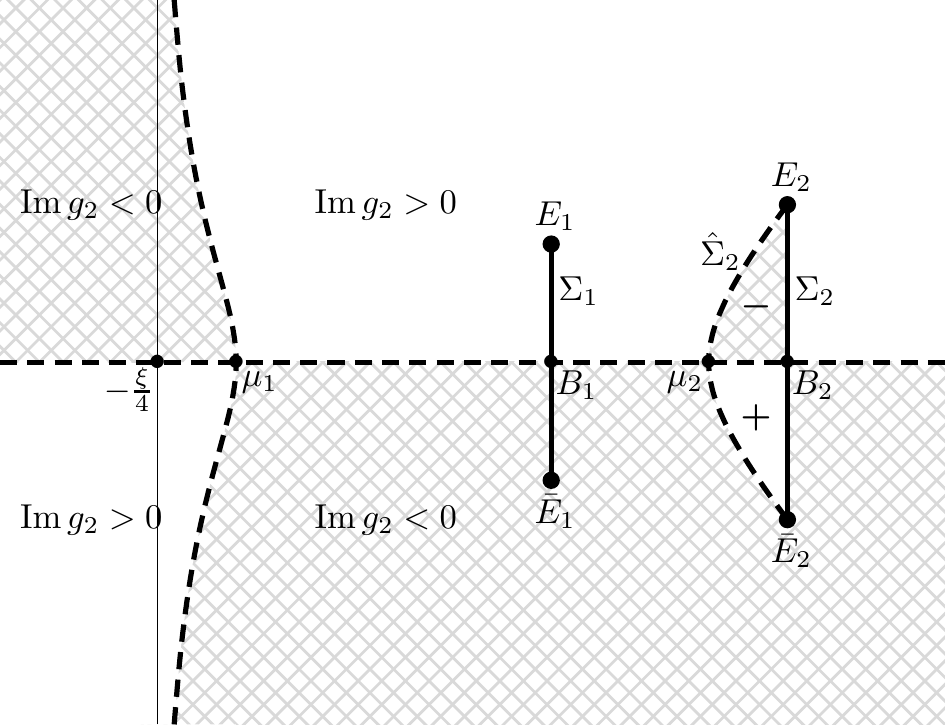}
\caption{Signature table of $\Im g_2(\xi,k)$ for $\xi\gg 0$: rarefaction (left), shock (right)} 
\label{fig:large-xi-ab}
\end{figure}
%---------------------------------------%

Indeed, the ``signature table'' (the distribution of signs of $\Im\theta(\xi,k)$ in the $k$-plane) shows that $J(\xi,k)$ contains exponentially growing entries if $\abs{\xi}\gg 0$. More precisely, for $\xi\ll 0$, the jump across $\Sigma_1$ is growing whereas the jump across the complementary arc $\Sigma_2$ is bounded, and for $\xi\gg 0$, the jump across $\Sigma_2$ is growing whereas the jump across the complementary arc $\Sigma_1$ is bounded (see Figure~\ref{fig:theta-ab}). For such values of $\xi$, we introduce the $g$-functions
\begin{equation}   \label{gj}
g_j(\xi,k)\coloneqq\Omega_j(k)+\xi X_j(k),
\end{equation}
with $j=1$ for $\xi\ll 0$ and $j=2$ for $\xi\gg 0$. These $g$-functions satisfy the above properties $g=g^*$ and \eqref{g-deriv-asy}. Thus, besides $\D{R}$, the level set $\Im g_j=0$ has another infinite branch asymptotic to the line $\Re k=-\frac{\xi}{4}$. It also has a finite branch $\hat\Sigma_j$ connecting $E_j$ and $\bar E_j$ (see Figure~\ref{fig:large-xi-ab}). 

%-------------------%
\begin{remark*}
Here and below, the division of the complex $k$-plane into the regions where $\Im g>0$ and $\Im g<0$ depends on the chosen branch cuts for the square roots involved in the definition of the corresponding $g$-function. In particular, here the cut for $g_j$ (i.e., for $\Omega_j$ and $X_j$) connecting $E_j$ and $\bar E_j$ is the line segment $(E_j,\bar E_j)$.
\end{remark*}
%-------------------%

We consider $m^{(1)}$ defined by
\[
m^{(1)}(x,t,k)\coloneqq\eul^{-\ii tg_j^{(0)}(\xi)\sigma_3}m(x,t,k) \eul^{\ii t(g_j(\xi,k)-\theta(\xi,k))\sigma_3},
\] 
where $j$ is as above and $g_j^{(0)}(\xi)\coloneqq\omega_j-\xi B_j=A_j^2-2B_j^2-\xi B_j$ is defined in such a way that
\begin{equation}
g_j(\xi,k)=2k^2+\xi k+g_j^{(0)}(\xi)+\ord(k^{-1}),\quad k\to\infty.
\label{gj-as}
\end{equation}
In terms of $m^{(1)}$, the jump relation becomes 
\[
m_+^{(1)}(x,t,k)=m_-^{(1)}(x,t,k)J^{(1)}(x,t,k),\quad k\in\Sigma.
\]
For $\xi\gg 0$, the jump $J^{(1)}(x,t,k)$ decays to the identity matrix $I$ as $t\to+\infty$ for $k\in\Sigma_1$, whereas for $k\in\Sigma_2\cap\D{C}^+$ we have (taking into account \eqref{ab12} and \eqref{reflection-anal})
\begin{align*}
J^{(1)}(x,t,k)
&=\begin{pmatrix}
\frac{a_-(k)}{a_+(k)}\eul^{\ii t(g_{2+}(\xi,k)-g_{2-}(\xi,k))}&\ii\eul^{\ii\phi_2}\\
0&\frac{a_+(k)}{a_-(k)}\eul^{-\ii t(g_{2+}(\xi,k)-g_{2-}(\xi,k))}
\end{pmatrix}\\
&=\begin{pmatrix}1&0\\-r_-(k)\eul^{2\ii tg_{2-}(\xi,k)}&1\end{pmatrix}\begin{pmatrix}0&\ii\eul^{\ii\phi_2}\\\ii\eul^{-\ii\phi_2}&0\end{pmatrix}\begin{pmatrix}1&0\\r_+(k)\eul^{2\ii tg_{2+}(\xi,k)}&1\end{pmatrix},
\end{align*}
and similarly for $k\in\Sigma_2\cap\D{C}^-$. The triangular factors above can be absorbed into a transformed RH problem when ``making lenses'' (see \cites{BM17,BKS11,BV07} for details), which finally leads to two model RH problems ($j=1,2$) of the form~\eqref{Nj}:
\begin{equation}  \label{model}
\begin{cases}
m^{\modell j}\in I+\dot E^2(\D{C}\setminus\Sigma_j),&\\[1mm]
m_+^{\modell j}(k)=m_-^{\modell j}(k)\begin{pmatrix}
	0 & \ii\eul^{\ii\phi_j}\\
	\ii\eul^{-\ii\phi_j} & 0
\end{pmatrix},&k\in\Sigma_j,
\end{cases}
\end{equation}
which apply for $(-1)^j\xi\gg 0$ and are explicitly solvable. Returning to $m(x,t,k)$, one obtains the large $t$ asymptotics for 
\[
q(x,t)=2\ii\lim_{k\to\infty}km_{12}(x,t,k)
\]
in the form
\begin{equation}\label{q-as-plane}
q(x,t)=A_j\eul^{-2\ii B_jx+2\ii\omega_jt+\ii\psi_j(\xi)}+\ord(t^{-\frac{1}{2}}),\quad (-1)^j\xi\gg 0,\ j=1,2,
\end{equation}
where $\psi_1(-\infty)=\phi_1$ and $\psi_2(+\infty)=\phi_2$.

%---------------------------------------%
%:fig 3.3
%---------------------------------------%
\begin{figure}[ht]
\centering\includegraphics[scale=.9]{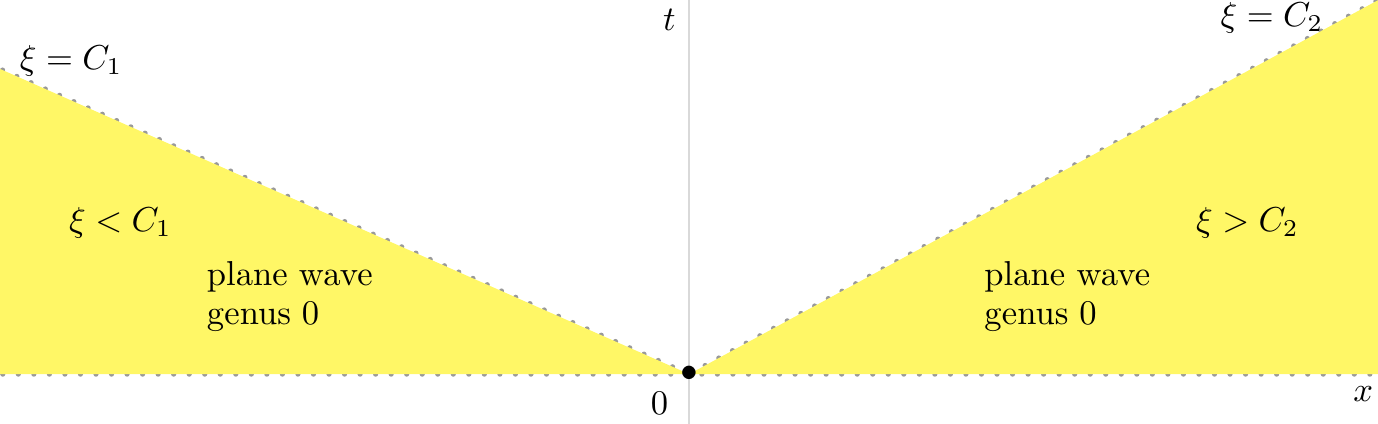}
\caption{The large $\abs{\xi}$ sectors} 
\label{fig:large-xi}
\end{figure}
%---------------------------------------%

%---------------------------------------------------------%
%:s.3.3
%---------------------------------------------------------%
\subsection{Asymptotics in other domains}

The $g$-function presented above is inappropriate in the region between the plane wave sectors $\xi<C_1$ and $\xi>C_2$. The asymptotic picture in this region is sharply different for the two cases 
\begin{itemize}
\item
$B_1>B_2$, rarefaction case,
\item
$B_1<B_2$, shock case.
\end{itemize}
In the following two sections, we study these two cases separately.

%---------------------------------------------------------%
%:s.4
%---------------------------------------------------------%
\section{Asymptotics: the rarefaction case} \label{sec:rarefaction}

In the rarefaction case $B_1>B_2$, the asymptotic picture does not qualitatively depend on the values of the amplitudes $A_1$ and $A_2$ and is actually a doubling of that found in the case where one of the backgrounds is zero, see \cite{BKS11}. The asymptotic picture in the half-plane $t>0$ consists of five sectors: two modulated plane wave sectors, a slow decay sector, and two modulated elliptic wave sectors (also known as transition regions). See Figure~\ref{fig:rarefaction}.
%---------------------------------------%
%:fig 4.1
%---------------------------------------%
\begin{figure}[ht]
\centering\includegraphics[scale=.9]{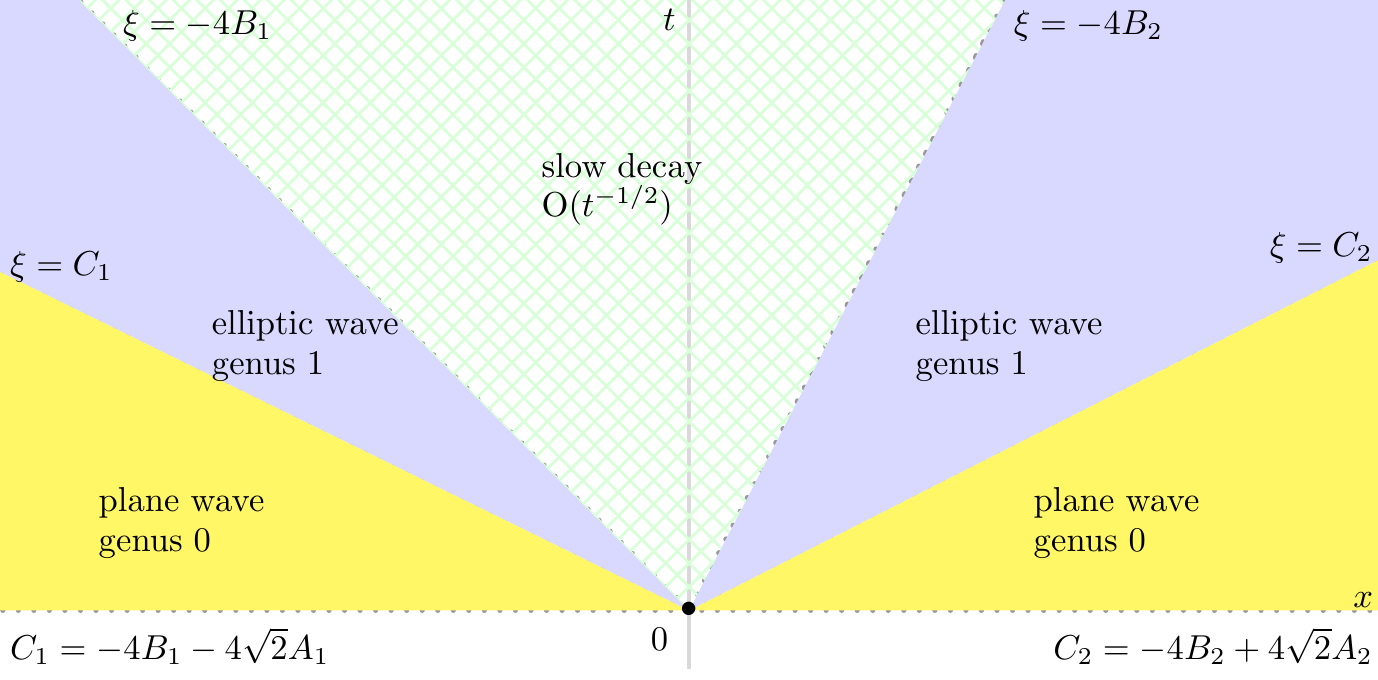}
\caption{The different sectors in the rarefaction case} 
\label{fig:rarefaction}
\end{figure}
%---------------------------------------%

%---------------------------------------------------------%
%:s.4.1
%---------------------------------------------------------%
\subsection{Plane waves: $\BS{\xi<-4B_1-4\sqrt{2}A_1}$ and $\BS{\xi>-4B_2+4\sqrt{2}A_2}$} \label{sec:rare-plane}

We already know that the asymptotics has the form of plane waves for $\xi<C_1$ and $\xi>C_2$, see section~\ref{sec:large-xi}. Here $C_1$ and $C_2$ are given by the same expressions as when one of the backgrounds is zero \cite{BKS11}:
\[
C_1=-4B_1-4\sqrt{2}A_1,\qquad C_2=-4B_2+4\sqrt{2}A_2.
\]
Indeed, suppose first that $\xi \gg 0$. Let $g\equiv g_2(\xi,k)$ be the plane wave $g$-function given by \eqref{gj} and let $g'$ be its derivative with respect to~$k$. In this case,
\begin{equation}   \label{g2}
g'(\xi,k)=4\frac{(k-\mu_1(\xi))(k-\mu_2(\xi))}{\sqrt{(k-E_2)(k-\bar E_2)}}\,,
\end{equation}
where $\mu_j\equiv\mu_j(\xi)$, $j=1,2$, are the two self-intersections of the curve $\Im g_2(\xi,k)=0$:
\begin{equation}   \label{mu12}
\mu_1=\frac{B_2}{2}-\frac{\xi}{8}-\frac{1}{8}\sqrt{(\xi+4B_2)^2-32A_2^2},\qquad\mu_2=\frac{B_2}{2}-\frac{\xi}{8}+\frac{1}{8}\sqrt{(\xi+4B_2)^2-32A_2^2}.
\end{equation}
Therefore, $-\frac{\xi}{4}<\mu_1<\mu_2<B_2$. As $\xi$ decreases, the infinite branch of $\Im g$ moves to the right and $g$ remains an appropriate $g$-function until the infinite branch hits the finite branch, i.e., until the zeros $\mu_1$ and $\mu_2$ merge, which happens at $\xi=\xi_{\merge}=-4B_2+4\sqrt{2}A_2=C_2$ (see Figure~\ref{fig:rarefact}). This indicates the end of the right plane wave sector and that a new $g$-function is required for the asymptotic analysis when $\xi<C_2$. A similar analysis for $\xi\ll 0$ shows that $C_1=-4B_1-4\sqrt{2}A_1$.
%---------------------------------------%
%:fig 4.2
%---------------------------------------%
\begin{figure}[ht]
\centering\includegraphics[scale=.74]{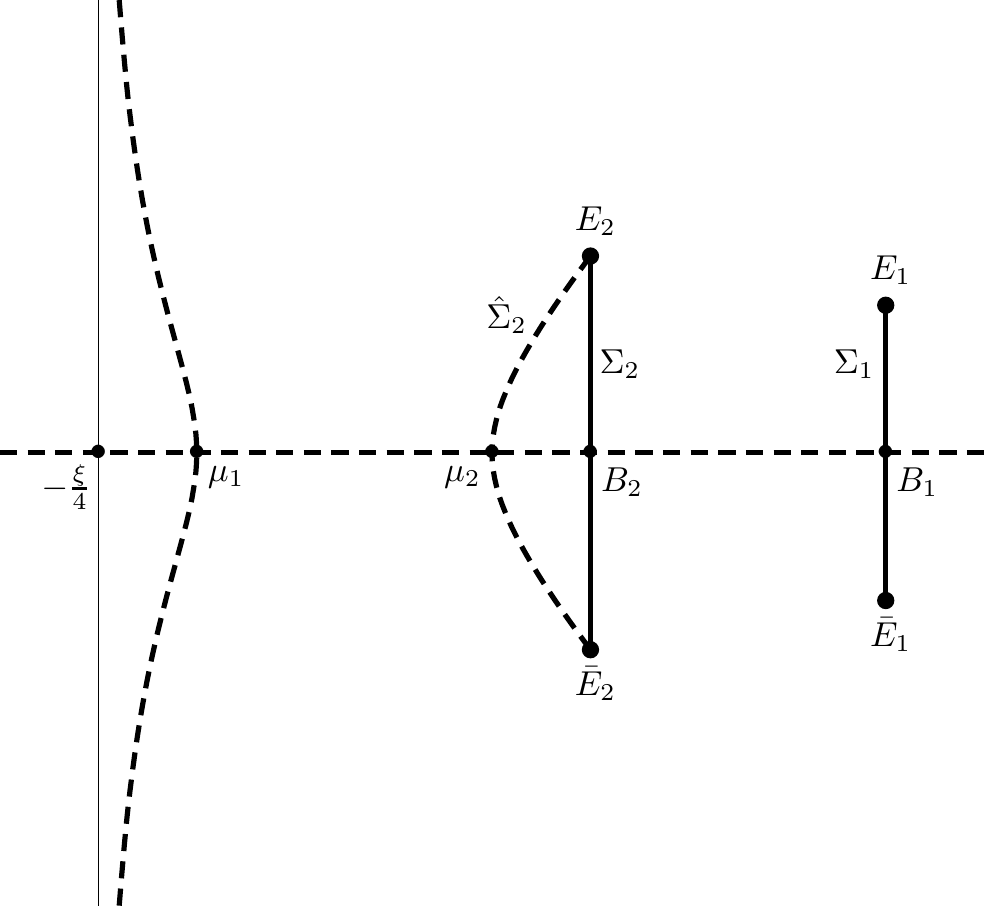}\hspace{4mm}\includegraphics[scale=.74]{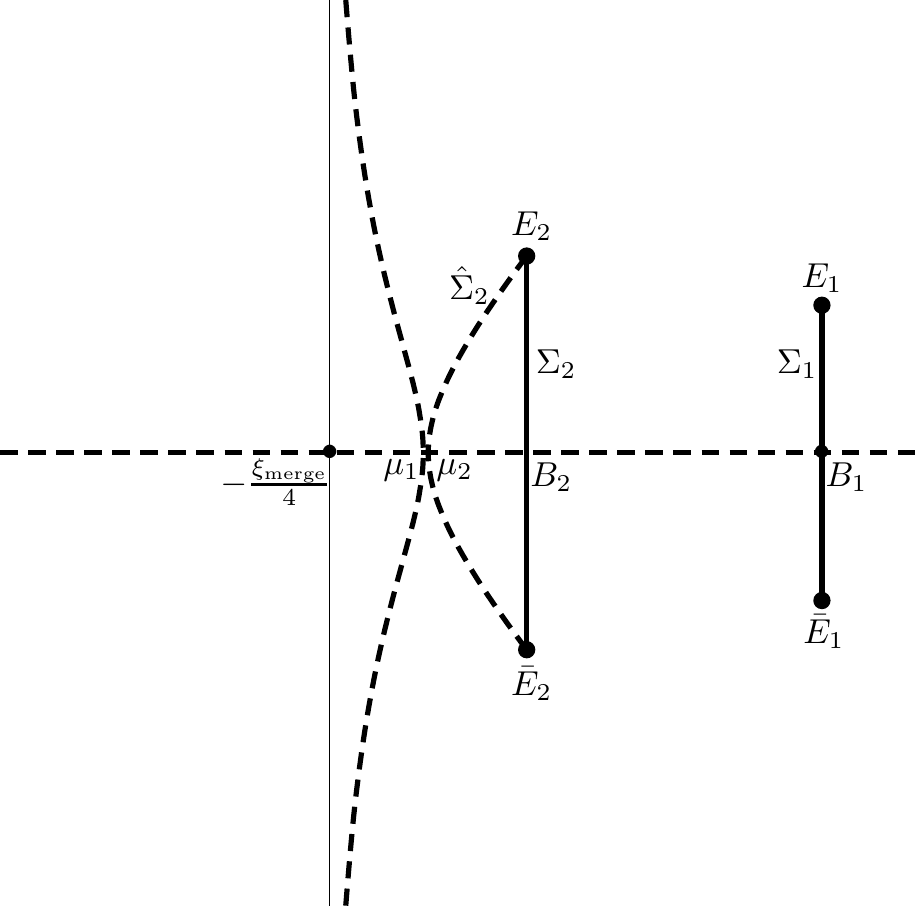}
\caption{Rarefaction: $\xi>\xi_{\merge}$ (left), $\xi=\xi_{\merge}$ (right)} 
\label{fig:rarefact}
\end{figure}
%---------------------------------------%

%---------------------------------------------------------%
%:s.4.2
%---------------------------------------------------------%
\subsection{Elliptic waves: $\BS{-4B_1-4\sqrt{2}A_1<\xi<-4B_1}$ and $\BS{-4B_2<\xi<-4B_2+4\sqrt{2}A_2}$} \label{sec:rarefac-elliptic}

As $\xi$ decreases from $C_2$, a new $g$-function $g\equiv\tilde g_2$ is needed. The transition from the right plane wave sector to the contiguous sector is reflected in the derivative $g'$ by the emergence of two complex conjugate zeros $\beta$ and $\bar\beta$, and the merging of the two real zeros $\mu_1$ and $\mu_2$ into a single real zero $\mu$:
\begin{equation}  \label{rare-g-ellip}
g'(\xi,k)=4\frac{(k-\mu(\xi))(k-\beta(\xi))(k-\bar\beta(\xi))}{\sqrt{(k-E_2)(k-\bar E_2)(k-\beta(\xi))(k-\bar\beta(\xi))}},
\end{equation}
where the parameters $\mu(\xi)$ and $\beta(\xi)$ are subject to the conditions:
\begin{enumerate}[(i)]
\item
Behavior at $k=\infty$:
\begin{equation} \label{g-deriv-as}
g'(\xi,k)=\theta'(\xi,k)+\ord(k^{-2})=4k+\xi+\ord(k^{-2}),\quad k\to\infty.
\end{equation}
\item
Normalization:
\begin{equation}\label{g-norm}
\int_{E_2}^{\bar E_2}\dd g=0.
\end{equation}
\end{enumerate}
The existence of such a $g$-function can be proved using the arguments in \cite{BKS11}*{Section 4.3.1}. This new $g$-function is appropriate for the analysis of the long-time asymptotics in the sector $\xi\in(-4B_2,-4B_2+4\sqrt{2}A_2)$. Further deformations of the RH problem (see \cite{BKS11}*{Section 4.3}) lead to the model RH problem:
\begin{equation}  \label{model-1}
\begin{cases}
m^{\model}\in I+\dot E^2(\D{C}\setminus(\Sigma_1\cup\Sigma_2)),&\\[1mm]
m_+^{\model}(k)=m_-^{\model}(k)\begin{pmatrix}
0&\ii\eul^{\ii D_lx+\ii G_lt+\phi_l}\\\ii\eul^{-\ii D_lx-\ii G_lt-\phi_l}&0
\end{pmatrix},&k\in\Sigma_l,\quad l=1,2.
\end{cases}
\end{equation}
Thus, the leading term of the asymptotics is given in terms of modulated elliptic waves attached to the genus $1$ Riemann surface $w^2=(k-E_2)(k-\bar E_2)(k-\beta(\xi))(k-\bar\beta(\xi))$ (see \cite{BKS11}*{Theorem 3}):
\[
q(x,t)=\hat A_2\frac{\Theta(\beta_2t+\gamma_2)}{\Theta(\beta_2t+\tilde\gamma_2)}\eul^{\ii\nu_2t}+\ord(t^{-1/2}).
\]
A similar analysis applies to the transition from the left plane wave sector to the contiguous sector $-4B_1-4\sqrt{2}A_1<\xi<-4B_1$.
%---------------------------------------------------------%
%:s.4.3
%---------------------------------------------------------%
\subsection{Slow decay: $\BS{-4B_1<\xi<-4B_2}$} \label{sec:rare-decay}

As $\xi\downarrow-4B_2$, the zero $\beta(\xi)$ approaches $E_2$ and $\mu(\xi)$ approaches $-\xi$. As a result, at $\xi=-4B_2$ the derivative of the $g$-function $g\equiv\tilde g_2$ takes the form
\[
g'(\xi,k)=4k+\xi=\theta'(\xi,k).
\]
This is consistent with the fact that for $-4B_1<\xi<-4B_2$, the original phase function $\theta(\xi,k)$ is such that the off-diagonal entries of the jump matrices in the original RH problem \eqref{rhp-basic-jump} across both arcs $\Sigma_1$ and $\Sigma_2$ decay (exponentially) to $0$ as $t\to+\infty$. 
%---------------------------------------%
%:fig 4.3
%---------------------------------------%
\begin{figure}[ht]
\centering\includegraphics[scale=.85]{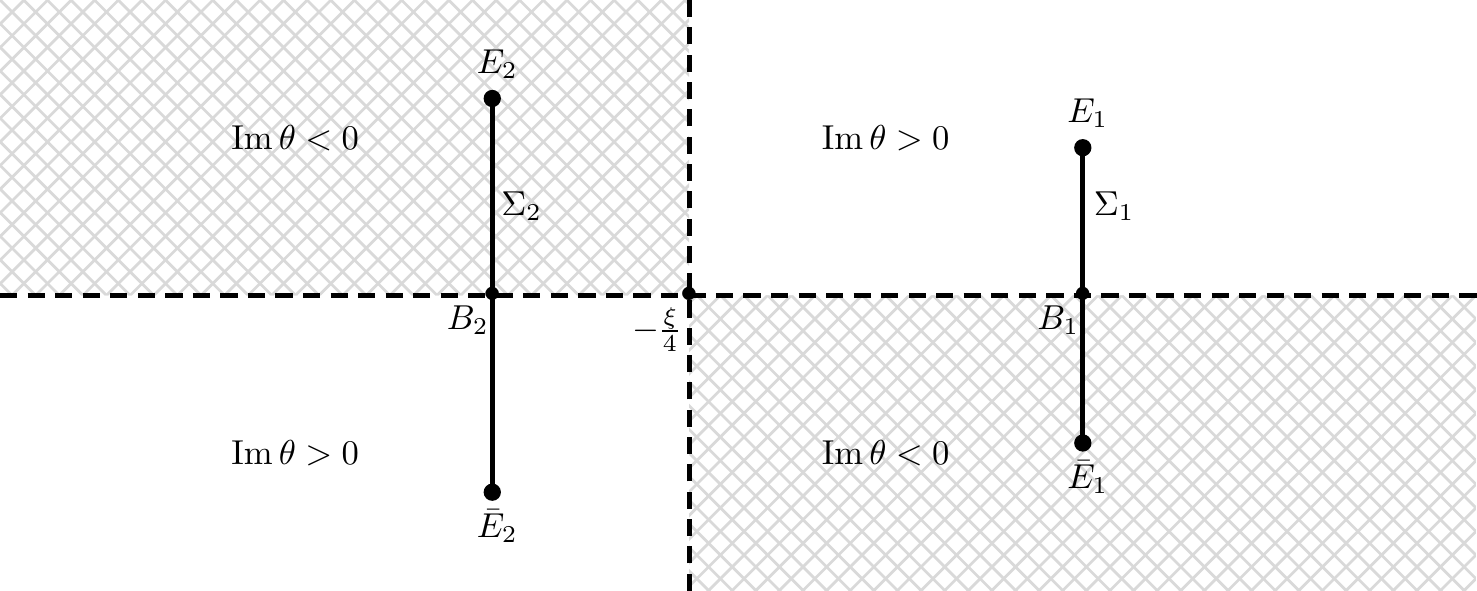}
\caption{Signature table of $\Im\theta(\xi,k)$ for $-4B_1<\xi<-4B_2$} 
\label{fig:theta-c}
\end{figure}
%---------------------------------------%
This suggests keeping $g(\xi,k)=\theta(\xi,k)$ for this range (see Figure~\ref{fig:theta-c}), which implies that the asymptotics for $\xi \in (-4B_1, -4B_2)$ is essentially the same as in the case of zero background, i.e., $q(x,t)=\ord(t^{-1/2})$ and this estimate can be made more precise by detailing the main contribution from the critical point $k=-\xi/4 \in \D{R}$ (see~\cite{DIZ}).
%-------------------%
%:prop 4.1
%-------------------%
\begin{proposition}[slow decay] \label{prop:slowdecay}
For $-4B_1<\xi<-4B_2$, the long time asymptotics of $q(x,t)$ has the form of slow decaying oscillations of Zakharov--Manakov type:
\begin{equation}  \label{slowdecay}
q(x,t)=\frac{c_0(\xi)}{\sqrt{t}}\eul^{\ii(c_1(\xi)t+c_2(\xi)\log t+c_3(\xi))}+\osmall(t^{-1/2}),
\end{equation}
where the coefficients $c_j(\xi)$ are determined in terms of the spectral functions $a(k)$ and $b(k)$ associated with the initial data $q(x,0)$, see \eqref{rarefact-cj}.
\end{proposition}
%-------------------%

%-------------------%
\begin{proof}
The proof is similar to the analogous proof in the case of zero background \cite{DIZ}; it is based on deformations of the original RH problem by ``opening lenses'' (from $-\infty$ to $-\frac{\xi}{4}$ and from $-\frac{\xi}{4}$ to $+\infty$), which leads to a RH problem on a cross centered at $k=-\frac{\xi}{4}$ with jump matrices decaying to the identity matrix uniformly outside any vicinity of $-\frac{\xi}{4}$. A specific feature of the present case of nonzero background is that one also needs to take care of the jumps across $\Sigma_1$ and $\Sigma_2$.
To deal with these jumps, we first introduce the function $d(k)\equiv d(\xi,k)$ which solves the scalar RH problem relative to the contour $\Sigma_d\coloneqq(-\infty,-\frac{\xi}{4})\cup\Sigma_2$ with the jump condition $d_+(k)=d_-(k)J_d(k)$, where
\begin{equation} \label{rarefact-Jd}
J_d=\begin{cases}
1+\abs{r}^2,&k\in(-\infty,-\frac{\xi}{4}),\\
\frac{a_-}{a_+},&k\in\Sigma_2\cap\D{C}^+,\\[.5mm]
\frac{a_+^*}{a_-^*},&k\in\Sigma_2\cap\D{C}^-,
\end{cases}
\end{equation}
and the normalization condition $d(k)\to 1$ as $k\to\infty$. Its solution is given by the Cauchy integral
\[
d(k)=\exp\left\lbrace\frac{1}{2\pi\ii}\int_{\Sigma_d}\frac{\log J_d(s)}{s-k}\dd s\right\rbrace=d_0(k)d_1(k)d_2(k),
\]
with
\begin{alignat*}{2}
d_0(k)&=\exp\left\lbrace\frac{1}{2\pi\ii}\int_{-\infty}^{-\frac{\xi}{4}}\frac{\log(1+\abs{r(s)}^2)}{s-k}\dd s\right\rbrace,&&\\
d_1(k)&=\exp\left\lbrace\frac{1}{2\pi\ii}\int_0^{\ii\infty}\frac{\log\frac{a_-(s)}{a_+(s)}}{s-k}\dd s\right\rbrace,&\quad&
d_2(k)=\exp\left\lbrace\frac{1}{2\pi\ii}\int_{-\ii\infty}^0\frac{\log\frac{a_+^*(s)}{a_-^*(s)}}{s-k}\dd s\right\rbrace.
\end{alignat*}
The behavior of $d_0(k)$ as $k\to-\frac{\xi}{4}$ is the same as in the case of zero background:
\[
d_0(k)=\left(k+\frac{\xi}{4}\right)^{\ii\nu(-\xi/4)}\eul^{\chi(k)},
\]
where
\begin{align*}
\nu(-\xi/4)&=-\frac{1}{2\pi}\log\left(1+\abs{r(-\xi/4)}^2\right)\in\D{R},\\
\chi(k)&=-\frac{1}{2\pi\ii}\int_{-\infty}^{-\xi/4}\log(k-s)\,\dd(1+\abs{r(s)}^2)\in\ii\D{R}.
\end{align*}
On the other hand,
\[
d_1\left(-\frac{\xi}{4}\right)d_2\left(-\frac{\xi}{4}\right)=\exp\left\lbrace\frac{\ii}{\pi}\Im\int_0^{\infty}\frac{\log\frac{a_-(\ii\tau)}{a_+(\ii\tau)}}{\ii\tau+\frac{\xi}{4}}\dd\tau\right\rbrace.
\]
Then, introducing $m^{(1)}(x,t,k)\coloneqq m(x,t,k)d(\xi,k)^{-\sigma_3}$, $k\in\D{C}\setminus\Sigma$ we have
\begin{equation}
m_+^{(1)}(x,t,k)=m_-^{(1)}(x,t,k)\eul^{-\ii t\theta(\xi,k)\sigma_3}J_0^{(1)}(k)\eul^{\ii t\theta(\xi,k)\sigma_3},\quad k\in\Sigma,
\end{equation}
where the jump $J_0^{(1)}=d_-^{\,\sigma_3}J_0d_+^{-\sigma_3}$ has the form of either triangular matrices whose diagonal part is the identity matrix (for $k\in\Sigma_1\cup\Sigma_2$), or products of such matrices (for $k\in\D{R}$):
\begin{equation}  \label{rarefact-J01}
J_0^{(1)}=\begin{cases}
\begin{pmatrix}
1&r^*d^2\\0&1\end{pmatrix}
\begin{pmatrix}
1&0\\rd^{-2}&1\end{pmatrix},&k\in(-\xi/4,+\infty),\\
\begin{pmatrix}
1&0\\\frac{r}{d_+d_-}&1\end{pmatrix}
\begin{pmatrix}
1&r^*d_+d_-\\0&1\end{pmatrix},&k\in(-\infty,-\xi/4),\\
\begin{pmatrix}
1&0\\\frac{\ii\eul^{-\ii\phi_1}}{a_+a_-}d^{-2}&1\end{pmatrix},&k\in\Sigma_1\cap\D{C}^+,\\
\begin{pmatrix}
1&\ii\eul^{\ii\phi_2}d_+d_-\\0&1\end{pmatrix},&k\in\Sigma_2\cap\D{C}^+,\\
\sigma_2J_0^{(1)*}\sigma_2,&k\in(\Sigma_1\cup\Sigma_2)\cap\D{C}^-.
\end{cases}
\end{equation}

The second transformation reduces the jump to the cross $\Sigma_{\cross}=\cup_{j=1}^4L_j$ centered at $k=-\frac{\xi}{4}$, see Figure~\ref{fig:rarefact-cross}. 
%---------------------------------------%
%:fig 4.4
%---------------------------------------%
\begin{figure}[ht]
\centering\includegraphics[scale=.85]{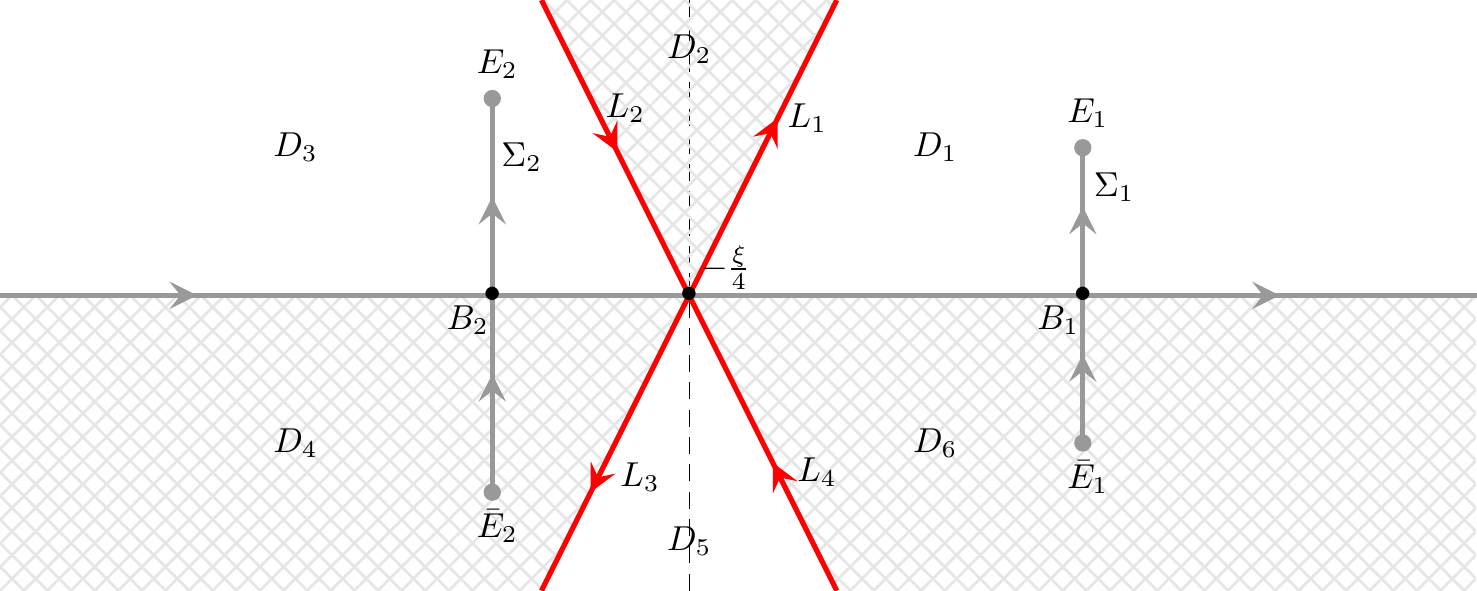}
\caption{Contour deformation for $-4B_1<\xi<-4B_2$} 
\label{fig:rarefact-cross}
\end{figure}
%---------------------------------------%
Introduce
\begin{equation}   \label{rarefact-m2}
m^{(2)}(x,t,k)\coloneqq m^{(1)}(x,t,k)\eul^{-\ii t\theta(\xi,k)\sigma_3}G(k)\eul^{\ii t\theta(\xi,k)\sigma_3},
\end{equation}
where $G\equiv G(k)$ is chosen as follows:
\begin{equation} \label{rarefact-G}
G=\begin{cases}
\begin{pmatrix}
1&0\\-rd^{-2}&1\end{pmatrix},&k\in D_1,\\
\begin{pmatrix}
1&-\tilde ra^2d^2\\0&1\end{pmatrix},&k\in D_3,\\
\begin{pmatrix}
1&0\\\tilde r^*a^{*2}d^{-2}&1\end{pmatrix},&k\in D_4,\\
\begin{pmatrix}
1&r^*d^2\\0&1\end{pmatrix},&k\in D_6,\\
I,&k\in D_2\cup D_5.
\end{cases}
\end{equation}
Recall that $\tilde r(k)\coloneqq\frac{b(k)}{a(k)}$. Then
\begin{equation}
m_+^{(2)}(x,t,k)=m_-^{(2)}(x,t,k)\eul^{-\ii t\theta(\xi,k)\sigma_3}J_0^{(2)}(k)\eul^{\ii t\theta(\xi,k)\sigma_3},\quad k\in\Sigma\cup\Sigma_{\cross},
\end{equation}
where $J_0^{(2)}=G_-^{-1}J_0^{(1)}G_+$ is as follows:
\begin{enumerate}[(1)]
\item
For $k\in\D{R}$, $J_0^{(2)}=I$ by the very construction of $G$.
\item
For $k\in\Sigma_1\cup\Sigma_2$, we also have $J_0^{(2)}=I$. Indeed, it follows from \eqref{rarefact-J01} and \eqref{rarefact-G} that for $k\in\Sigma_1\cap\D{C}^+$, $J_0^{(2)}=\left(\begin{smallmatrix}1&0\\\#&1\end{smallmatrix}\right)$ with
\[
\#\coloneqq d^{-2}\left(r_--r_++\frac{\ii\eul^{-\ii\phi_1}}{a_+a_-}\right)=0,
\]
in view of \eqref{r+r-sigma1}. Similarly, it follows from \eqref{rarefact-Jd} and \eqref{tr+tr-sigma2} that for $k\in\Sigma_2\cap\D{C}^+$, $J_0^{(2)}=\left(\begin{smallmatrix}1&\sharp\\0&1\end{smallmatrix}\right)$ with
\[
\sharp\coloneqq\tilde r_-a_-^2d_-^2-\tilde r_+a_+^2d_+^2+\ii\eul^{\ii\phi_2}d_+d_-=a_+a_-d_+d_-\left(\tilde r_--\tilde r_++\frac{\ii\eul^{\ii\phi_2}}{a_+a_-}\right)=0.
\]
Then, by symmetry, $J_0^{(2)}=I$ also for $k\in(\Sigma_1\cup\Sigma_2)\cap\D{C}^-$.
\item
For $k\in\Sigma_{\cross}$, 
\[
J_0^{(2)}=\begin{cases}
\begin{pmatrix}
1&0\\rd^{-2}&1\end{pmatrix},&k\in L_1,\\
\begin{pmatrix}
1&\tilde ra^2d^2\\0&1\end{pmatrix},&k\in L_2,\\
\begin{pmatrix}
1&0\\-\tilde r^*a^{*2}d^{-2}&1\end{pmatrix},&k\in L_3,\\
\begin{pmatrix}
1&-r^*d^2\\0&1\end{pmatrix},&k\in L_4.
\end{cases}
\]
\end{enumerate}
The RH problem for $m^{(2)}$ is the same as in the case of zero background (see \cite{DIZ}), the only difference being an additional factor (depending on $\xi$ only) in the approximation
\[
d(k)\sim\left(k+\frac{\xi}{4}\right)^{\ii\nu(-\xi/4)}\eul^{\tilde\chi(-\xi/4)},\quad k\to-\frac{\xi}{4},
\]
where
\[
\tilde\chi(-\xi/4)=\chi(-\xi/4)+\frac{\ii}{\pi}\Im\int_0^{+\infty}\frac{\log\frac{a_-(\ii\tau)}{a_+(\ii\tau)}}{\ii\tau+\frac{\xi}{4}}\dd\tau.
\]
It follows that the asymptotics of $q(x,t)$ has the form \eqref{slowdecay} with $c_0$, $c_1$, $c_2$, and $c_3$ given by
\begin{equation}\label{rarefact-cj}
\begin{split}
c_0(\xi)&=\left(\frac{1}{4\pi}\log(1+\abs{r(-\xi/4)}^2)\right)^{1/2},\\
c_1(\xi)&=\frac{\xi^2}{4},\\
c_2(\xi)&=-\nu(-\xi/4),\\
c_3(\xi)&=-3\log 2\,\nu(-\xi/4)+\frac{\pi}{4}+\arg\Gamma(\ii\nu(-\xi/4))-\arg r(-\xi/4)-2\ii\tilde\chi(-\xi/4).
\end{split}\qedhere
\end{equation}
\end{proof}
%-------------------%

%---------------------------------------------------------%
%:s.4.4
%---------------------------------------------------------%
\subsection{Summary} \label{sec:rarefaction-summary}

In the rarefaction case there is only \emph{one} asymptotic scenario.

%-------------------%
%:thm 4.2
%-------------------%
\begin{theorem}[rarefaction] \label{thm:rarefaction}
Suppose $B_1>B_2$. The long-time asymptotics is then as follows.
\begin{enumerate}[\rm(i)]
\item
\emph{Plane wave region:} $\xi<-4B_1+4\sqrt2\,A_1$ ($j=1$) and $\xi>-4B_2+4\sqrt2\,A_2$ ($j=2$). The leading term is a plane wave of constant amplitude:
\[
q(x,t)=A_j\eul^{-2\ii B_jx+2\ii\omega_jt+\ii\psi_j(\xi)}+\ord(t^{-\frac{1}{2}}),\quad j=1,2.
\]
\item
\emph{Elliptic wave region:} $-4B_1-4\sqrt2\,A_1<\xi<-4B_1$ ($j=1$) and $-4B_2<\xi<-4B_2+4\sqrt2\,A_2$ ($j=2$). The leading term is a modulated elliptic wave:
\[
q(x,t)=\hat A_j\frac{\Theta(\beta_jt+\gamma_j)}{\Theta(\beta_jt+\tilde\gamma_j)}\eul^{\ii\nu_jt}+\ord(t^{-1/2}),
\]
where all coefficients depend on $\xi$. Moreover, $\Theta(z)\coloneqq\sum_{m\in\D{Z}}\eul^{2\ii\pi(\frac{1}{2}\tau m^2+mz)}$ is the Jacobi theta function with modular invariant $\tau\equiv\tau(\xi)$ and $\hat A_j$ is of the order of $A_j$.
\item
\emph{Slow decay region:} $-4B_1<\xi<-4B_2$. The leading term is a modulated plane wave whose amplitude is slowly decaying:
\[
q(x,t)=\frac{c_0(\xi)}{\sqrt{t}}\,\eul^{\ii(c_1(\xi)t+c_2(\xi)\log t+c_3(\xi))}+\osmall(t^{-1/2})=\ord(t^{-1/2}).
\]
\end{enumerate}
\end{theorem}
%-------------------%

%---------------------------------------------------------%
%:s.5
%---------------------------------------------------------%
\section{Asymptotics: the shock case}   \label{sec:shock}

The shock case $B_1<B_2$ turns out to be much richer than the rarefaction case. There are several asymptotic scenarios depending on the values of $A_1/(B_2-B_1)$ and $A_2/(B_2-B_1)$, see Section~\ref{sec:reduction}. 

%-------------------%
\begin{assumption*}
Henceforth, for simplicity, we assume we are in the \emph{symmetric shock case}, i.e.,
\begin{equation}  \label{symmetriccase}
A_1=A_2=A>0\quad\text{and}\quad B_2=-B_1=B>0.
\end{equation}
Asymptotic scenarios then depend only on the ratio $A/B$.
\end{assumption*}
%-------------------%

%---------------------------------------------------------%
%:s.5.1
%---------------------------------------------------------%
\subsection{Plane waves: $\BS{\abs{\xi}\gg 0}$} \label{sec:shock-plane}

As already seen in Section~\ref{sec:plane}, and as in the rarefaction case, appropriate $g$-functions for $\abs{\xi}\gg 0$ are still $g_j$, $j=1,2$ given by \eqref{gj}, and the asymptotics are plane waves of type \eqref{q-as-plane}. The asymptotics is characterized by two properties:
\begin{enumerate}[(i)] 
\item
the infinite and finite branches of $\Im g_j(\xi,k)=0$ defined by \eqref{gj} cross the real axis at two distinct points, $\mu_1(\xi)$ and $\mu_2(\xi)$, respectively;
\item
the points $E_1$ and $E_2$ are located on the same side from the  infinite branch of $\Im g_j(\xi,k)=0$. 
\end{enumerate}
As $\abs{\xi}$ decreases, the end of the plane wave asymptotic region is associated with the violation of either (i) or (ii).

For large positive values of $\xi$, let $g\equiv g_2(\xi,k)$ be the plane wave $g$-function given by \eqref{gj}. The two real zeros $\mu_j\equiv\mu_j(\xi)$, $j=1,2$ of $\Im g(\xi,k)$ are given by (see \eqref{mu12})
\begin{equation}   \label{mu12bis}
\mu_1=\frac{B}{2}-\frac{\xi}{8}-\frac{1}{8}\sqrt{(\xi+4B)^2-32A^2},\quad\mu_2=\frac{B}{2}-\frac{\xi}{8}+\frac{1}{8}\sqrt{(\xi+4B)^2-32A^2},
\end{equation}
and
\begin{equation}   \label{g2bis}
g'(\xi,k)=4\frac{(k-\mu_1(\xi))(k-\mu_2(\xi))}{\sqrt{(k-E_2)(k-\bar E_2)}}.
\end{equation}
%---------------------------------------%
%:fig 5.1
%---------------------------------------%
\begin{figure}[ht]
\centering\includegraphics[scale=.7]{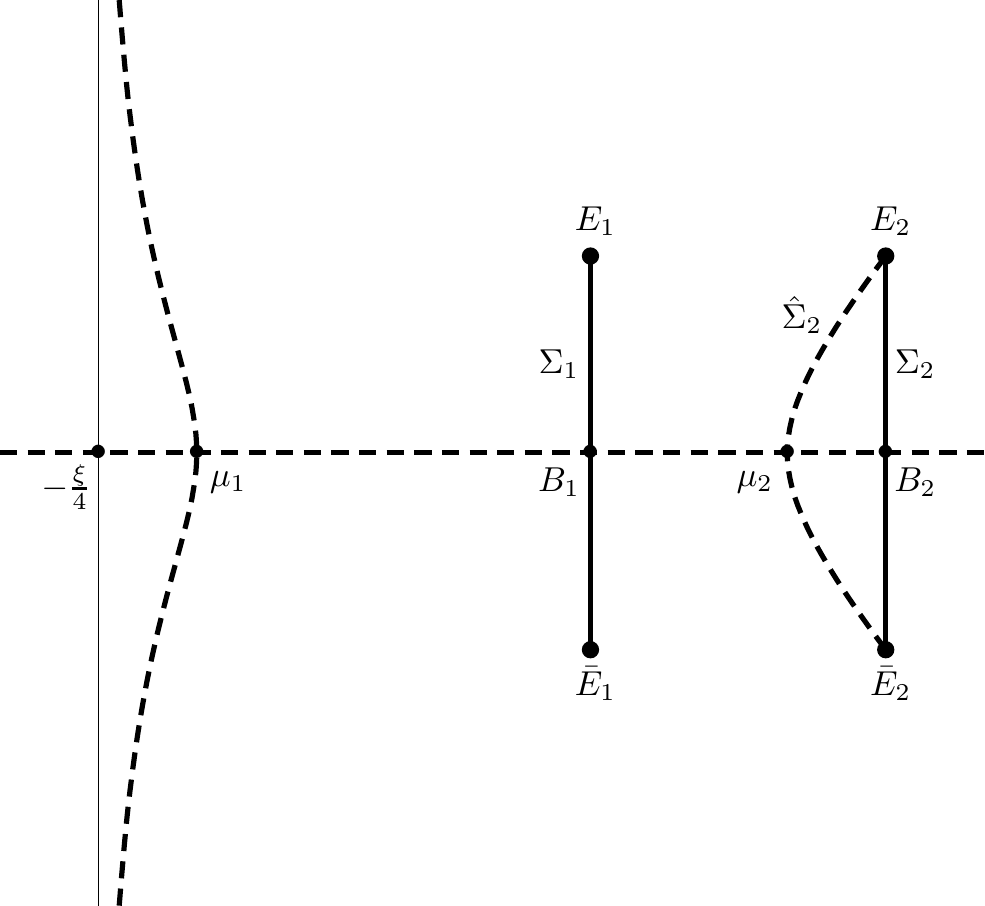}\hspace{8mm}\includegraphics[scale=.7]{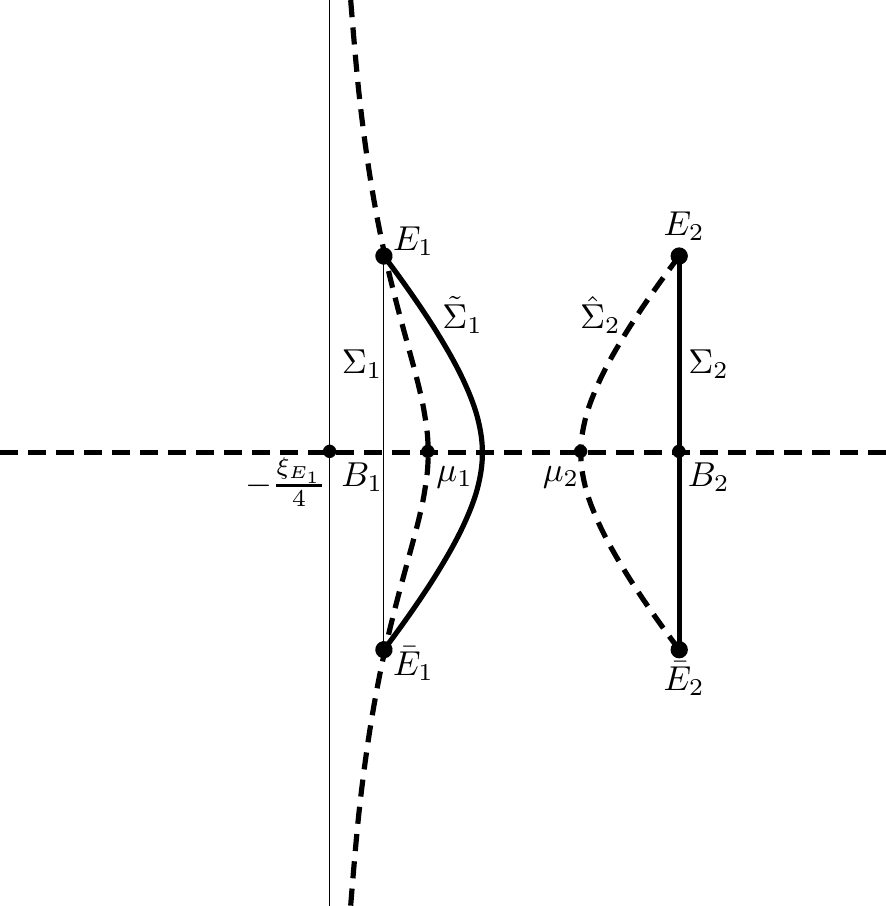}
\caption{Shock, case 1, $\xi_{E_1}>\xi_{\merge}$: $\xi>\xi_{E_1}$ (left), $\xi=\xi_{E_1}$ (right)} 
\label{fig:shock-1}
\end{figure}
%---------------------------------------%
As $\xi$ decreases, the infinite branch of the curve $\Im g=0$ moves to the right. In contrast with the rarefaction case where there was only one possibility, there are now three possibilities (see Figures~\ref{fig:shock-1} and \ref{fig:shock-2}):
\begin{enumerate}[{Case} 1.]
\item 
The infinite branch hits $E_1$ and $\bar E_1$ before the two real zeros $\mu_1$ and $\mu_2$ merge.
\item
The two real zeros $\mu_1$ and $\mu_2$ merge before the infinite branch hits $E_1$ and $\bar E_1$.
\item
The infinite branch hits $E_1$ and $\bar E_1$ at the same time as the two real zeros $\mu_1$ and $\mu_2$ merge.
\end{enumerate}

%-------------------%
\begin{remark*}
To clearly see that the events listed in Cases 1 to 3 are the only events that signify the ending of the plane wave sector, it is better to first deform the part $\Sigma_1$ of the contour of the RH problem which connects $E_1$ with $\bar E_1$ into an arc $\tilde\Sigma_1=(E_1,\bar E_1)$ which is located to the right of the infinite branch of $\Im g=0$. Under assumption \eqref{qnot}, this deformation can be made in a particularly simple way, replacing the branch cut $\Sigma_1$ by $\tilde\Sigma_1$ in the definitions of $X_1(k)$, $\Omega_1(k)$, and $N_1(k)$, see \eqref{x-om}--\eqref{Phi0-N}. Then in all cases, the jump matrix on $\tilde\Sigma_1$ decays to the identity matrix as $t\to+\infty$ and thus does not contribute to the main asymptotic term. Consequently, the ending of the plane wave sector related to the interaction of the infinite branch with the jump contour connecting the branch points $E_1$ and $\bar E_1$ is as described in Cases 1 and 3 (but not with the moment when the infinite branch touches the line segment $\Sigma_1$).
\end{remark*}
%-------------------%

The infinite branch hits $E_1$ and $\bar E_1$ for $\xi=\xi_{E_1}$ where
\begin{equation}\label{xi-E-1}
\xi_{E_1}=2(B+\sqrt{A^2+B^2}).
\end{equation}
On the other hand, the two real zeros of $g'$ merge for $\xi=\xi_{\merge}$ where
\begin{equation}\label{xi-merge}
\xi_{\merge}=4(-B+\sqrt{2}A).
\end{equation}
Hence the infinite branch of $\Im g=0$ hits $E_1$ and $\bar E_1$ before the zeros merge if $\xi_{E_1}>\xi_{\merge}$, i.e., if
\[
\frac{A}{B}<\frac{2}{7}(2+3\sqrt{2})\approx 1.7836.
\]
Thus: 
\begin{itemize}
\item
Case 1 occurs if $\frac{A}{B}<\frac{2}{7}(2+3\sqrt{2})$, 
\item
Case 2 occurs if $\frac{A}{B}>\frac{2}{7}(2+3\sqrt{2})$, 
\item
Case 3 occurs for $\frac{A}{B}=\frac{2}{7}(2+3\sqrt{2})$.
\end{itemize} 
%---------------------------------------%
%:fig 5.2
%---------------------------------------%
\begin{figure}[ht]
\centering\includegraphics[scale=.7]{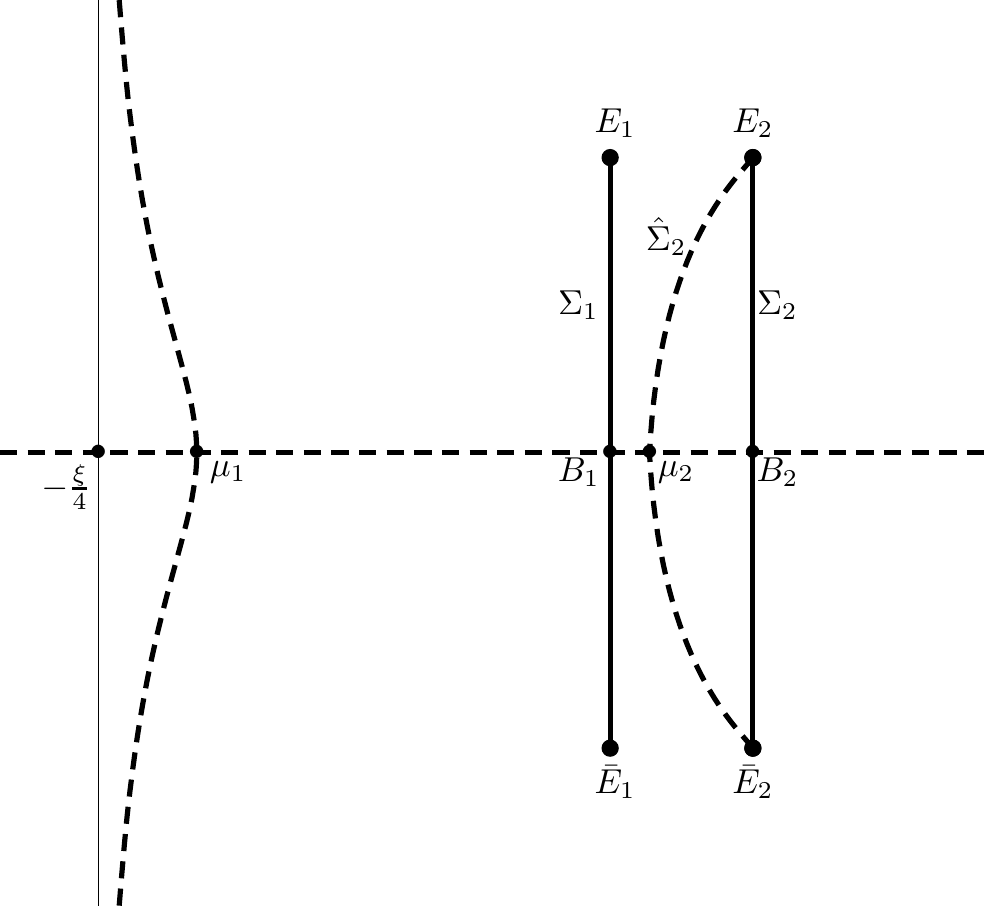}\hspace{8mm}\includegraphics[scale=.7]{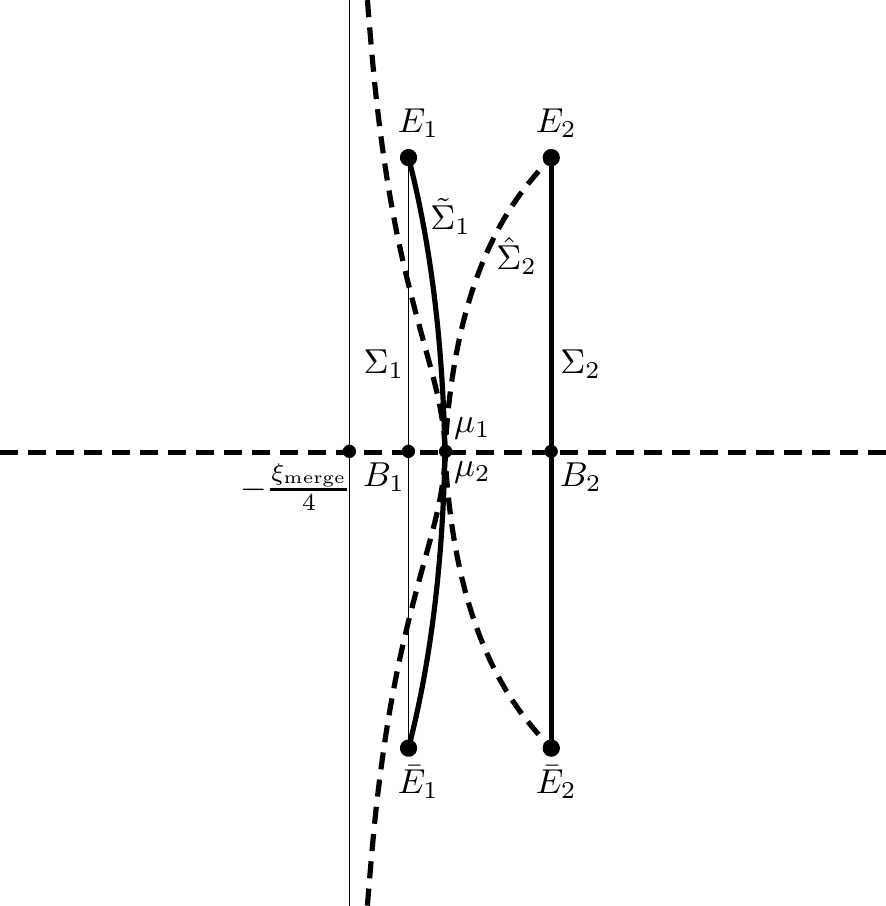}
\caption{Shock, case 2, $\xi_{\merge}>\xi_{E_1}$: $\xi>\xi_{\merge}$ (left), $\xi=\xi_{\merge}$ (right)} 
\label{fig:shock-2}
\end{figure}
%---------------------------------------%

Each of these cases signifies the ending of the plane wave sector, because the $g$-function $g_2(\xi,k)$ from \eqref{gj} stops to provide a signature table appropriate for subsequent deformations (see, e.g., \cites{BKS11,BV07} for details) and thus a more complicated $g$-function is required. In particular, Case 1 was addressed in \cite{BV07}*{Section 4}, where a genus~$2$ region adjacent to the plane wave region was specified. In \cite{BV07}, this region was characterized as the values of $\xi$ for which a system of nonlinear equations \cite{BV07}*{Eqs.~(4.12)--(4.15)} is solvable, giving the parameters of the asymptotics in this region. The solvability issue for this system was not addressed in \cite{BV07}. The value of $\xi$ separating the plane wave sector from the genus~$2$ sector was given, in our notation, as the value for which $\mu_1(\xi)=-B$, i.e., the value at which the infinite branch of $\Im g_2=0$ touches the vertical segment $(E_1,\bar E_1)$. This value, which in our notation is $4B+\frac{A^2}{B}$, is strictly greater than the correct value $\xi_{E_1}$ given by \eqref{xi-E-1}. Also notice that the other two possibilities were not considered in \cite{BV07}. One can show that in Case 2, the asymptotics in the adjacent sector is given in terms of a genus~$1$ elliptic wave (here the transition is similar to that occurring in the rarefaction case, see \cite{BKS11}), whereas in Case 3, the asymptotics in the adjacent sector is given in terms of a genus~$3$ hyperelliptic wave.

A similar analysis applies to the left plane wave sector.

%---------------------------------------------------------%
%:s.5.2
%---------------------------------------------------------%
\subsection{Asymptotics for small $\BS{\abs{\xi}}$} \label{sec:smallxi}

We next analyze the possible asymptotic scenarios in the ``middle'' domain. The distribution of the asymptotic sectors is expected to be symmetric under $\xi\mapsto-\xi$, and thus special attention will be paid to the case $\xi=0$, i.e., to the asymptotics along the $t$-axis.
%---------------------------------------%
%:fig 5.3
%---------------------------------------%
\begin{figure}[ht]
\centering\includegraphics[scale=.73]{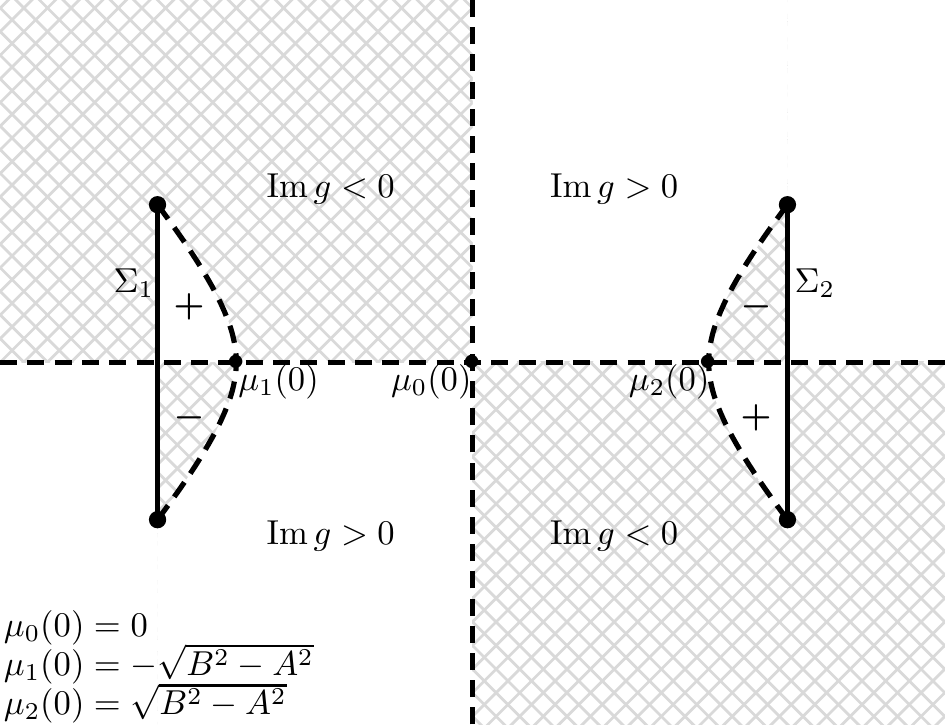}\hspace{4mm}\includegraphics[scale=.73]{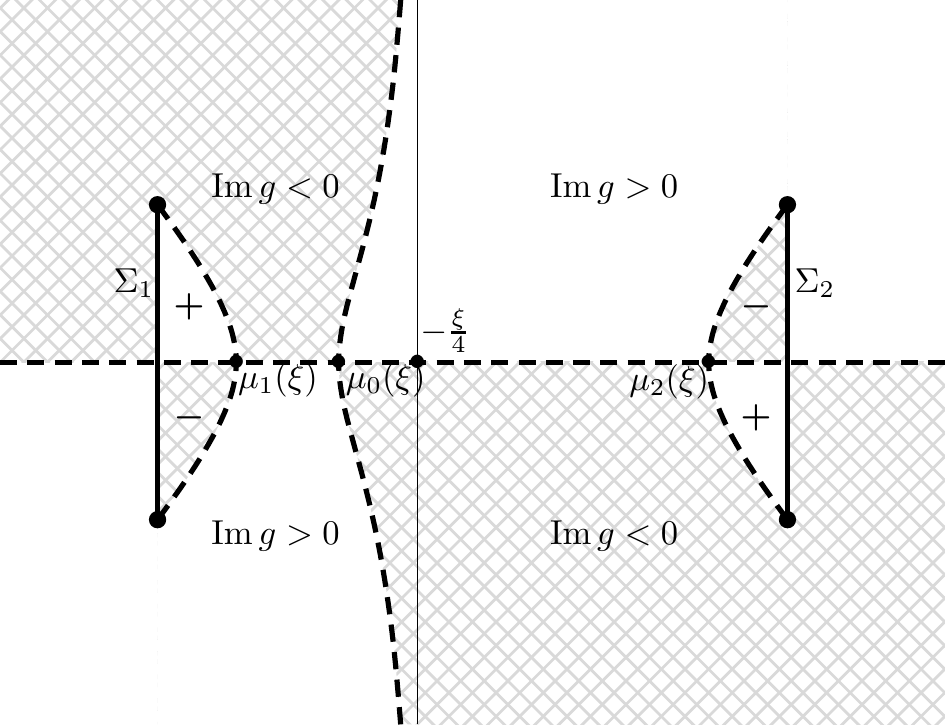}
\caption{Signature table of $\Im g(\xi,k)$ in case $A<B$: $\xi=0$ (left), $0<\xi<\xi_0$ (right)}
\label{fig:signature-ab}
\end{figure}
%---------------------------------------%

%---------------------------------------------------------%
%:s.5.2.1
%---------------------------------------------------------%
\subsubsection{Case $\frac{A}{B}<1$ and $\abs{\xi}<\xi_0$} \label{sec:smallxi-a}

In \cite{BV07}*{Section 3}, the asymptotics in the sector $\accol{\xi:\abs{\xi}<\xi_0}$ (for some $\xi_0>0$) was actually discussed \emph{under the assumption} that the signature table of $\Im g$ for the associated $g$-function was as in \cite{BV07}*{Figure 3.3 (a)}; see also Figure~\ref{fig:signature-ab} (right). In terms of the derivative $g'=\dd g/\dd k$ of the associated $g$-function, it means that $g'(\xi,k)$ has the form
\begin{equation} \label{g-small-bv}
g'(\xi,k)=4\frac{(k-\mu_1(\xi))(k-\mu_0(\xi))(k-\mu_2(\xi))}{\sqrt{(k-E_1)(k-\bar E_1)(k-E_2)(k-\bar E_2)}},
\end{equation}
where the branch cuts for $g'$ are $\Sigma_1$ and $\Sigma_2$ and $\mu_1(\xi)<\mu_0(\xi)<\mu_2(\xi)$ are all real: they are the self-intersection points of the curve $\Im g(\xi,k)=0$. In \cite{BV07}*{Formula (3.27)} the associated $g$-function is of the form $f+2G$, with $f(\xi,k)=g_2(\xi/2,k)$, and $G=\ord(1)$ as $k\to\infty$.

Let us check the validity of this assumption considering $\xi=0$. In this case, the symmetry implies that $\mu_0(0)=0$ whereas $\mu_2(0)=-\mu_1(0)>0$, and the signature table has the form indicated in Figure~\ref{fig:signature-ab} (left). Then, as $k\to\infty$, from \eqref{g-small-bv} we have
\begin{equation}  \label{g-0-bv-as}
g'(0,k)=4k\left(1+\frac{1}{k^2}[-\mu_2^2(0)+B^2-A^2]+\ord(k^{-3})\right).
\end{equation}
Comparing this with 
\begin{equation}\label{g-ass}
g'(\xi,k)=4k+\xi+\ord(k^{-2}),
\end{equation}
which follows from \eqref{gj-as} (we indeed have $g=g_2+\ord(1)$ as $k\to\infty$), we obtain that
\[
\mu_2^2(0)=B^2-A^2,
\]
which can only be valid in the case $A<B$ (recall that $\mu_2(0)$ is real and nonzero).

The signature table for small enough $\abs{\xi}$ has a similar structure, see Figure~\ref{fig:signature-ab} (right), and, as it was shown in \cite{BV07}*{Section 3}, a $g$-function with derivative of the form \eqref{g-small-bv} is indeed suitable for the asymptotic analysis in the sector $\abs{\xi}<\xi_0$, leading to genus~$1$ asymptotics in this sector.

On the other hand, in the case $A\geq B$, the situation is different.
%---------------------------------------------------------%
%:s.5.2.2
%---------------------------------------------------------%
\subsubsection{Case $\frac{A}{B}\geq 1$ and $\xi=0$} \label{sec:smallxi-b}
%-------------------%
%:prop 5.1
%-------------------%
\begin{proposition}  \label{prop-0}
Assume that \eqref{symmetriccase} holds with $\frac{A}{B}\geq 1$. Then, for $\xi=0$ an appropriate $g$-function has a derivative of the form
\begin{equation}  \label{g-0}
g'(0,k)=4\frac{k(k^2+\alpha_0^2)}{\sqrt{(k-E_1)(k-\bar E_1)(k-E_2)(k-\bar E_2)}},
\end{equation}
where $\alpha_0=\sqrt{A^2-B^2}$, generating genus~$1$ asymptotics for $x=0$.
\end{proposition}
%-------------------%

%-------------------%
\begin{comment*}
The proof of Proposition \ref{prop-0} consists in performing the asymptotic analysis for $\xi=0$ using the $g$-function \eqref{g-0} and showing that it leads to genus~$1$ asymptotics, expressed in terms of elliptic functions attached to the Riemann surface $w^2=(k-E_1)(k-\bar E_1)(k-E_2)(k-\bar E_2)$. Details will be given elsewhere.
\end{comment*}
%-------------------%

%---------------------------------------%
%:fig 5.4
%---------------------------------------%
\begin{figure}[ht]
\centering\includegraphics[scale=.73]{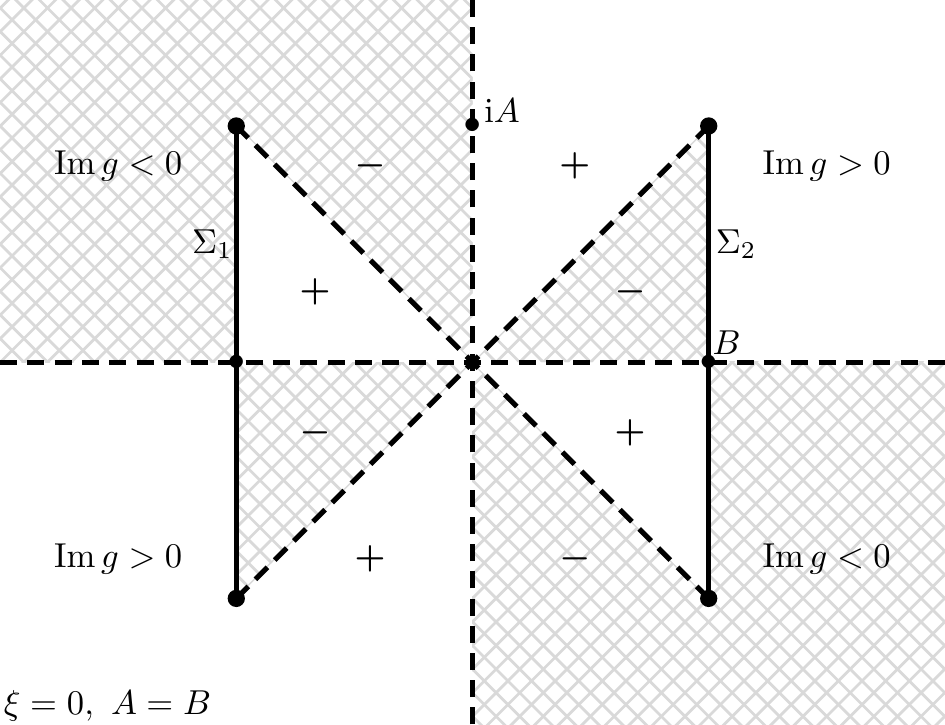}\hspace{4mm}\includegraphics[scale=.73]{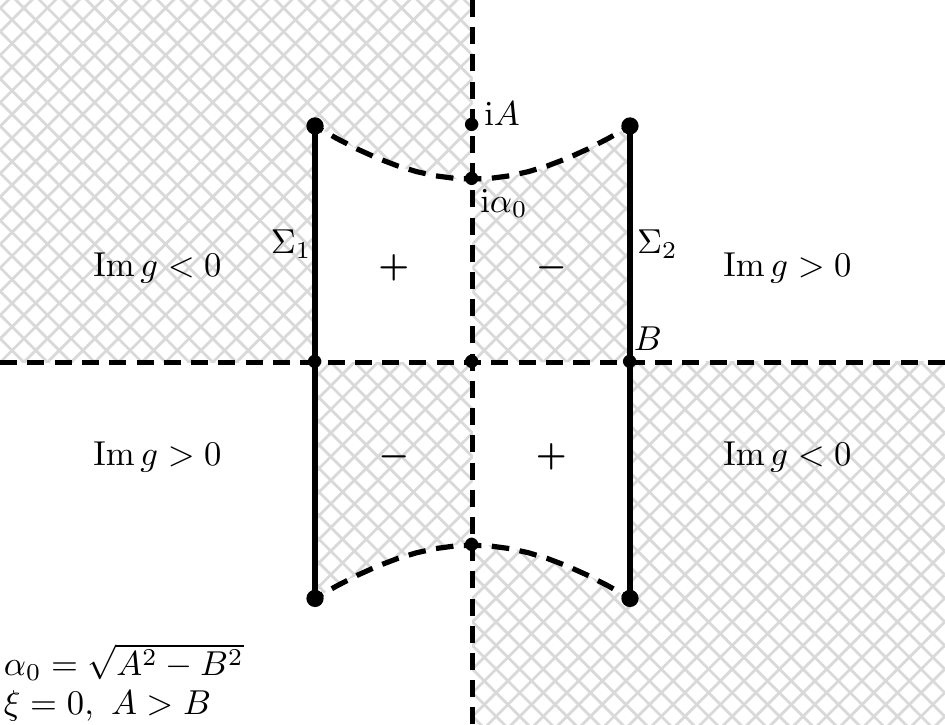}
\caption{Signature table of $\Im g(0,k)$ in cases $A=B$ (left) and $A>B$ (right)} 
\label{fig:signature-cd}
\end{figure}
%---------------------------------------%

The signature tables are shown in Figure~\ref{fig:signature-cd} in the cases $\frac{A}{B}=1$ (left) and $\frac{A}{B}>1$ (right).

%---------------------------------------%
%:fig 5.5
%---------------------------------------%
\begin{figure}[ht]
\centering\includegraphics[scale=.73]{nls-signature-dd}\hspace{4mm}\includegraphics[scale=.73]{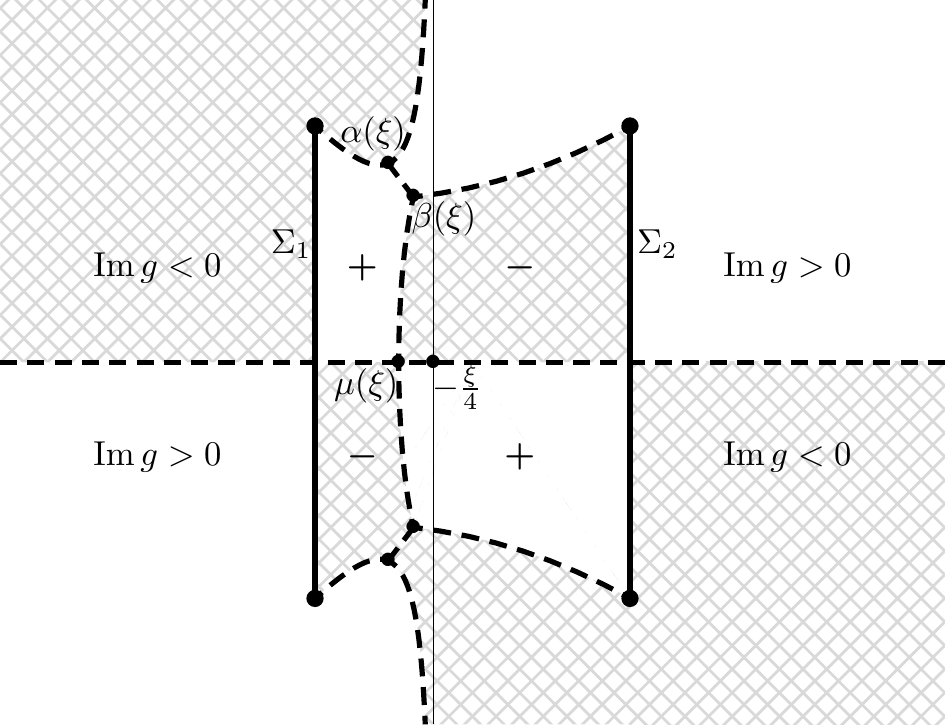}
\caption{Signature table of $\Im g(\xi,k)$ in case $A>B$: $\xi=0$ (left), $0<\xi<\xi_0$ (right)} 
\label{fig:signature-de}
\end{figure}
%---------------------------------------%

%---------------------------------------------------------%
%:s.5.2.3
%---------------------------------------------------------%
\subsubsection{Case $\frac{A}{B}>1$ and $0<\varepsilon<\abs{\xi}<\xi_0$} \label{sec:smallxi-c}

The form of the derivative of a $g$-function given by \eqref{g-0} is unstable with respect to $\xi$. In particular, in the case $A>B$ we have the following.

%-------------------%
%:prop 5.2
%-------------------%
\begin{proposition}\label{gen3}
Assume that \eqref{symmetriccase} holds with $\frac{A}{B}>1$. Then, for all $\xi$ with $\varepsilon<\abs{\xi}<\xi_0$, for some $\xi_0>0$ and any $\varepsilon\in(0,\xi_0)$, an appropriate $g$-function has a derivative of the following form, generating genus~$3$ asymptotics (see Figure~\ref{fig:signature-de} (right)):
\begin{equation}  \label{g-xi}
g'(\xi,k)=4\frac{(k-\mu(\xi))(k-\alpha(\xi))(k-\bar\alpha(\xi))(k-\beta(\xi))(k-\bar\beta(\xi))}{w(\xi,k)},
\end{equation}
where $w^2=(k-E_1)(k-\bar E_1)(k-E_2)(k-\bar E_2)(k-\alpha(\xi))(k-\bar\alpha(\xi))(k-\beta(\xi))(k-\bar\beta(\xi))$. As $\xi\to 0$, $\alpha(\xi)$ and $\beta(\xi)$ approach $\alpha(0)=\beta(0)=\ii\alpha_0\equiv\ii\sqrt{A^2-B^2}$. Here the branch cuts for $g'$ are $\Sigma_1$, $\Sigma_2$, $(\alpha,\beta)$, and $(\bar\alpha,\bar\beta)$.
\end{proposition}
%-------------------%

Thus, the long-time asymptotics of $q$ is given in terms of hyperelliptic functions attached to the genus $3$ Riemann surface $M\equiv M(\xi)$ defined by $w^2=(k-E_1)(k-\bar E_1)(k-E_2)(k-\bar E_2)(k-\alpha(\xi))(k-\bar\alpha(\xi))(k-\beta(\xi))(k-\bar\beta(\xi))$.

%-------------------%
\begin{comment*}
The proof of Proposition \ref{gen3} relies on the solvability of a system of equations which characterize genus~$3$ asymptotics (see \cite{BLS20b}):
\begin{subequations}  \label{dg-conditions}
\begin{align}  \label{dg-first-conditions}
&\int_{a_1}\widehat{\dd g}=\int_{a_2}\widehat{\dd g}=\int_{a_3}\widehat{\dd g}=0,\\
\label{dg-last-conditions}
&\lim_{k\to\infty}\left(\frac{\dd g}{\dd k}-4k\right)=\xi,\quad       
\lim_{k\to\infty}k\left(\frac{\dd g}{\dd k}-4k-\xi\right)=0,
\end{align}
\end{subequations}
where $\widehat{\dd g}$ denotes the differential on $M$ given by $\dd g$
on the upper sheet and $-\dd g$ on the lower sheet, and $a_1$, $a_2$, $a_3$ are certain paths on $M$. The definition \eqref{g-xi} of $g'$ depends on five real parameters $\alpha_1=\Re\alpha$, $\alpha_2=\Im\alpha$, $\beta_1=\Re\beta$, $\beta_2=\Im\beta$, and $\mu$, and \eqref{dg-conditions} is actually a system of five equations. The proof of solvability reduces to the application of the implicit function theorem for the vector function $\BS{g}(\xi)=\accol{\alpha_1(\xi),\alpha_2(\xi),\beta_1(\xi),\beta_2(\xi),\mu(\xi)}$. Details are given in \cite{BLS20b}.
\end{comment*}
%-------------------%

Proposition \ref{gen3} justifies the importance of studying the genus~$3$ sector as well as the merging of $\alpha(\xi)$ and $\beta(\xi)$ characterizing a transition zone (smaller than any sector $\abs{\xi}<\varepsilon$ for any $\varepsilon>0$) connecting the axis $\xi=0$, where the asymptotics is genus~$1$, to the genus~$3$ sector $\varepsilon<\xi<\xi_0$ (similarly for the negative values of $\xi$). Details are given in \cite{BLS20c}.

%---------------------------------------------------------%
%:s.5.3
%---------------------------------------------------------%
\subsection{Overview of scenarios in the symmetric shock case} \label{sec:overview}

In this subsection, we describe the five possible asymptotic scenarios that may arise in the symmetric shock case. The first three scenarios correspond to Case 1, the fourth to Case 3, and the fifth to Case 2. There are two ``bifurcation values'' of $\frac{A}{B}$: the first, $\frac{A}{B}=\frac{2}{7}(2+3\sqrt{2})$, determines the three cases $1$, $2$, and $3$, the second, $\frac{A}{B}=1$, determines the three subcases of Case 1. By symmetry, it is enough to consider $\xi\geq 0$.

%---------------------------------------------------------%
%:s.5.3.1
%---------------------------------------------------------%
\subsubsection{1st Scenario}  \label{sec:scenario-1}
This is the scenario developed by Buckingham and Venakides in \cite{BV07}.

%-------------------%
%:table 5.1
%-------------------%
\begin{table}[ht]
\begin{tabular}{|c|c|c|c|c|}
\hline
$0\leq\xi<\xi_{\alpha}$&$\xi=\xi_{\alpha}$&$\xi_{\alpha}<\xi<\xi_{E_1}$&$\xi=\xi_{E_1}$&$\xi>\xi_{E_1}$\\
\hline
genus $1$&&genus $2$&&genus $0$\\ 
\hline
residual region&$\alpha$, $\bar\alpha$ merge into&transition region&infinite branch&plane wave\\ 
&a third real zero&&hits $E_1$, $\bar E_1$&region\\ 
\hline
\multicolumn{5}{c}{}\\
\end{tabular}
\caption{1st scenario: $0<\frac{A}{B}<1$.}
\label{1stscenariotable}
\end{table}
%-------------------%

%---------------------------------------%
%:fig 5.6
%---------------------------------------%
\begin{figure}[ht]
 \subcaptionbox{$\xi > \xi_{E_1}$}{
 \begin{overpic}[width=.3\textwidth]{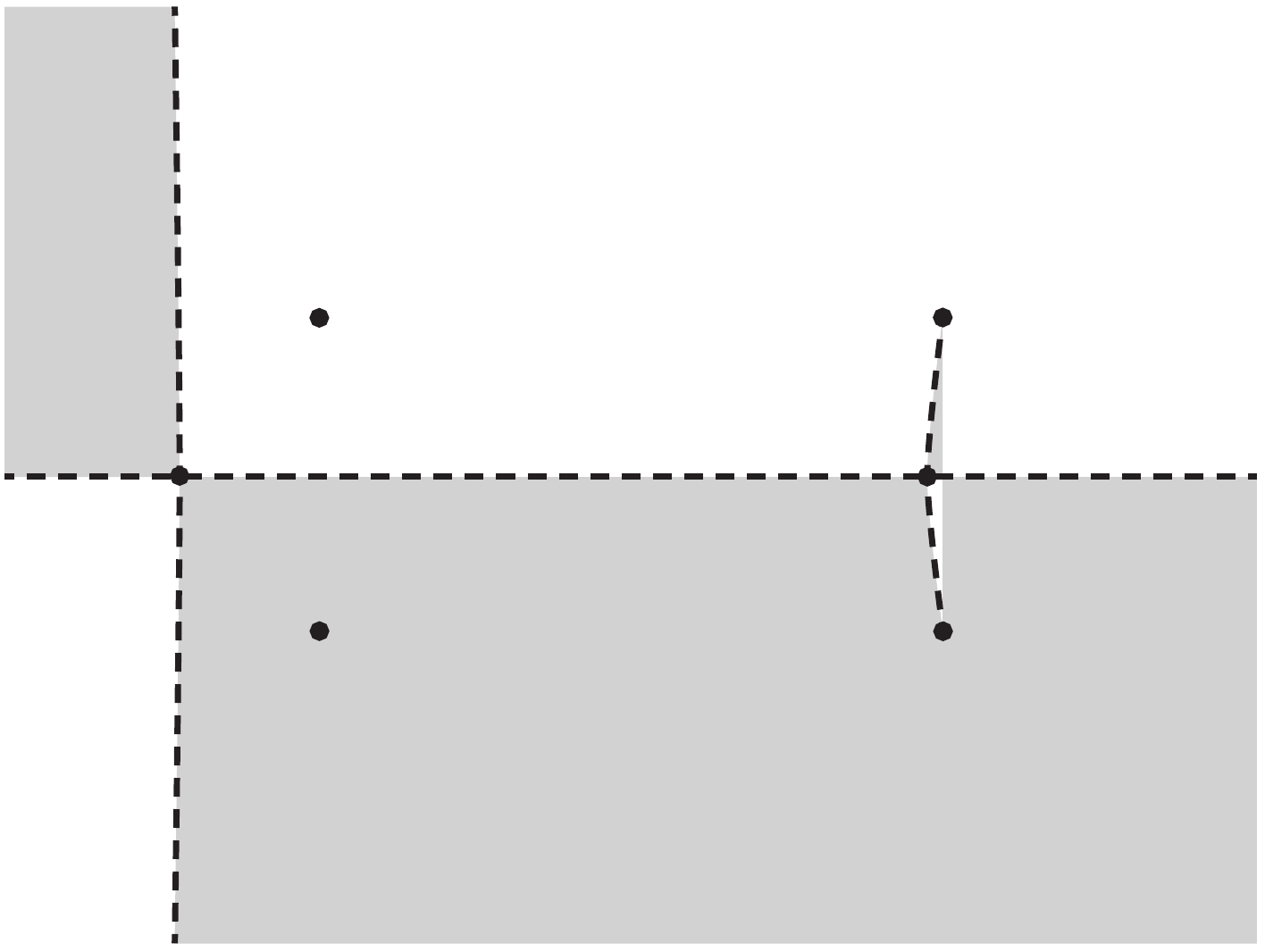}
       \put(60,62){\small $\Im g>0$}
   %    \put(79,40){\small $\Im g=0$}
       \put(60,10){\small $\Im g<0$}
     \put(16,49){\small $E_1$}
      \put(16,24){\small $\bar{E}_1$}
      \put(76,49){\small $E_2$}
      \put(76,24){\small $\bar{E}_2$}
      \put(16,40){\small $\mu_1$}
      \put(65,40){\small $\mu_2$}
\end{overpic}}
\hspace{2cm}
 \subcaptionbox{$\xi = \xi_{E_1}$}{
\begin{overpic}[width=.3\textwidth]{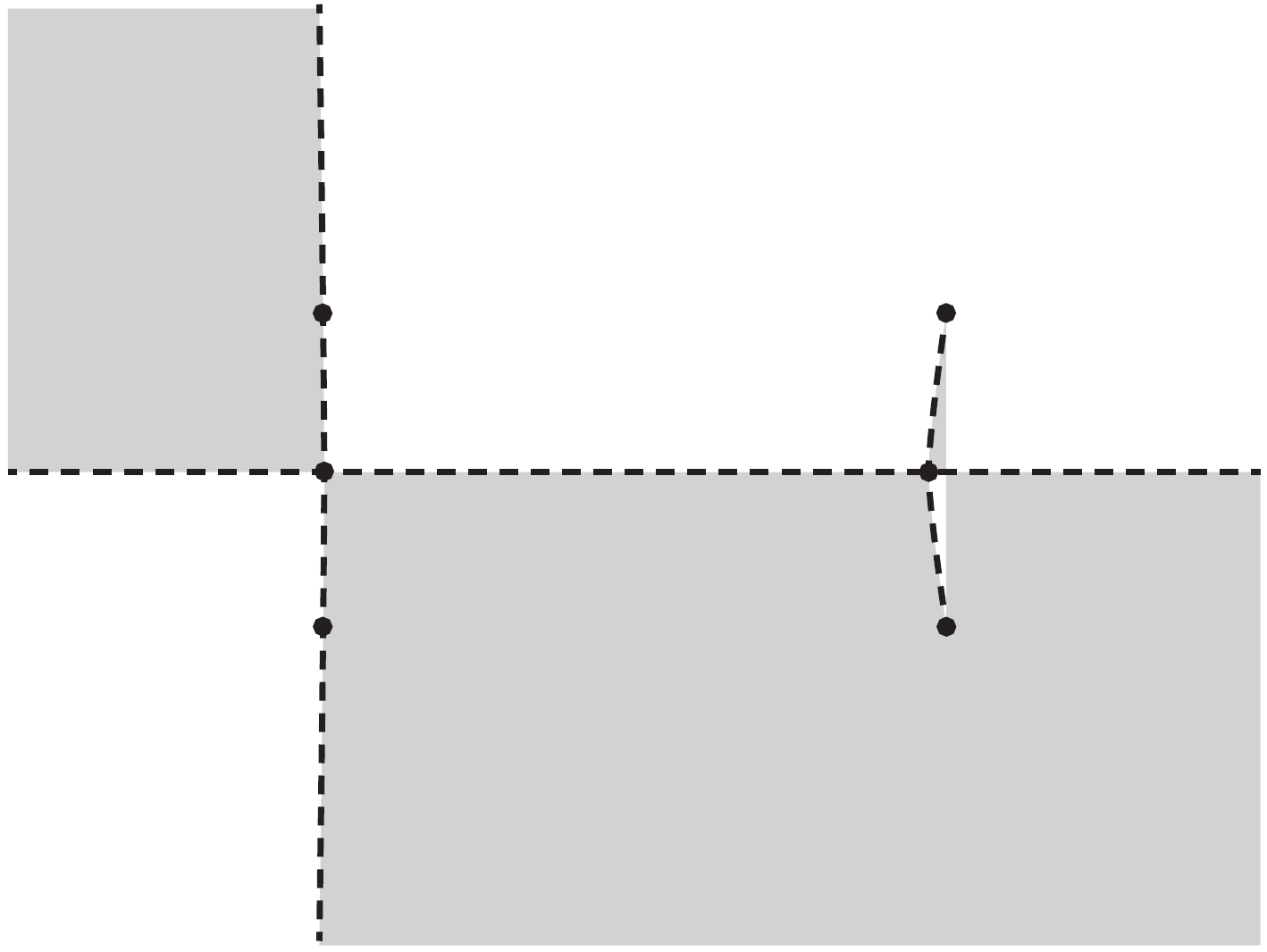}
       \put(60,62){\small $\Im g>0$}
   %    \put(79,40){\small $\Im g=0$}
       \put(60,10){\small $\Im g<0$}
     \put(16,49){\small $E_1$}
      \put(16,24){\small $\bar{E}_1$}
      \put(76,49){\small $E_2$}
      \put(76,24){\small $\bar{E}_2$}
      \put(27,40){\small $\mu_1$}
      \put(65,40){\small $\mu_2$}
\end{overpic}}
	\\ \vspace{.2cm}
 \subcaptionbox{$\xi_\alpha < \xi < \xi_{E_1}$}{
\begin{overpic}[width=.3\textwidth]{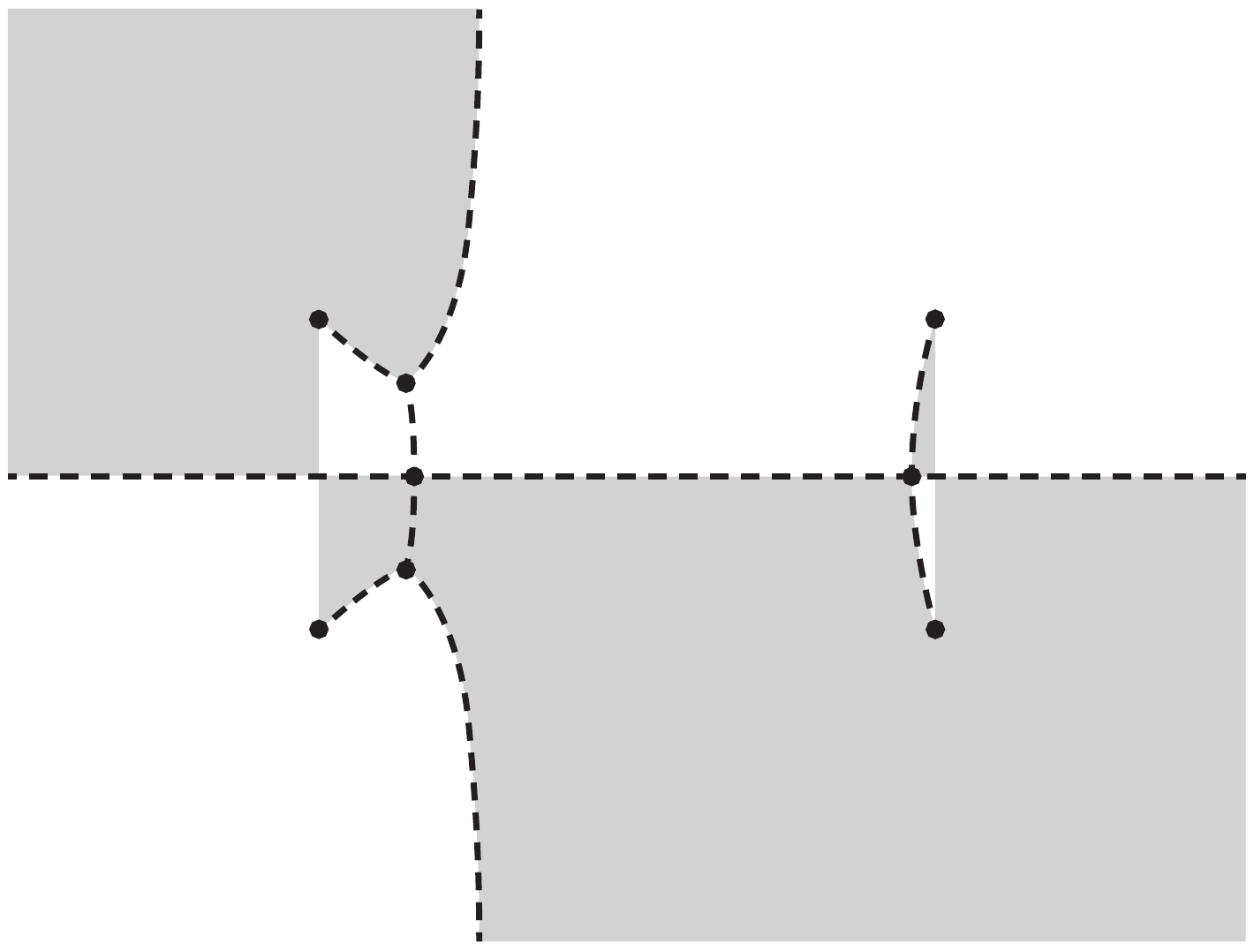}
       \put(60,62){\small $\Im g>0$}
   %    \put(79,40){\small $\Im g=0$}
       \put(60,10){\small $\Im g<0$}
     \put(16,49){\small $E_1$}
      \put(16,24){\small $\bar{E}_1$}
      \put(76,49){\small $E_2$}
      \put(76,24){\small $\bar{E}_2$}
      \put(35,40){\small $\mu_1$}
      \put(64,40){\small $\mu_2$}
      \put(30,48){\small $\alpha$}
      \put(30,25){\small $\bar{\alpha}$}
\end{overpic}}
\hspace{2cm}
 \subcaptionbox{$\xi = \xi_\alpha$}{
\begin{overpic}[width=.3\textwidth]{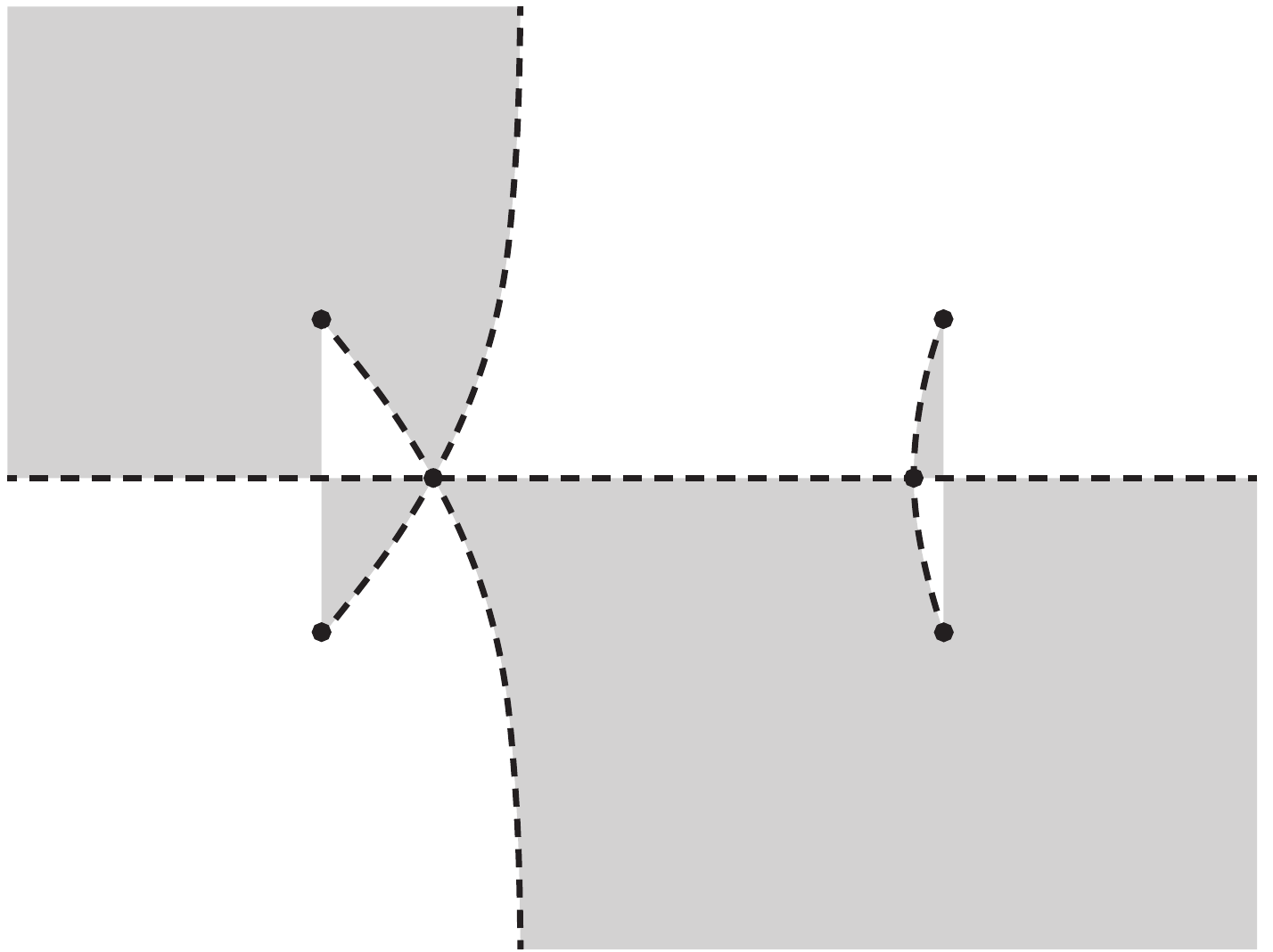}
       \put(60,62){\small $\Im g>0$}
   %    \put(79,40){\small $\Im g=0$}
       \put(60,10){\small $\Im g<0$}
     \put(16,49){\small $E_1$}
      \put(16,24){\small $\bar{E}_1$}
      \put(76,49){\small $E_2$}
      \put(76,24){\small $\bar{E}_2$}
      \put(37,39.5){\small $\mu_0 = \mu_1$}
      \put(63,39.5){\small $\mu_2$}
\end{overpic}}
	\\ \vspace{.2cm}
 \subcaptionbox{$0 < \xi < \xi_\alpha$}{
\begin{overpic}[width=.3\textwidth]{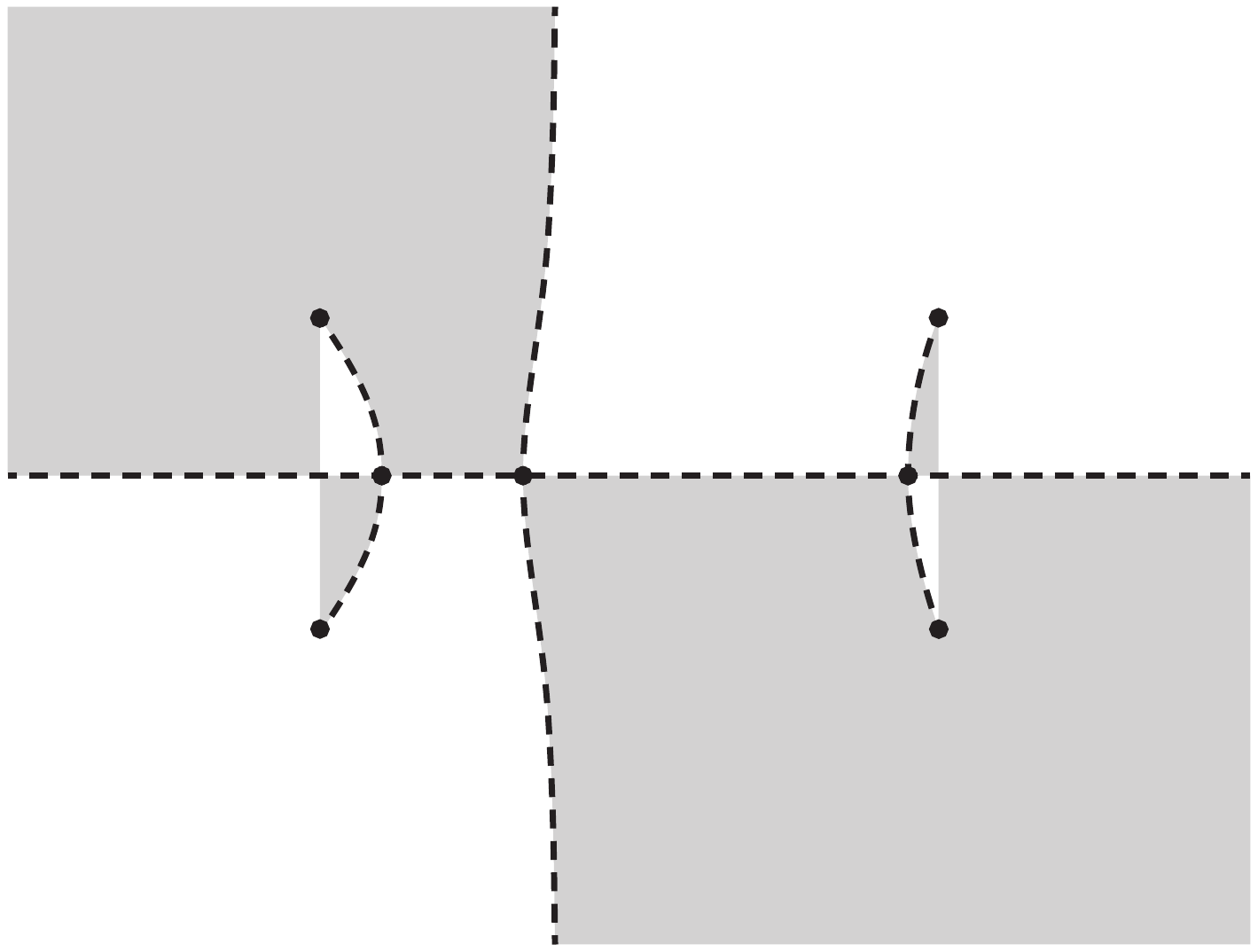}
       \put(60,62){\small $\Im g>0$}
   %    \put(79,40){\small $\Im g=0$}
       \put(60,10){\small $\Im g<0$}
     \put(16,49){\small $E_1$}
      \put(16,24){\small $\bar{E}_1$}
      \put(76,49){\small $E_2$}
      \put(76,24){\small $\bar{E}_2$}
      \put(32,40.5){\small $\mu_0$}
      \put(44,40.5){\small $\mu_1$}
      \put(64,40.5){\small $\mu_2$}
\end{overpic}}
\caption{Signature tables of $\Im g(\xi,k)$ corresponding to the five columns of Table \ref{1stscenariotable} of the 1st scenario. Each figure shows the zero level set $\Im g=0$ (dashed) and the regions where $\Im g<0$ (shaded) and $\Im g>0$ (white) in the complex $k$-plane for $\xi$ as indicated.} 
\label{fig:1stscenarioimg}
\end{figure}
%---------------------------------------%

%---------------------------------------%
%:fig 5.7
%---------------------------------------%
\begin{figure}[ht]
\centering\includegraphics[scale=.55]{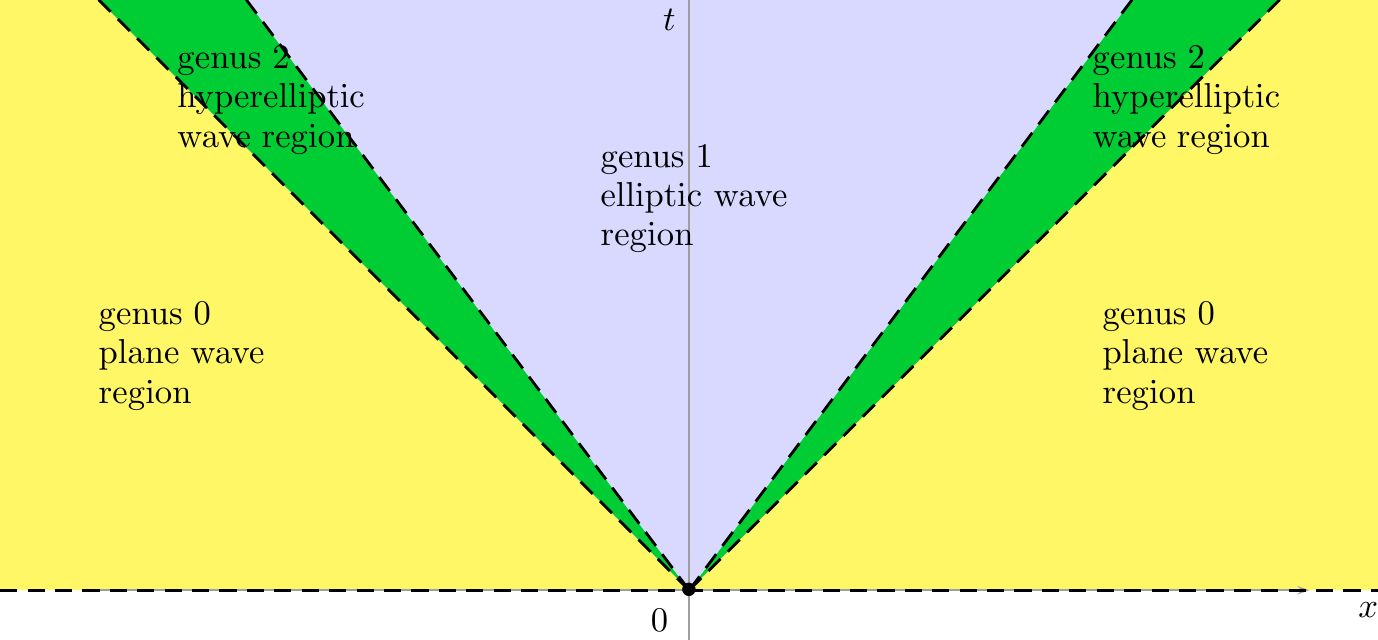}
\caption{1st scenario (symmetric shock case): $\frac{A}{B}<1$} 
\label{fig:shock-scenario-1}
\end{figure}
%---------------------------------------%

We are in Case 1. As $\xi$ decreases from $+\infty$, the $g$-function $g_2$ can be used to carry out the asymptotic analysis until the infinite branch of $\Im g_2=0$ hits $E_1$ and $\bar E_1$, i.e., as long as $\xi>\xi_{E_1}$. For $\xi<\xi_{E_1}$, a new $g$-function is needed, whose existence is established in Section \ref{implicitfunctiontheorem}. The derivative of this $g$-function has two real zeros $\mu_1$ and $\mu_2$, and two nonreal zeros $\alpha$ and $\bar\alpha$ which emerge from $E_1$ and $\bar E_1$ at $\xi=\xi_{E_1}$:
\begin{equation} \label{g-sce1}
g'(\xi,k)=4\frac{(k-\mu_1(\xi))(k-\mu_2(\xi))(k-\alpha(\xi))(k-\bar\alpha(\xi))}{\sqrt{(k-E_1)(k-\bar E_1)(k-E_2)(k-\bar E_2)(k-\alpha(\xi))(k-\bar\alpha(\xi))}}.
\end{equation}
The asymptotic analysis associated with \eqref{g-sce1} is developed in \cite{BV07}, assuming implicitly that the system of associated equations \cite{BV07}*{Eqs.~(3.29)} determining the parameters involved in \eqref{g-sce1} has a solution. It leads to genus~$2$ asymptotics for $q(x,t)$, in terms of functions attached to the hyperelliptic Riemann surface $M(\xi)$ defined by $w^2=(k-E_1)(k-\bar E_1)(k-E_2)(k-\bar E_2)(k-\alpha(\xi))(k-\bar\alpha(\xi))$. This new $g$-function remains appropriate until the nonreal zeros $\alpha(\xi)$ and $\bar\alpha(\xi)$ merge into a third real zero $\mu_0(\xi)$, which happens for $\xi=\xi_{\alpha}$. The real zeros $\mu_0$ and $\mu_1$ coincide for $\xi=\xi_\alpha$, but they move away from each other as $\xi$ decreases further. A numerically generated sequence of snapshots showing the zero level set $\Im g=0$ for different choices of $\xi$ corresponding to the five columns of Table \ref{1stscenariotable} are displayed in Figure~\ref{fig:1stscenarioimg}. The structure of the associated asymptotic sectors in the $(x,t)$-plane are shown in Figure~\ref{fig:shock-scenario-1}.

For $0\leq\xi\leq\xi_{\alpha}$, the asymptotic analysis can be carried out as in \cite{BV07}*{Section 3}, using a $g$-function whose derivative is as in \eqref{g-small-bv}:
\begin{equation}  \label{g-1-2}
g'(\xi,k)=4\frac{(k-\mu_1(\xi))(k-\mu_2(\xi))(k-\mu_0(\xi))}{\sqrt{(k-E_1)(k-\bar E_1)(k-E_2)(k-\bar E_2))}}\,.
\end{equation}
It is this scenario, with the $g$-functions \eqref{g-sce1} and \eqref{g-1-2}, that is presented in detail in \cite{BV07}.

%---------------------------------------------------------%
%:s.5.3.2
%---------------------------------------------------------%
\subsubsection{2nd Scenario}  \label{sec:scenario-2}

This is a limit case of the first scenario. In this case, $\xi_{\alpha}$ becomes $0$ and thus the genus $1$ range from the previous case shrinks to the single value $\xi=0$, with $g'(0,k)$ given by \eqref{g-0} and $\alpha(0)=\bar\alpha(0)=\mu_1(0)=0$.

%-------------------%
%:table 5.2
%-------------------%
\begin{table}[ht]
\begin{tabular}{|c|c|c|c|}
\hline
$\xi=0$&$0<\xi<\xi_{E_1}$&$\xi=\xi_{E_1}$&$\xi>\xi_{E_1}$\\
\hline
genus $1$&genus $2$&&genus $0$\\ 
\hline
$\alpha$, $\bar\alpha$, $\mu_1$ all&&the infinite branch&\\ 
merge at the origin&&hits $E_1$, $\bar E_1$&\\ 
\hline
\multicolumn{4}{c}{}\\
\end{tabular}
\caption{2nd scenario: $\frac{A}{B}=1$.}
\label{2ndscenariotable}
\end{table}
%-------------------%

%---------------------------------------%
%:fig 5.8
%---------------------------------------%
\begin{figure}[ht]
 \subcaptionbox{$\xi>\xi_{E_1}$}{
 \begin{overpic}[width=.28\textwidth]{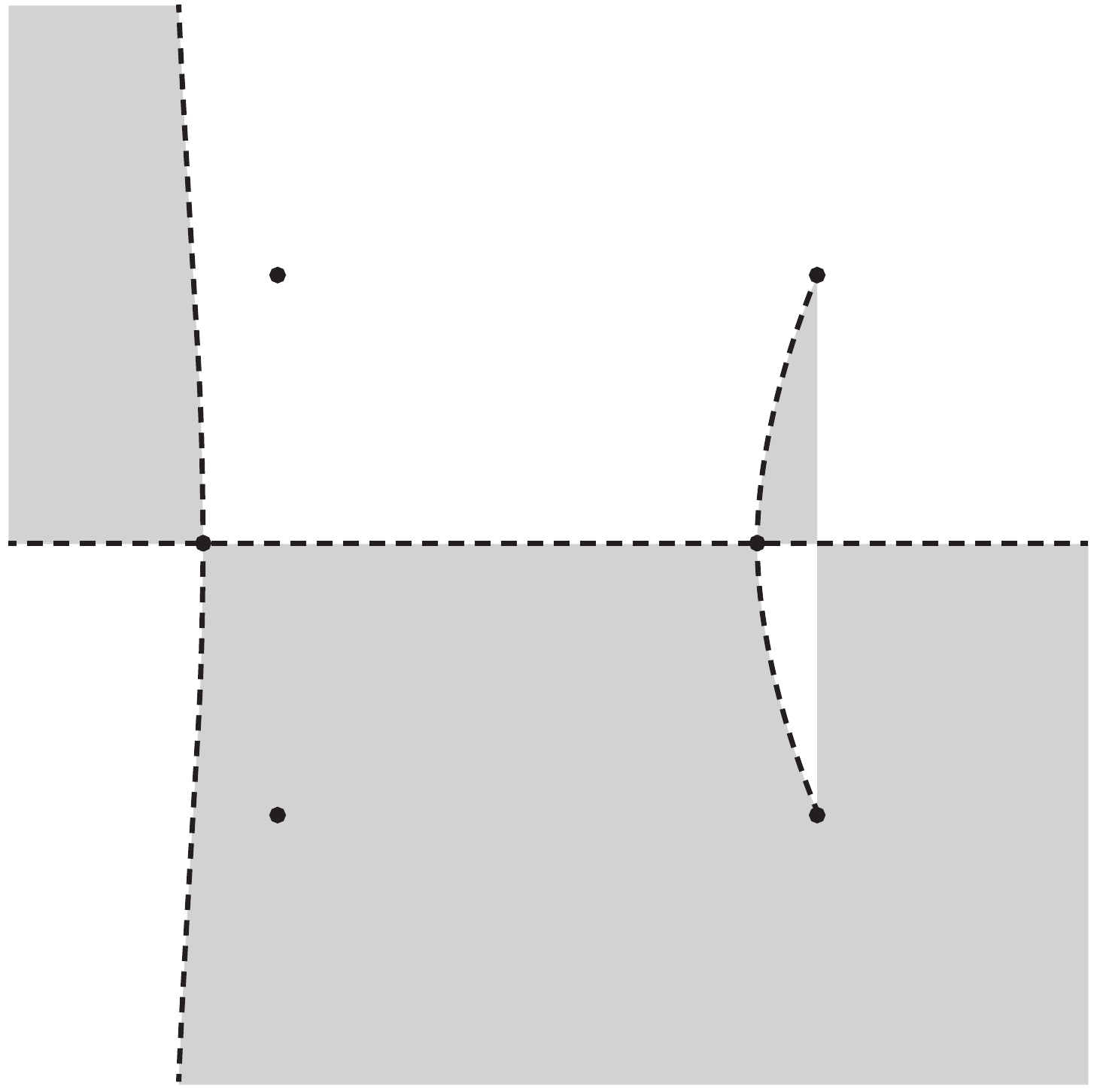}
       \put(60,86){\small $\Im g>0$}
       \put(60,10){\small $\Im g<0$}
     \put(27,72){\small $E_1$}
      \put(27,23){\small $\bar{E}_1$}
      \put(77,72){\small $E_2$}
      \put(77,23){\small $\bar{E}_2$}
      \put(20,53){\small $\mu_1$}
      \put(59,53){\small $\mu_2$}
\end{overpic}}
\hspace{2cm}
 \subcaptionbox{$\xi = \xi_{E_1}$}{
\begin{overpic}[width=.28\textwidth]{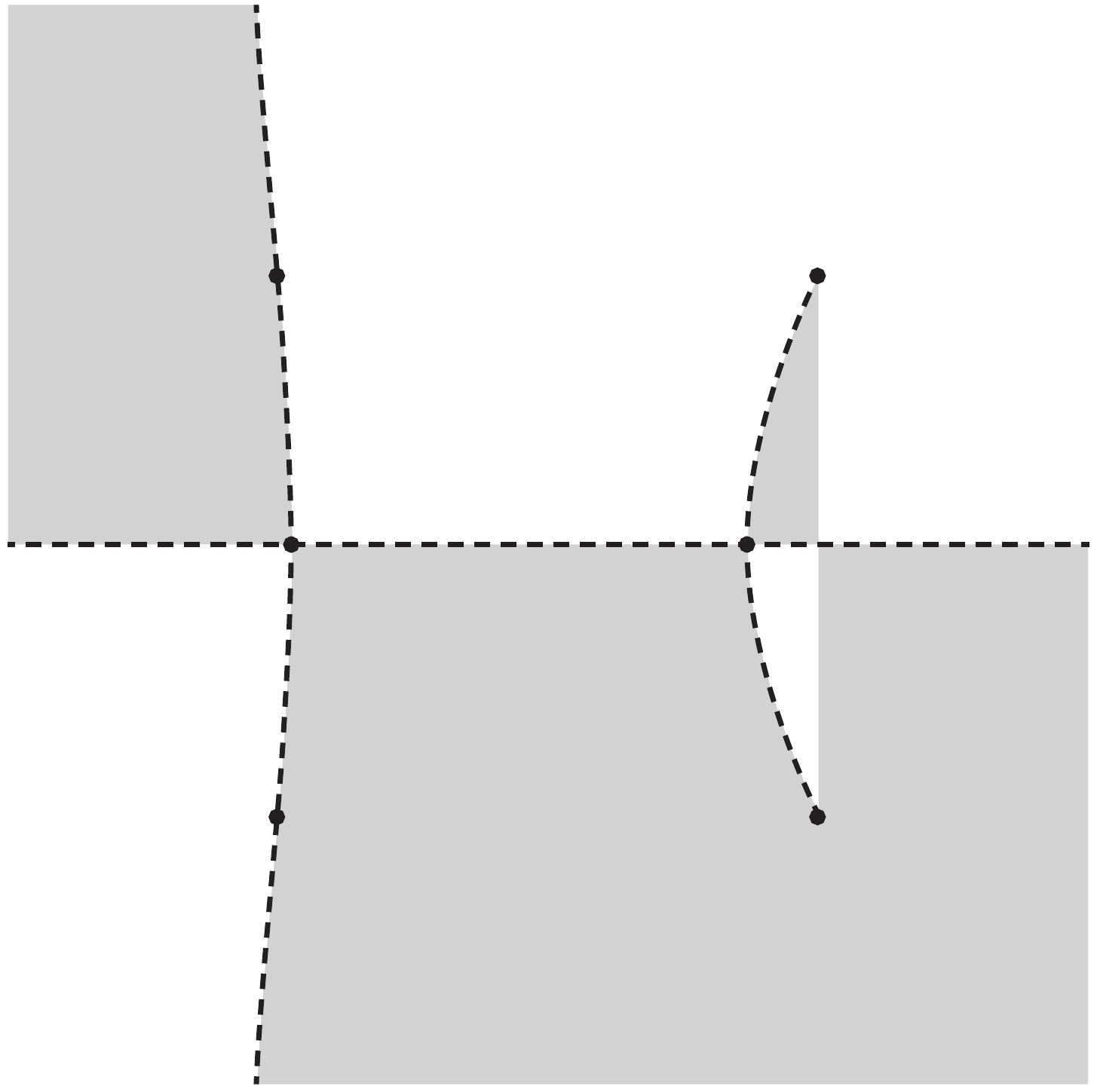}
       \put(60,86){\small $\Im g>0$}
       \put(60,10){\small $\Im g<0$}
    \put(14,72){\small $E_1$}
      \put(14,23){\small $\bar{E}_1$}
      \put(77,72){\small $E_2$}
      \put(77,23){\small $\bar{E}_2$}
      \put(28,53){\small $\mu_1$}
      \put(58,53){\small $\mu_2$}
\end{overpic}}
	\\ \vspace{.2cm}
 \subcaptionbox{$\xi_\alpha < \xi < \xi_{E_1}$}{
\begin{overpic}[width=.28\textwidth]{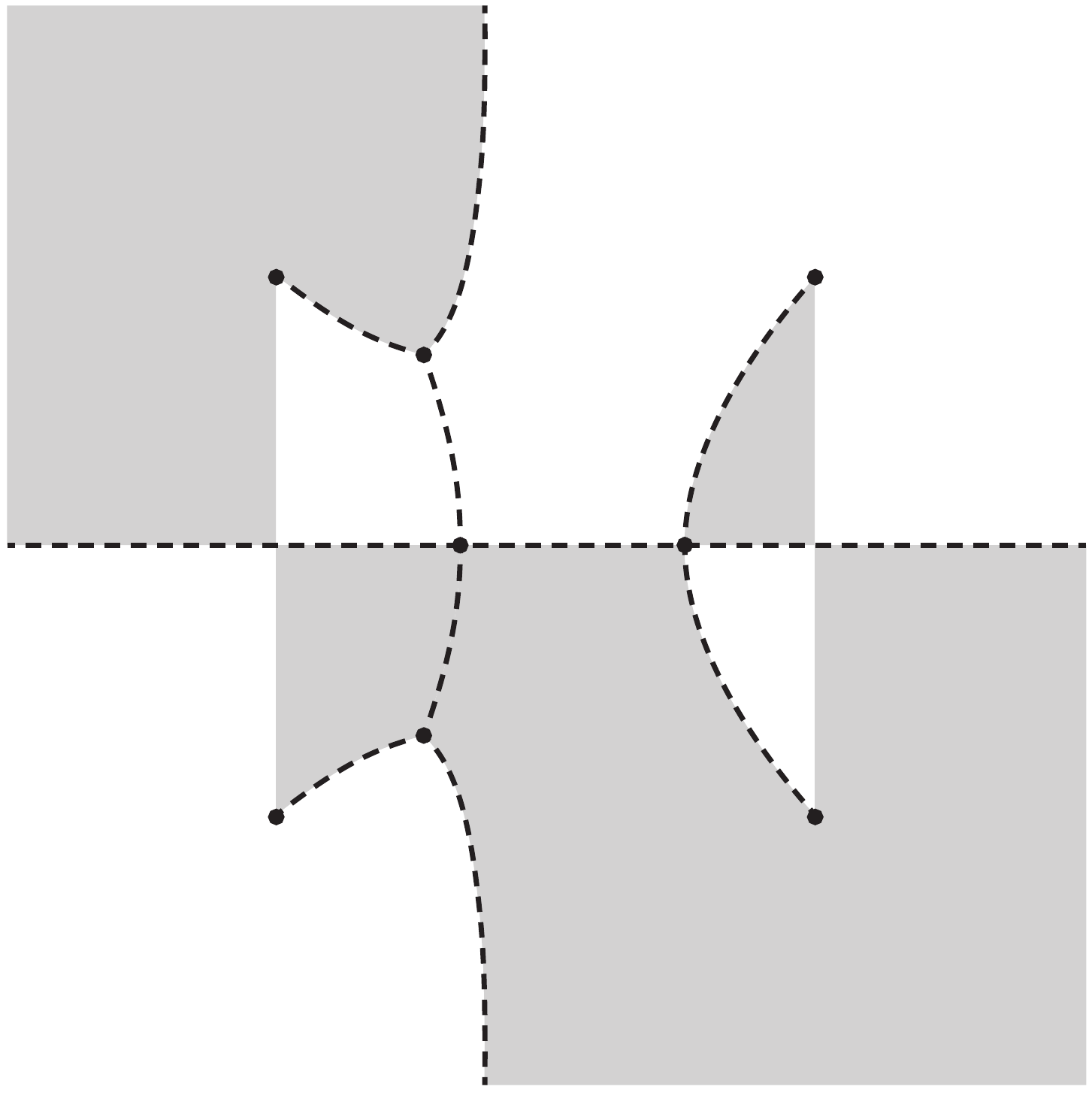}
       \put(60,86){\small $\Im g>0$}
       \put(60,10){\small $\Im g<0$}
     \put(14,72){\small $E_1$}
      \put(14,23){\small $\bar{E}_1$}
      \put(77,72){\small $E_2$}
      \put(77,23){\small $\bar{E}_2$}
      \put(32,53){\small $\mu_1$}
      \put(53,53){\small $\mu_2$}
     \put(41,66){\small $\alpha$}
      \put(41,32){\small $\bar{\alpha}$}
\end{overpic}}
\hspace{2cm}
 \subcaptionbox{$\xi = 0$}{
\begin{overpic}[width=.28\textwidth]{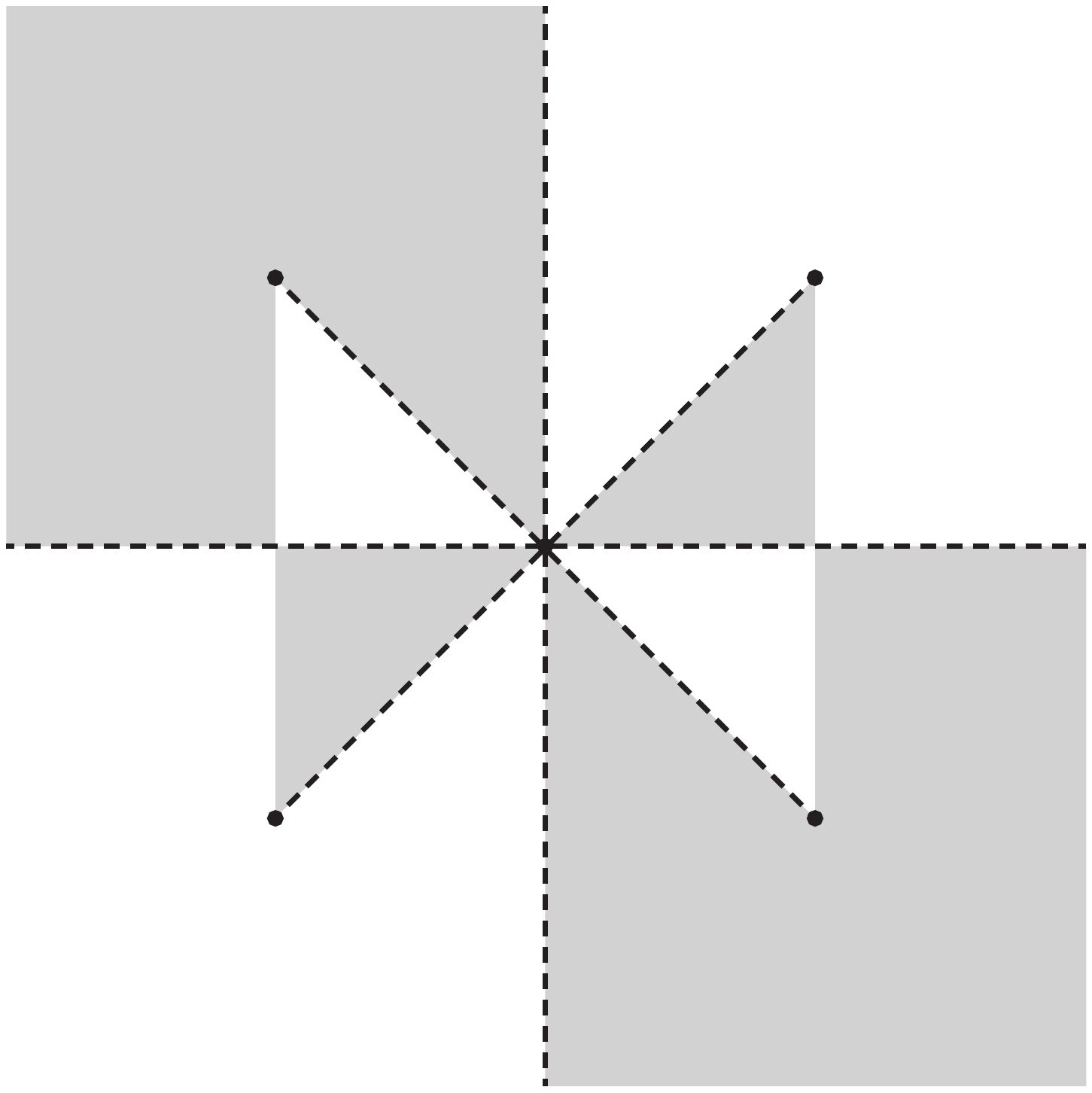}
      \put(60,86){\small $\Im g>0$}
       \put(60,10){\small $\Im g<0$}
     \put(14,72){\small $E_1$}
      \put(14,23){\small $\bar{E}_1$}
      \put(77,72){\small $E_2$}
      \put(77,23){\small $\bar{E}_2$}
     \put(57,53){\small $\mu_1 = \mu_2$}
\end{overpic}}
\caption{Signature tables of $\Im g(\xi,k)$ corresponding to the four columns of Table \ref{2ndscenariotable} of the 2nd scenario. Each figure shows the zero level set $\Im g=0$ (dashed) and the regions where $\Im g<0$ (shaded) and $\Im g>0$ (white) in the complex $k$-plane for $\xi$ as indicated.}
\label{fig:2ndscenarioimg}
\end{figure}
%---------------------------------------%

%---------------------------------------%
%:fig 5.9
%---------------------------------------%
\begin{figure}[ht]
\centering\includegraphics[scale=.55]{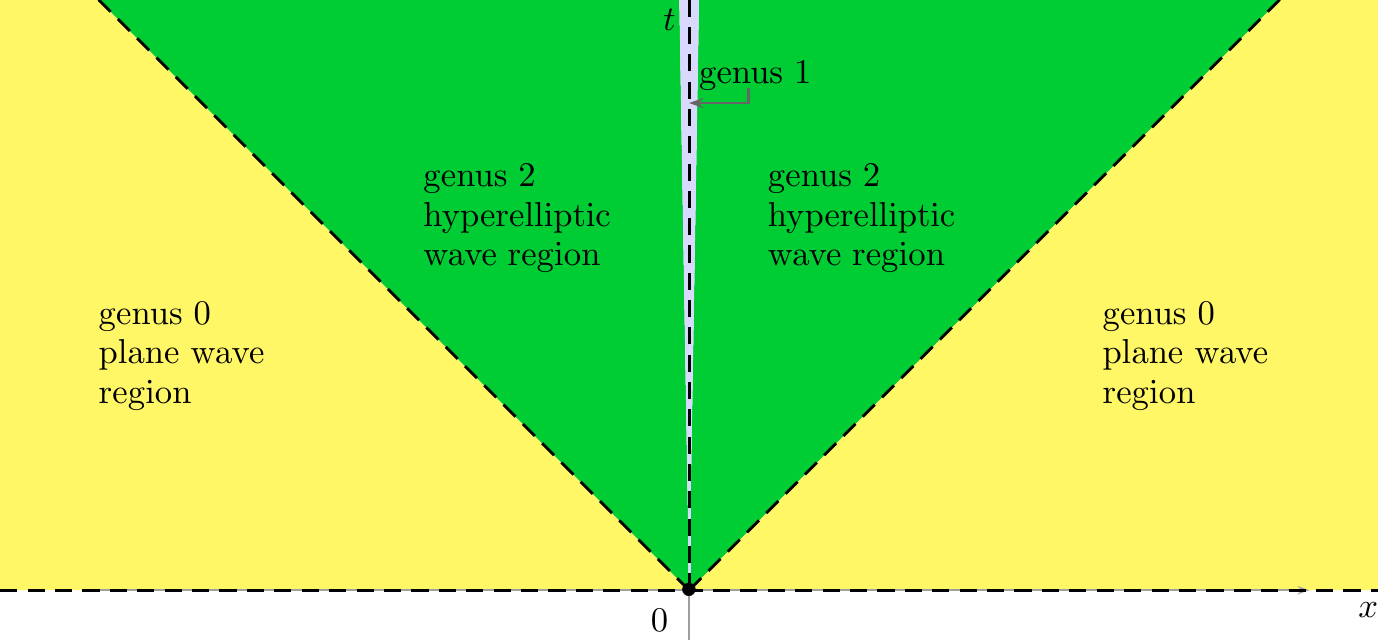}
\caption{2nd scenario (symmetric shock case): $\frac{A}{B}=1$} 
\label{fig:shock-scenario-2}
\end{figure}
%---------------------------------------%

%---------------------------------------------------------%
%:s.5.3.3
%---------------------------------------------------------%
\subsubsection{3rd Scenario}  \label{sec:scenario-3}

We are still in Case 1. As $\xi$ decreases from $+\infty$, the $g$-function $g_2$ is appropriate as long as $\xi>\xi_{E_1}$. Then, a new $g$-function is required whose derivative $g'$ has the form \eqref{g-sce1} and thus the asymptotics can be computed as in \cite{BV07}*{Section 4}. This $g$-function remains appropriate until the two real zeros $\mu_1(\xi)$ and $\mu_2(\xi)$ of $g'$ merge, which happens for $\xi=\xi_{\mu}$. Finally, for $0<\xi<\xi_{\mu}$, a third $g$-function is to be considered with derivative of the form \eqref{g-xi}, that is,
\[
g'(\xi,k)=4\frac{(k-\mu(\xi))(k-\alpha(\xi))(k-\bar\alpha(\xi))(k-\beta(\xi))(k-\bar\beta(\xi))}{w(\xi,k)},
\]
with
\begin{equation} \label{w3}
w^2=(k-E_1)(k-\bar E_1)(k-E_2)(k-\bar E_2)(k-\alpha(\xi))(k-\bar\alpha(\xi))(k-\beta(\xi))(k-\bar\beta(\xi)).
\end{equation}
This leads to a genus $3$ asymptotic formula, expressed in terms of hyperelliptic functions attached to the Riemann surface defined by \eqref{w3}. All details are given in \cite{BLS20b}. Results on the asymptotics in the transition zone near $\xi=0$ where the Riemann surface degenerates from genus $3$ to genus $1$ can be found in \cite{BLS20c}.

%-------------------%
%:table 5.3
%-------------------%
\begin{table}[ht]
\begin{tabular}{|c|c|c|c|c|c|}
\hline
$\xi=0$&$0<\xi<\xi_{\mu}$&$\xi=\xi_{\mu}$&$\xi_{\mu}<\xi<\xi_{E_1}$&$\xi=\xi_{E_1}$&$\xi>\xi_{E_1}$\\
\hline
genus $1$&genus $3$&&genus $2$&&genus $0$\\ 
\hline
$\alpha$, $\beta$ merge&&the real zeros&&the infinite branch&\\ 
&&$\mu_1$, $\mu_2$ merge&&hits $E_1$, $\bar E_1$&\\ 
\hline
\multicolumn{6}{c}{}\\
\end{tabular}
\caption{3rd scenario: $1<\frac{A}{B}<\frac{2}{7}(2+3\sqrt{2})$.}\label{3rdscenariotable}
\end{table}
%-------------------%

%---------------------------------------%
%:fig 5.10
%---------------------------------------%
\begin{figure}[ht]
 \subcaptionbox{$\xi>\xi_{E_1}$}{
 \begin{overpic}[width=.28\textwidth]{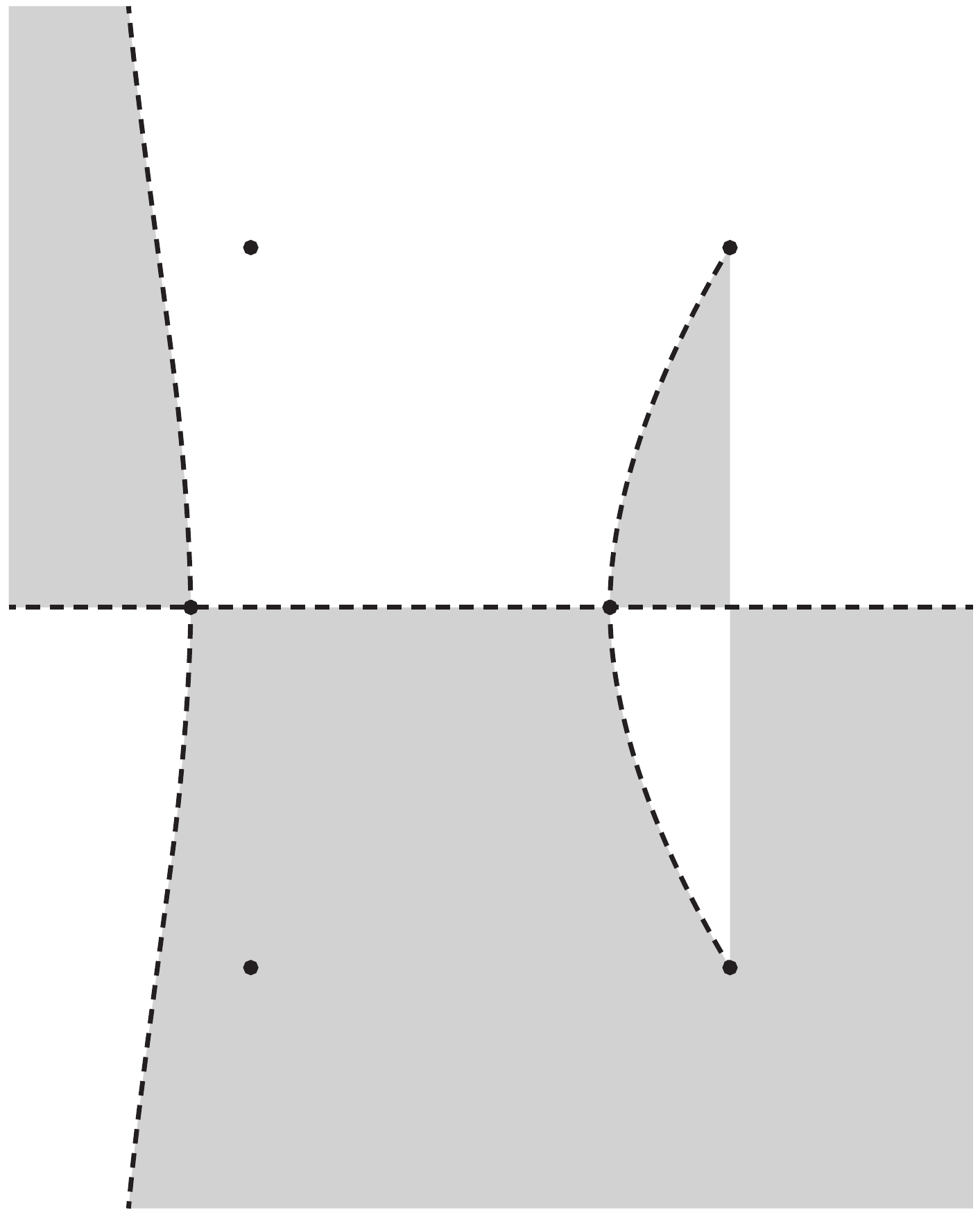}
       \put(40,87){\small $\Im g>0$}
       \put(40,8){\small $\Im g<0$}
     \put(22,77){\small $E_1$}
      \put(22,18){\small $\bar{E}_1$}
      \put(62,77){\small $E_2$}
      \put(62,18){\small $\bar{E}_2$}
      \put(17,53){\small $\mu_1$}
      \put(42,53){\small $\mu_2$}
\end{overpic}}
\hspace{.5cm}
 \subcaptionbox{$\xi=\xi_{E_1}$}{
\begin{overpic}[width=.28\textwidth]{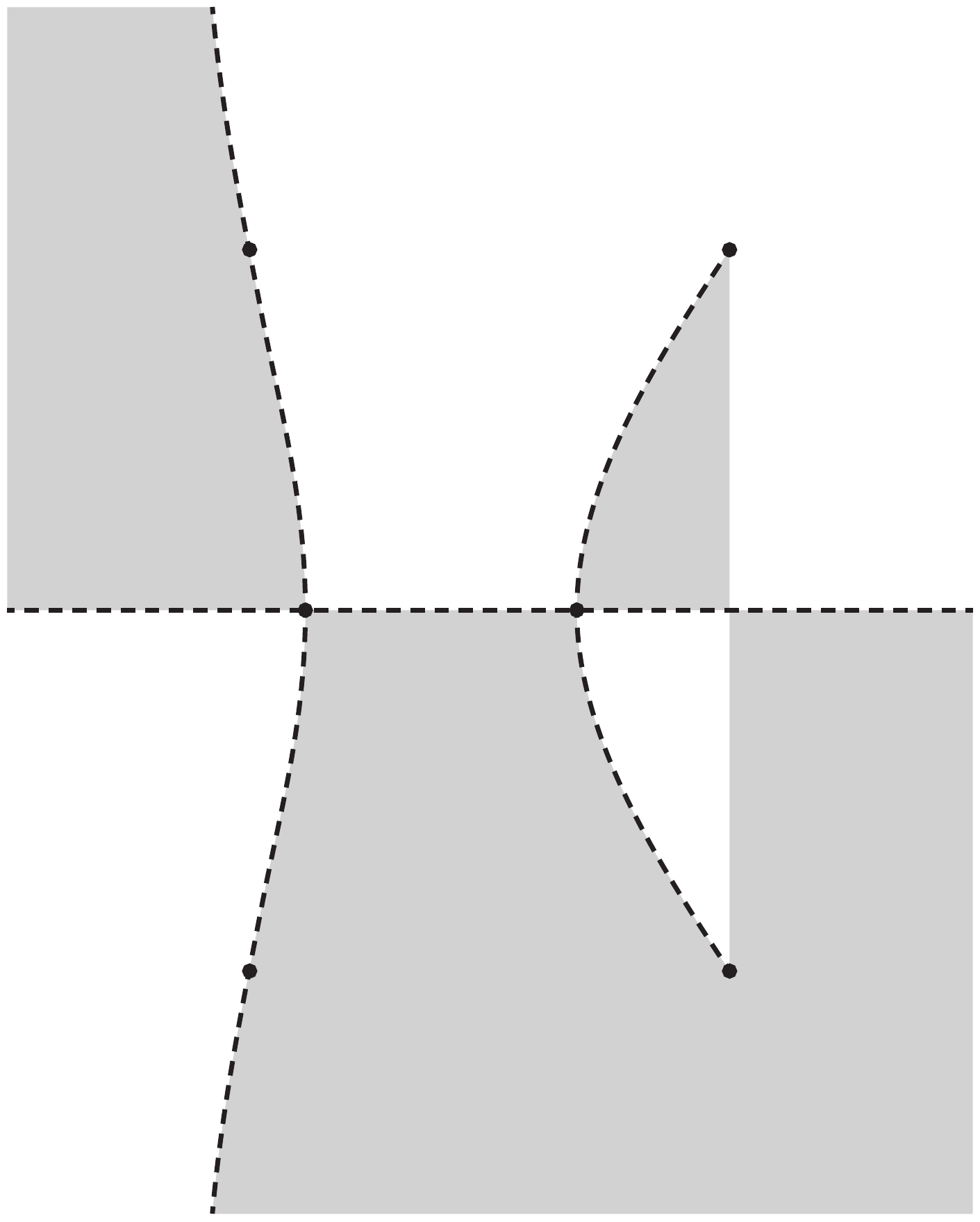}
       \put(40,87){\small $\Im g>0$}
       \put(40,8){\small $\Im g<0$}
     \put(11,77){\small $E_1$}
      \put(11,18){\small $\bar{E}_1$}
      \put(62,77){\small $E_2$}
      \put(62,18){\small $\bar{E}_2$}
      \put(17,53){\small $\mu_1$}
      \put(39,53){\small $\mu_2$}
\end{overpic}}
\hspace{.5cm}
 \subcaptionbox{$\xi_\mu<\xi<\xi_{E_1}$}{
\begin{overpic}[width=.28\textwidth]{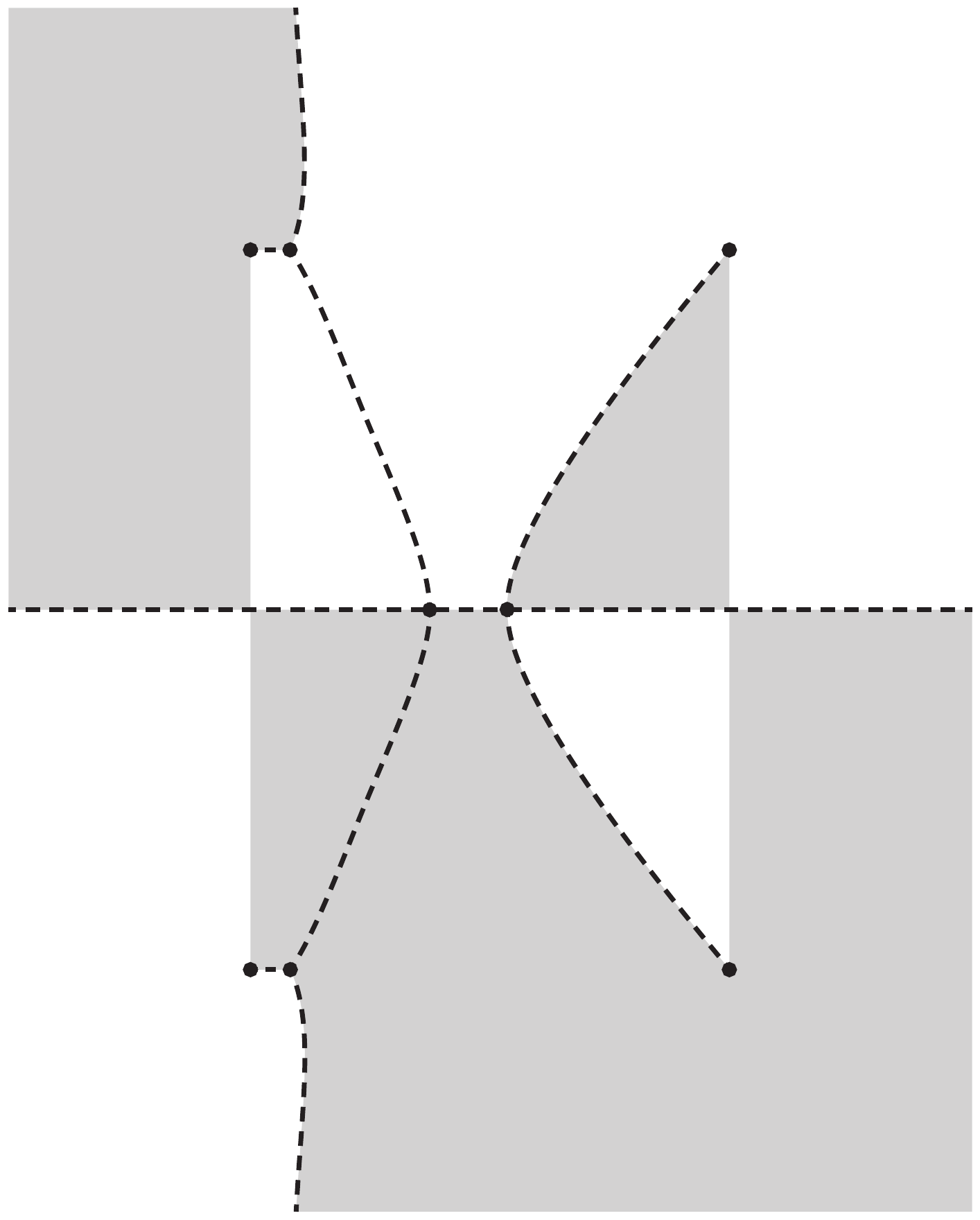}
        \put(40,87){\small $\Im g>0$}
       \put(40,8){\small $\Im g<0$}
    \put(11,77){\small $E_1$}
      \put(11,18){\small $\bar{E}_1$}
       \put(62,77){\small $E_2$}
      \put(62,18){\small $\bar{E}_2$}
      \put(27,53){\small $\mu_1$}
      \put(44,53){\small $\mu_2$}
   \put(26,78){\small $\alpha$}
      \put(26,19){\small $\bar{\alpha}$}
 \end{overpic}}
	\\ \vspace{.2cm}
 \subcaptionbox{$\xi=\xi_\mu$}{
\begin{overpic}[width=.28\textwidth]{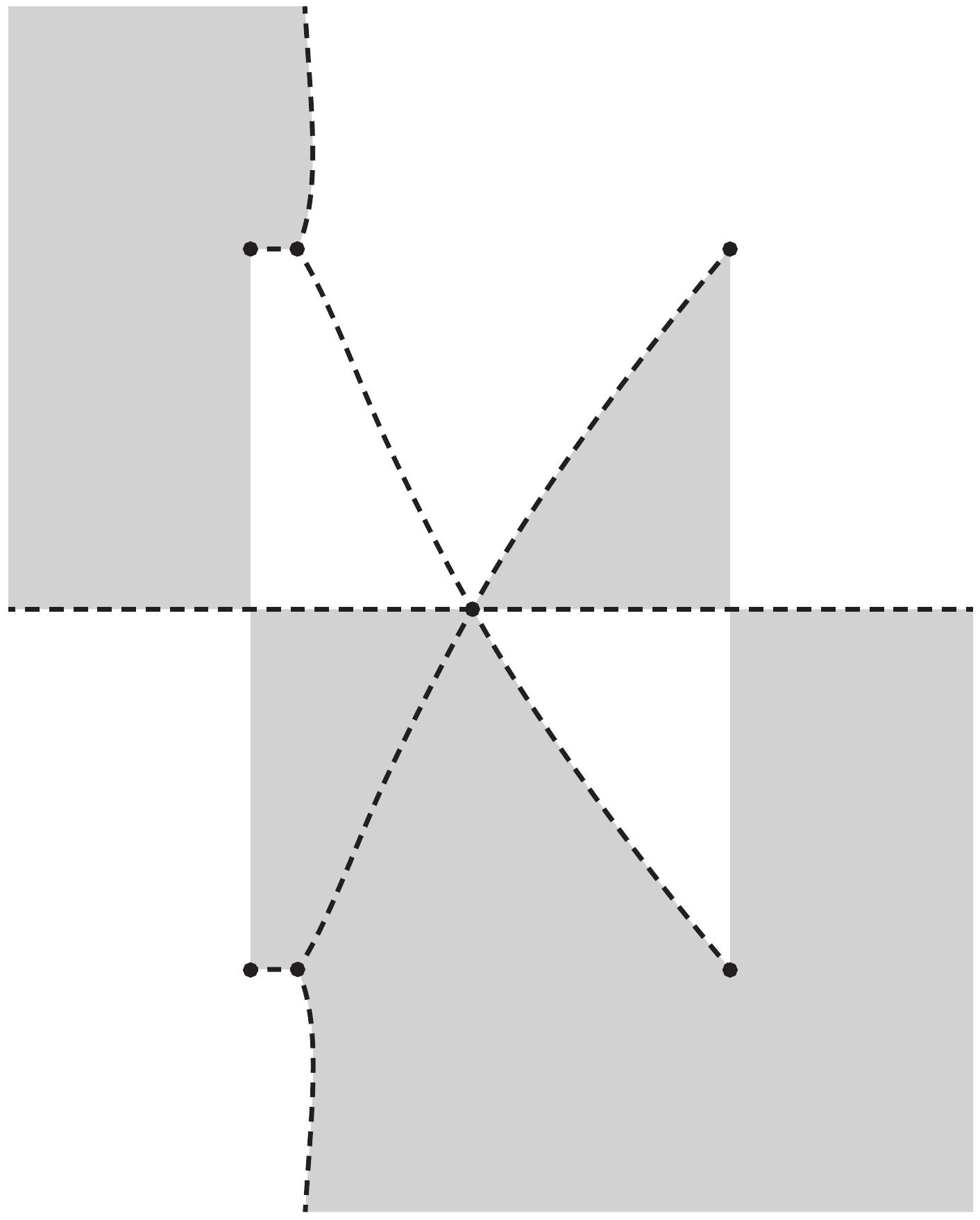}
        \put(40,87){\small $\Im g>0$}
       \put(40,8){\small $\Im g<0$}
     \put(11,77){\small $E_1$}
      \put(11,18){\small $\bar{E}_1$}
      \put(62,77){\small $E_2$}
      \put(62,18){\small $\bar{E}_2$}
      \put(43,53){\small $\mu_1=\mu_2$}
   \put(26,78){\small $\alpha$}
      \put(26,19){\small $\bar{\alpha}$}
\end{overpic}}
\hspace{.5cm}
 \subcaptionbox{$0 < \xi < \xi_\mu$}{
\begin{overpic}[width=.28\textwidth]{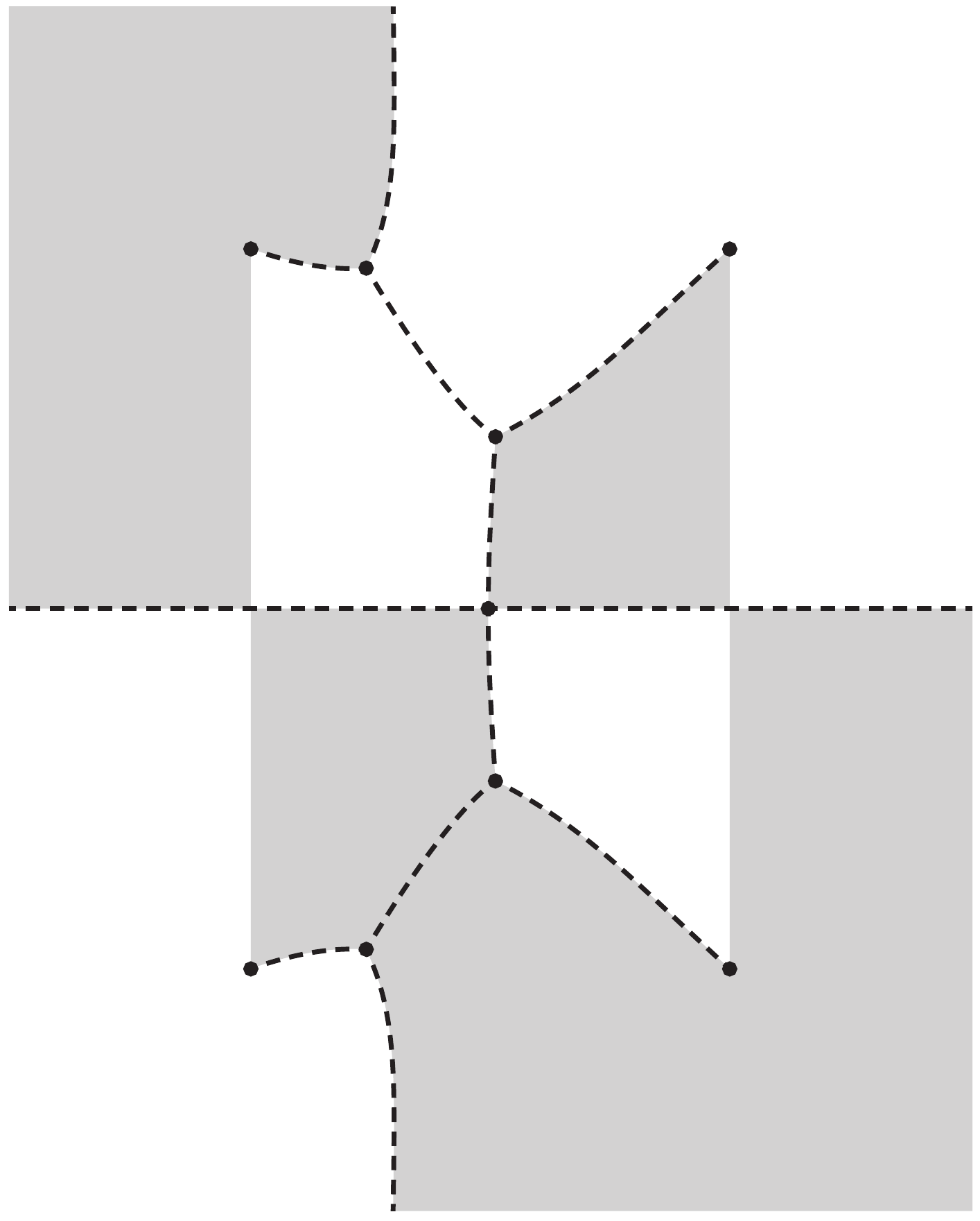}
       \put(40,87){\small $\Im g>0$}
       \put(40,8){\small $\Im g<0$}
     \put(11,77){\small $E_1$}
      \put(11,18){\small $\bar{E}_1$}
      \put(62,77){\small $E_2$}
      \put(62,18){\small $\bar{E}_2$}
      \put(35,52.5){\small $\mu$}
   \put(32,76){\small $\alpha$}
      \put(32,21){\small $\bar{\alpha}$}
   \put(42,60){\small $\beta$}
      \put(42,37){\small $\bar{\beta}$}
\end{overpic}}
\hspace{.5cm}
 \subcaptionbox{$\xi = 0$}{
\begin{overpic}[width=.28\textwidth]{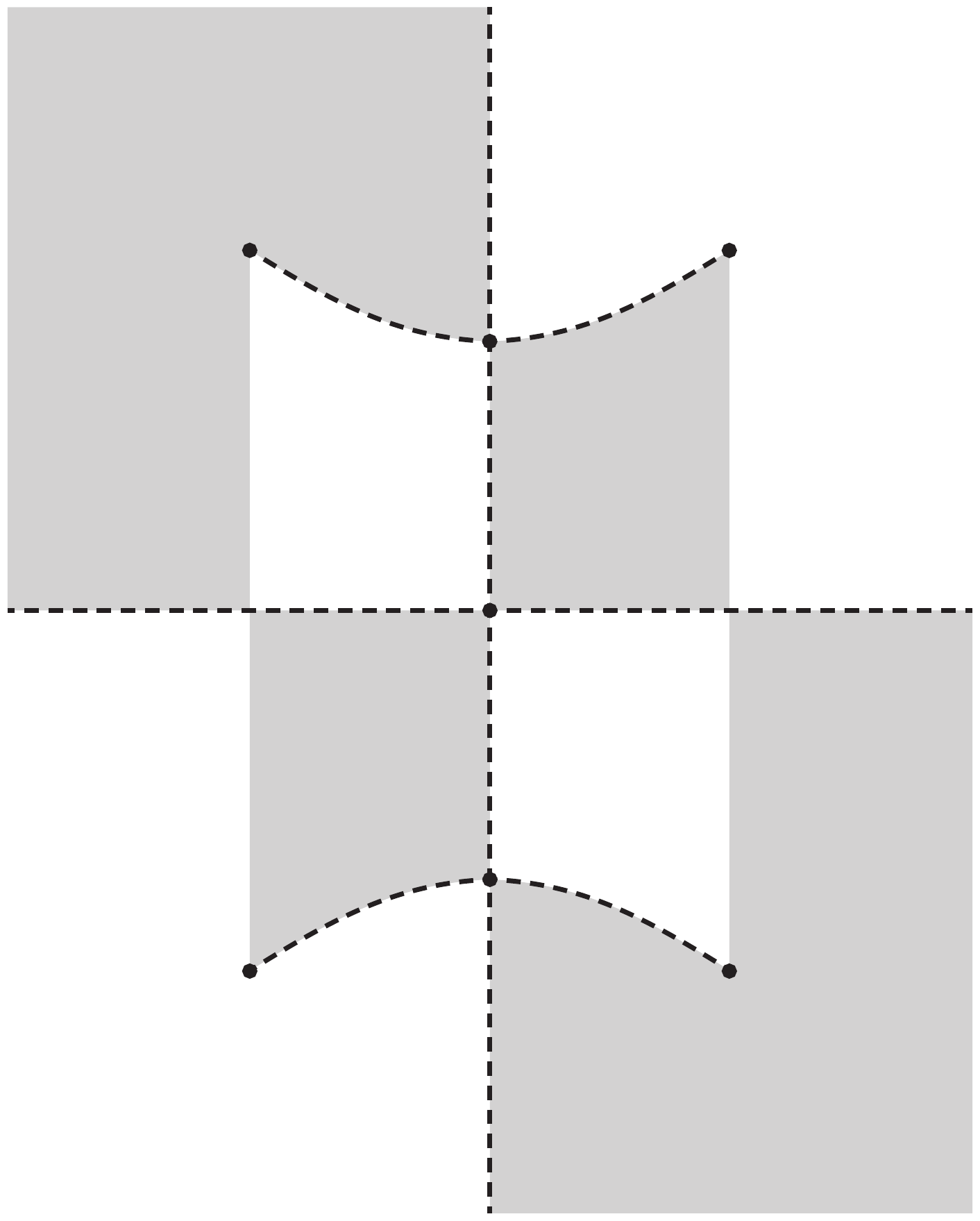}
        \put(50,87){\small $\Im g>0$}
       \put(50,8){\small $\Im g<0$}
     \put(11,77){\small $E_1$}
      \put(11,18){\small $\bar{E}_1$}
      \put(62,77){\small $E_2$}
      \put(62,18){\small $\bar{E}_2$}
      \put(35,52.5){\small $\mu$}
   \put(41,68){\small $\alpha = \beta$}
      \put(41,29.5){\small $\bar{\alpha}= \bar{\beta}$}
\end{overpic}}
\caption{Signature tables of $\Im g(\xi,k)$ corresponding to the six columns of Table \ref{3rdscenariotable} of the 3rd scenario. Each figure shows the zero level set $\Im g=0$ (dashed) and the regions where $\Im g<0$ (shaded) and $\Im g>0$ (white) in the complex $k$-plane for $\xi$ as indicated.}
\label{fig:3rdscenarioimg}
\end{figure}
%---------------------------------------%

%---------------------------------------%
%:fig 5.11
%---------------------------------------%
\begin{figure}[ht]
\centering\includegraphics[scale=.55]{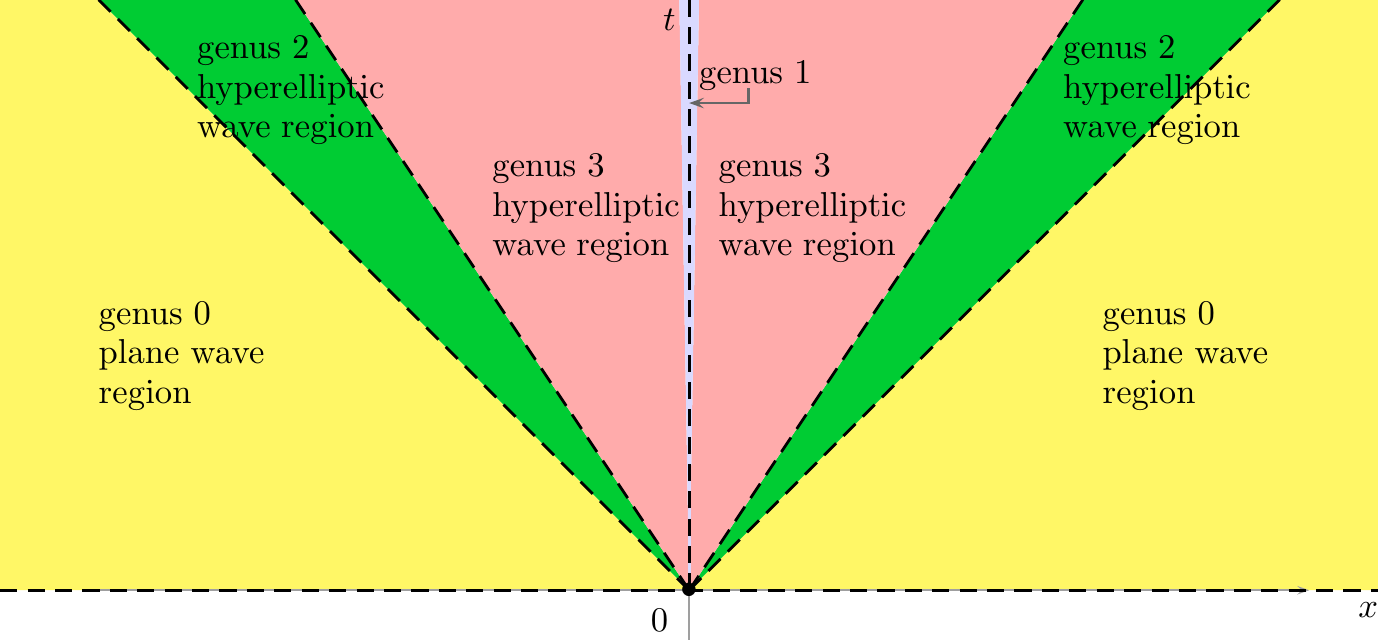}
\caption{3rd scenario (symmetric shock case): $1<\frac{A}{B}<\frac{2}{7}(2+3\sqrt{2})$} 
\label{fig:shock-scenario-3}
\end{figure}
%---------------------------------------%

%---------------------------------------------------------%
%:s.5.3.4
%---------------------------------------------------------%
\subsubsection{4th Scenario}  \label{sec:scenario-4}

%-------------------%
%:table 5.4
%-------------------%
\begin{table}[ht]
\begin{tabular}{|c|c|c|c|}
\hline
$\xi=0$&$0<\xi<\xi_{E_1}$&$\xi=\xi_{E_1}=\xi_{\merge}$&$\xi>\xi_{\merge}$\\
\hline
genus $1$&genus $3$&&genus $0$\\ 
\hline
$\alpha$, $\beta$ merge&&the infinite branch hits $E_1$, $\bar E_1$&\\ 
&&and the real zeros $\mu_1$, $\mu_2$ merge&\\ 
\hline
\multicolumn{4}{c}{}\\
\end{tabular}
\caption{4th scenario: $\frac{A}{B}=\frac{2}{7}(2+3\sqrt{2})$.}\label{4thscenariotable}
\end{table}
%-------------------%

We are in Case 3, where $\xi_{E_1}=\xi_{\merge}$. This is a limiting case of the third scenario when $\xi_{\mu}=\xi_{E_1}$. Thus, the genus $2$ sector collapses and the genus $3$ sector $0<\xi<\xi_{\mu}$ becomes directly adjacent to the plane wave sector.

%---------------------------------------%
%:fig 5.12
%---------------------------------------%
\begin{figure}[ht]
 \subcaptionbox{$\xi > \xi_{\merge}$}{
 \begin{overpic}[width=.28\textwidth]{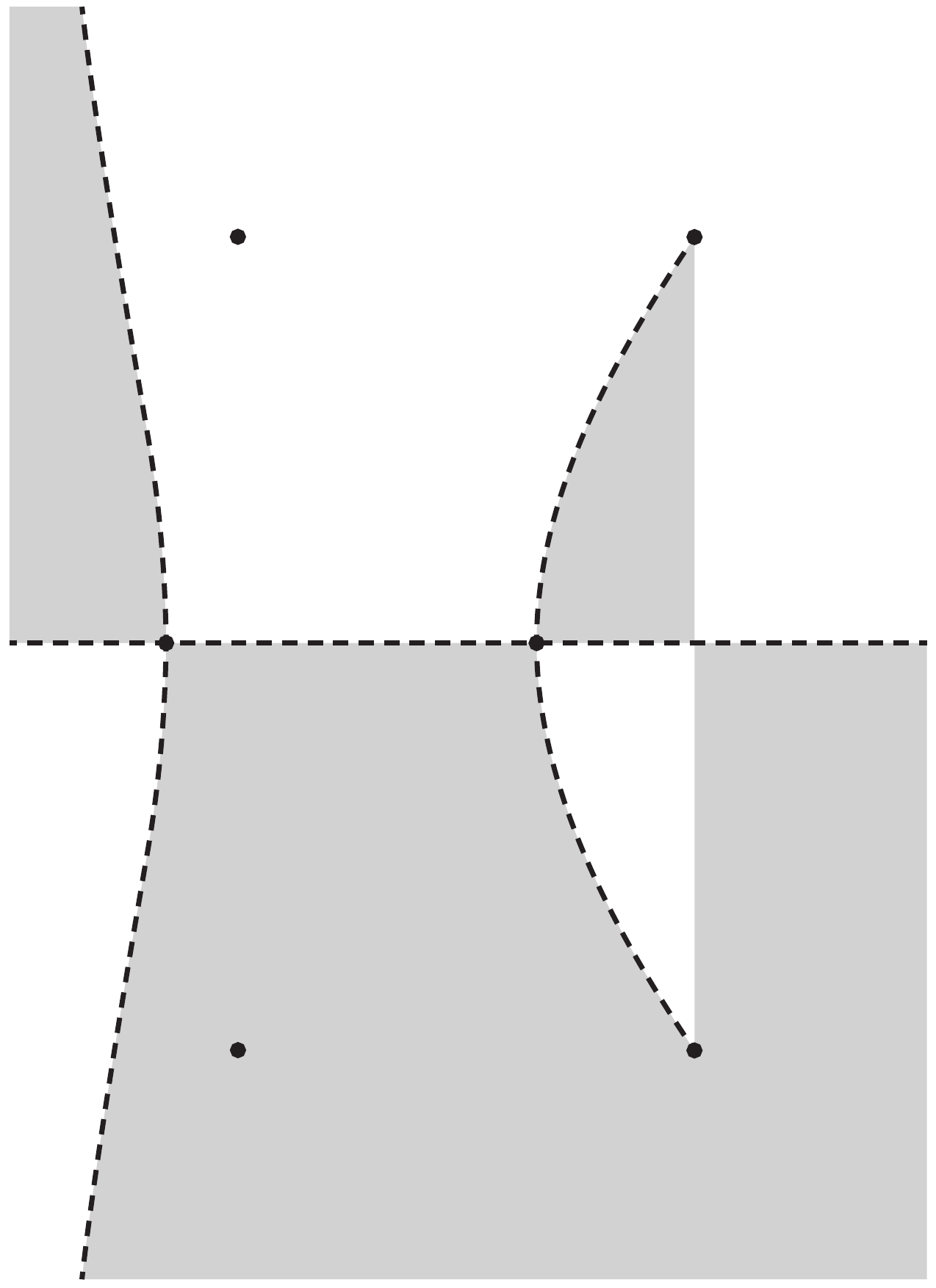}
       \put(30,90){\small $\Im g>0$}
       \put(30,8){\small $\Im g<0$}
     \put(11,80){\small $E_1$}
      \put(11,17){\small $\bar{E}_1$}
      \put(55.5,80){\small $E_2$}
      \put(55.5,17){\small $\bar{E}_2$}
      \put(14.5,52.5){\small $\mu_1$}
      \put(34.5,52.5){\small $\mu_2$}
\end{overpic}}
\hspace{.5cm}
 \subcaptionbox{$\xi = \xi_{E_1} = \xi_\merge$}{
\begin{overpic}[width=.28\textwidth]{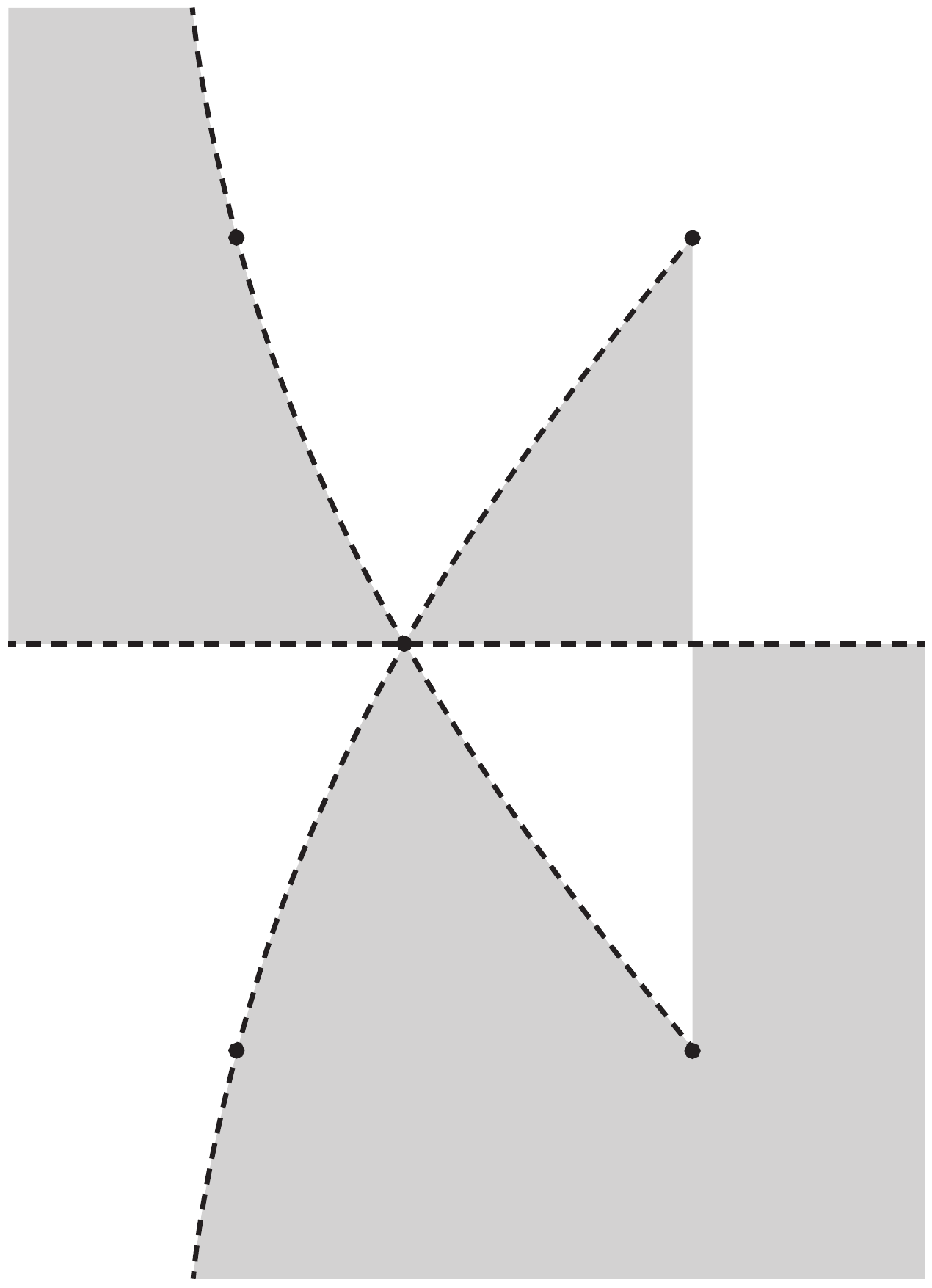}
       \put(30,90){\small $\Im g>0$}
       \put(30,8){\small $\Im g<0$}
     \put(11,80){\small $E_1$}
      \put(11,17){\small $\bar{E}_1$}
      \put(55.5,80){\small $E_2$}
      \put(55.5,17){\small $\bar{E}_2$}
      \put(34.5,52.5){\small $\mu_1 = \mu_2$}
\end{overpic}}
	\\ \vspace{.2cm}
 \subcaptionbox{$0 < \xi < \xi_{E_1}$}{
\begin{overpic}[width=.28\textwidth]{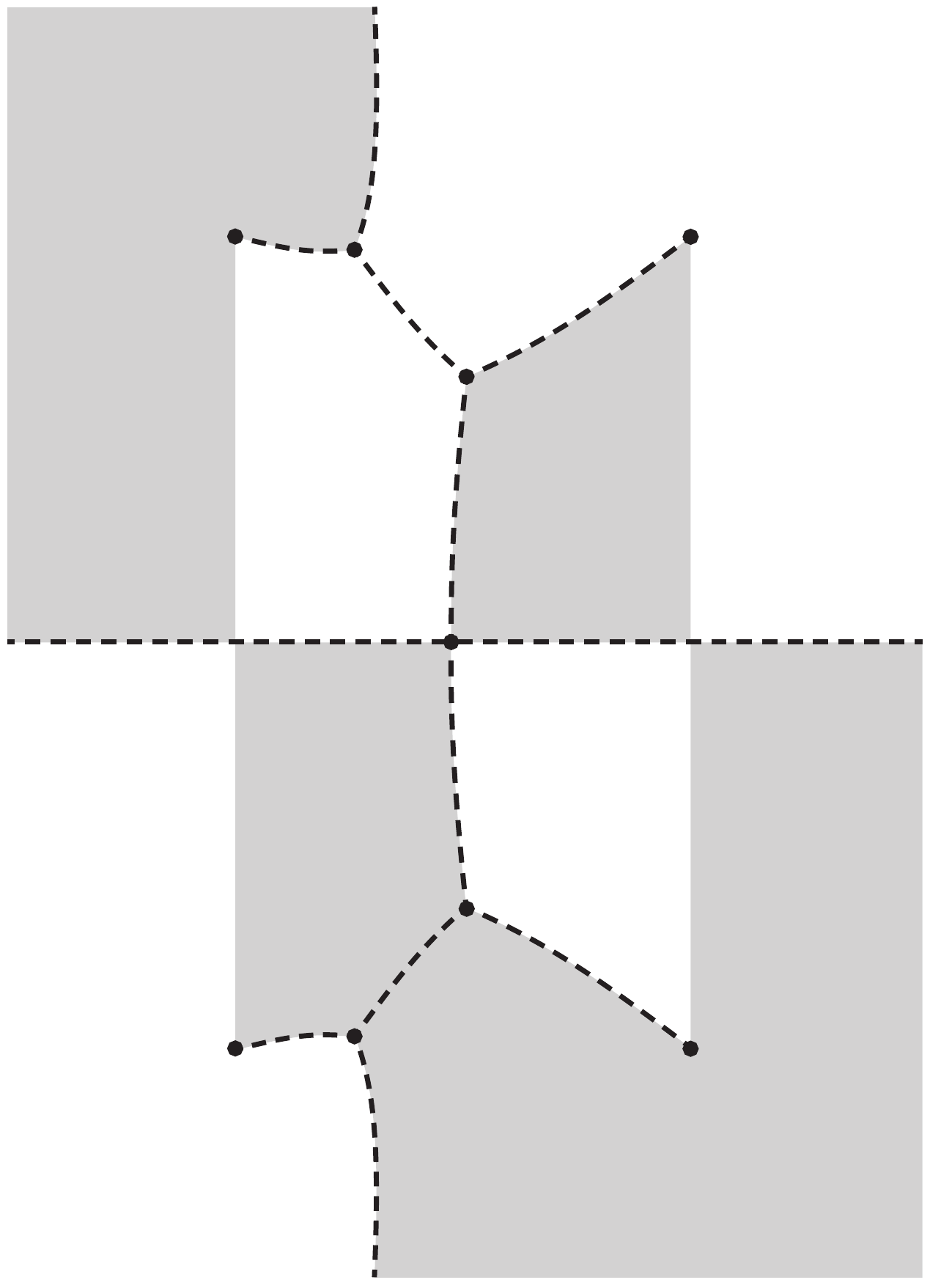}
      \put(40,90){\small $\Im g>0$}
       \put(40,8){\small $\Im g<0$}
     \put(11,80){\small $E_1$}
      \put(11,17){\small $\bar{E}_1$}
      \put(55.5,80){\small $E_2$}
      \put(55.5,17){\small $\bar{E}_2$}
      \put(36.5,52.5){\small $\mu$}
   \put(29.5,80){\small $\alpha$}
      \put(29.5,18){\small $\bar{\alpha}$}
   \put(37,67){\small $\beta$}
      \put(37,31){\small $\bar{\beta}$}
 \end{overpic}}
\hspace{.5cm}
 \subcaptionbox{$\xi = 0$}{
\begin{overpic}[width=.28\textwidth]{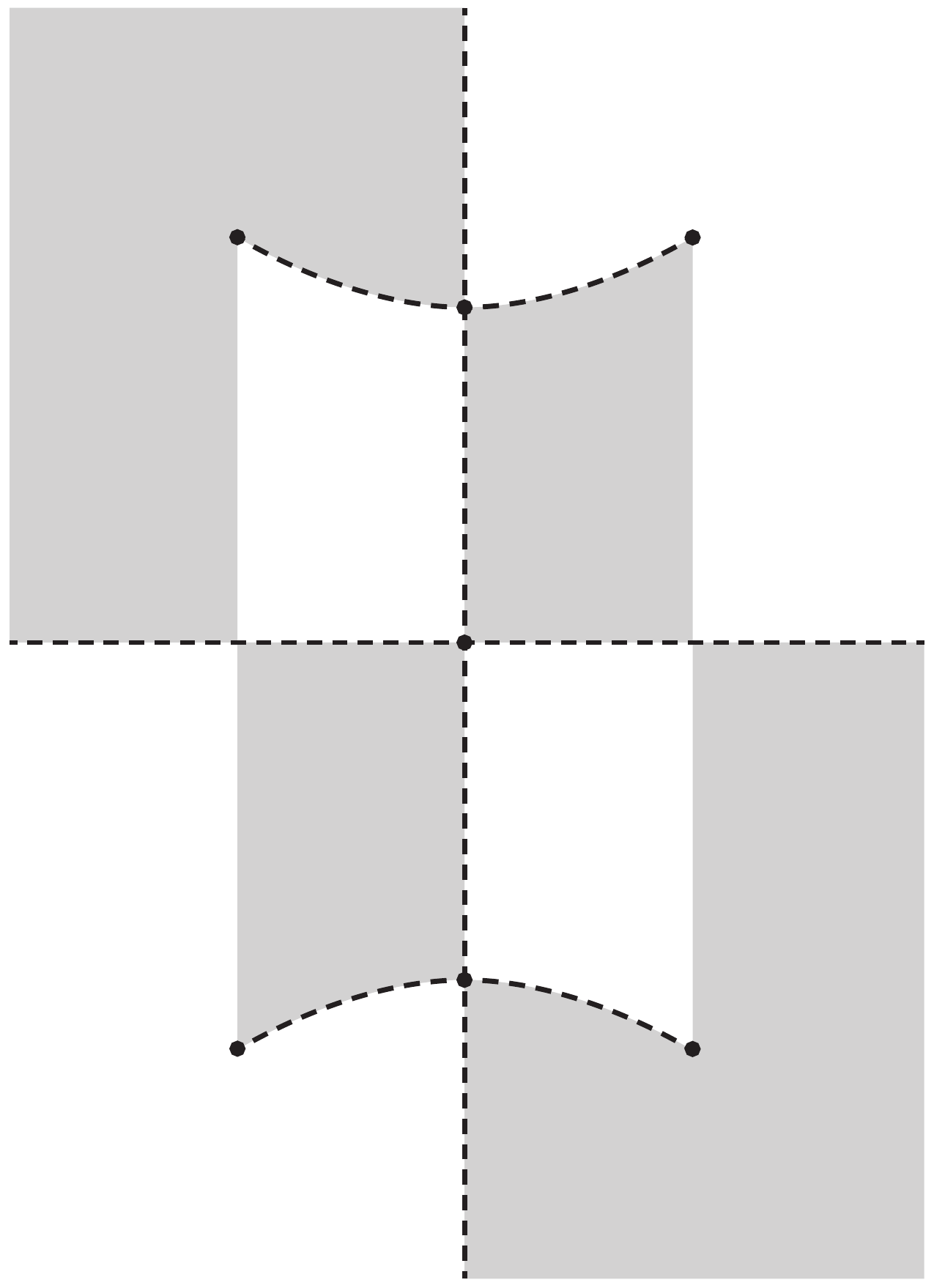}
      \put(45,90){\small $\Im g>0$}
       \put(45,8){\small $\Im g<0$}
     \put(11,80){\small $E_1$}
      \put(11,17){\small $\bar{E}_1$}
      \put(55.5,80){\small $E_2$}
      \put(55.5,17){\small $\bar{E}_2$}
      \put(37.5,52.5){\small $\mu$}
         \put(37,72){\small $\alpha = \beta$}
      \put(37,25){\small $\bar{\alpha}= \bar{\beta}$}
\end{overpic}}
\caption{Signature tables of $\Im g(\xi, k)$ corresponding to the four columns of Table \ref{4thscenariotable} of the 4th scenario. Each figure shows the zero level set $\Im g=0$ (dashed) and the regions where $\Im g<0$ (shaded) and $\Im g>0$ (white) in the complex $k$-plane for $\xi$ as indicated.}
\label{fig:4thscenarioimg}
\end{figure}
%---------------------------------------%

%---------------------------------------%
%:fig 5.13
%---------------------------------------%
\begin{figure}[ht]
\centering\includegraphics[scale=.55]{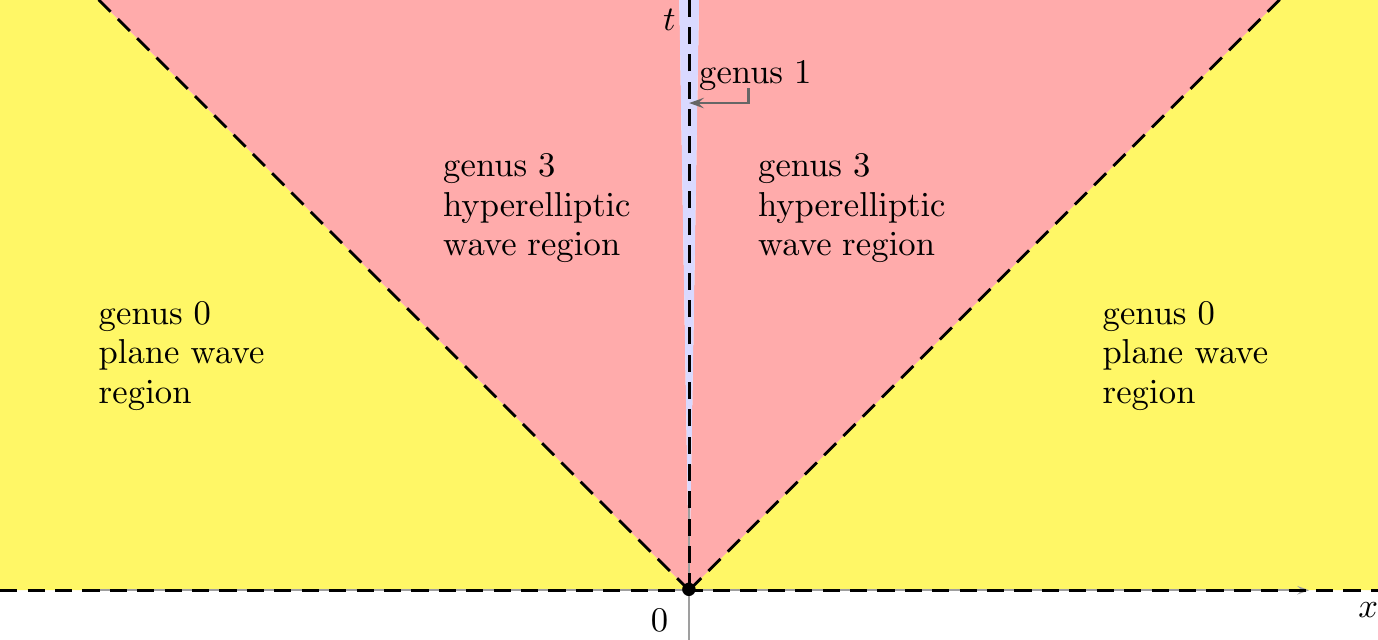}
\caption{4th scenario (symmetric shock case): $\frac{A}{B}=\frac{2}{7}(2+3\sqrt{2})$} 
\label{fig:shock-scenario-4}
\end{figure}
%---------------------------------------%

%---------------------------------------------------------%
%:s.5.3.5
%---------------------------------------------------------%
\subsubsection{5th Scenario}  \label{sec:scenario-5}

%-------------------%
%:table 5.5
%-------------------%
\begin{table}[ht]
\begin{tabular}{|c|c|c|c|c|c|}
\hline
$\xi=0$&$0<\xi<\xi_{E_1}^{\new}$&$\xi=\xi_{E_1}^{\new}$&$\xi_{E_1}^{\new}<\xi<\xi_{\merge}$&$\xi=\xi_{\merge}$&$\xi>\xi_{\merge}$\\
\hline
genus $1$&genus $3$&&genus $1$&&genus $0$\\ 
\hline
$\alpha$, $\beta$ merge&&the infinite branch&&the real zeros&\\ 
&&hits $E_1$, $\bar E_1$&&$\mu_1$, $\mu_2$ merge&\\ 
\hline
\multicolumn{6}{c}{}\\
\end{tabular}
\caption{5th scenario: $\frac{A}{B}>\frac{2}{7}(2+3\sqrt{2})$.}\label{5thscenariotable}
\end{table}
%---------------------------------------%

%---------------------------------------%
%:fig 5.14
%---------------------------------------%
\begin{figure}[ht]
 \subcaptionbox{$\xi > \xi_{\merge}$}{
 \begin{overpic}[width=.28\textwidth]{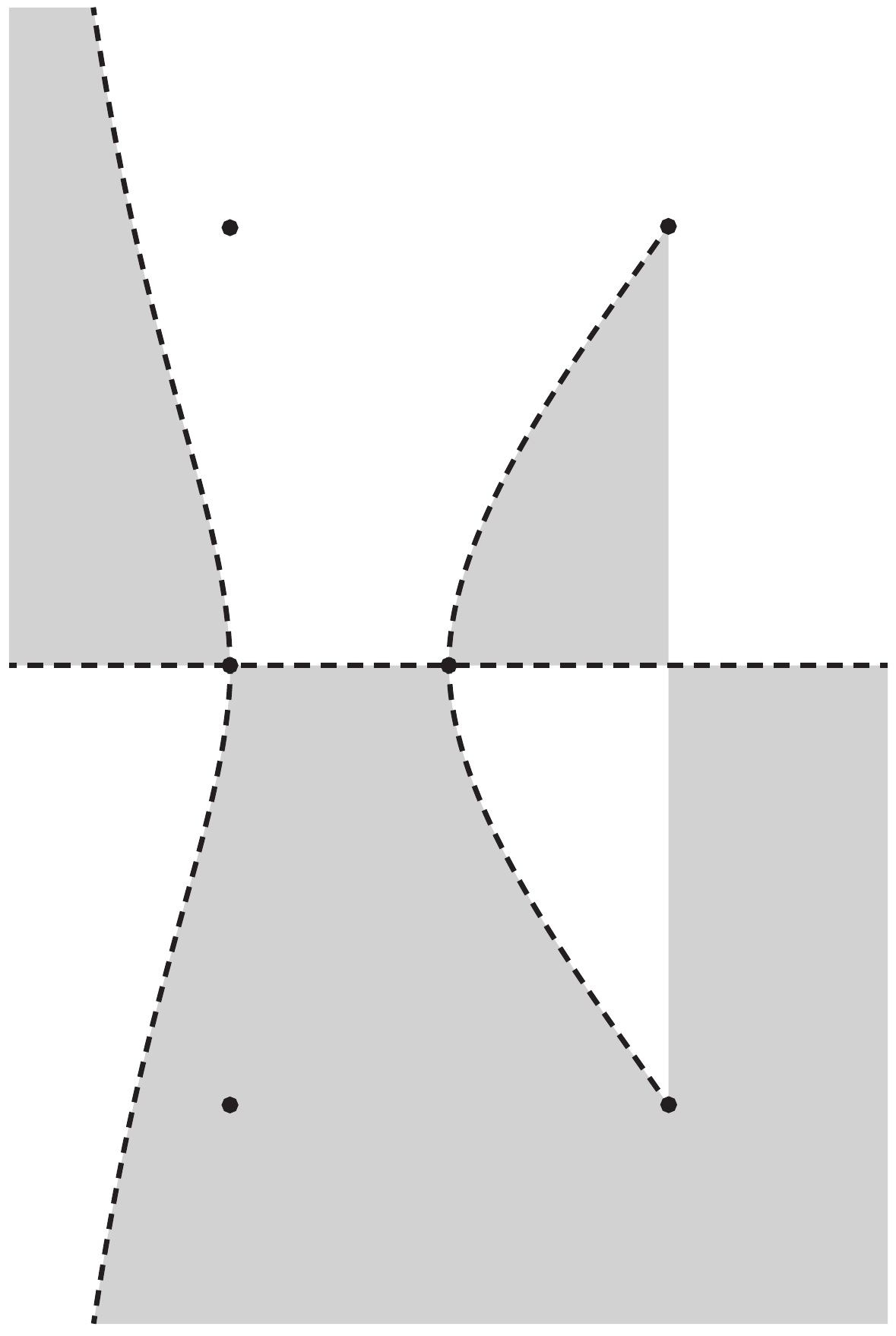}
       \put(30,90){\small $\Im g>0$}
       \put(30,7){\small $\Im g<0$}
     \put(18.5,81.5){\small $E_1$}
      \put(18.5,15.5){\small $\bar{E}_1$}
      \put(51.5,81.5){\small $E_2$}
      \put(51.5,15.5){\small $\bar{E}_2$}
      \put(11,52.5){\small $\mu_1$}
      \put(35,52.5){\small $\mu_2$}
\end{overpic}}
\hspace{.5cm}
 \subcaptionbox{$\xi = \xi_{\merge}$}{
\begin{overpic}[width=.28\textwidth]{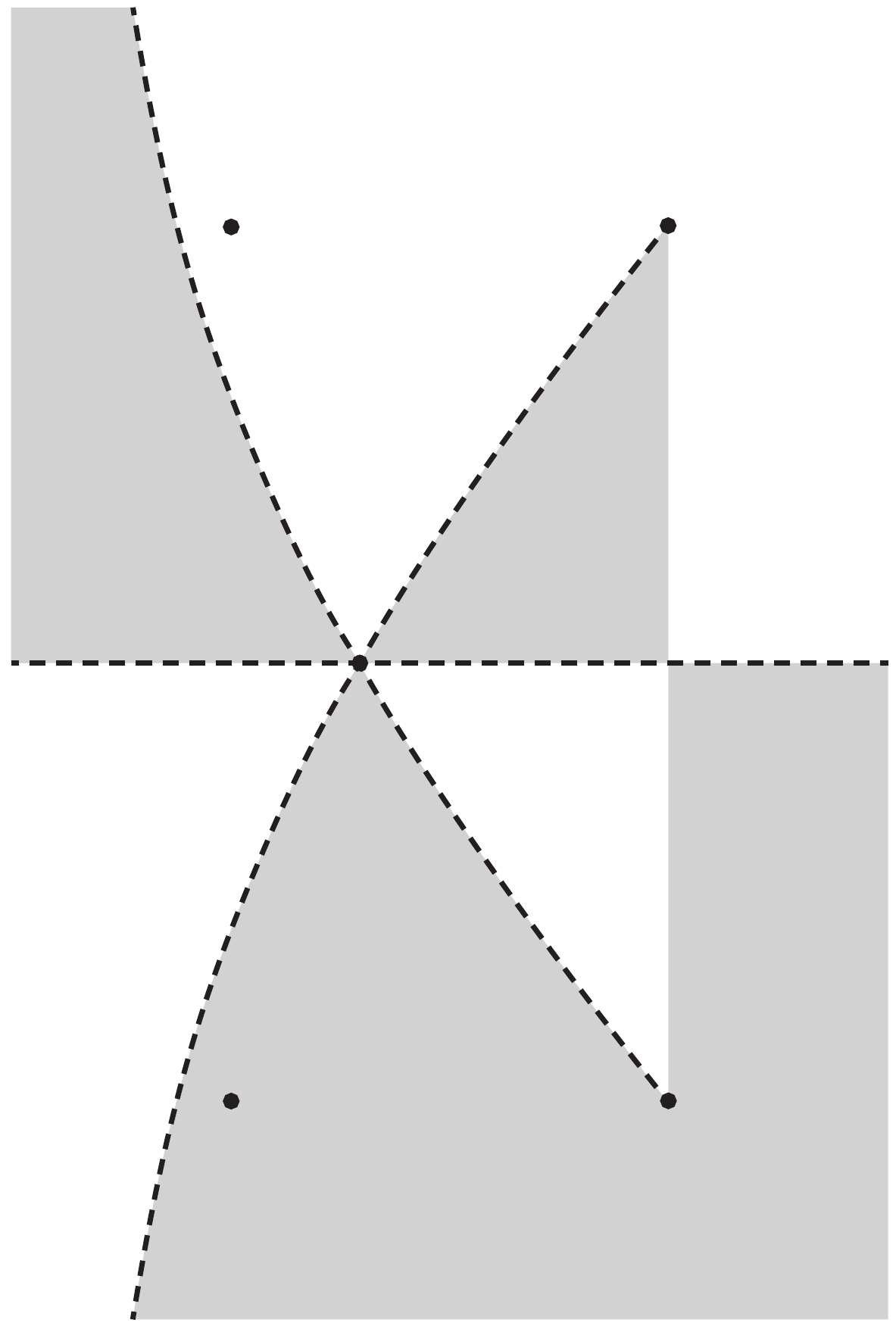}
      \put(30,90){\small $\Im g>0$}
       \put(30,7){\small $\Im g<0$}
     \put(18.5,81.5){\small $E_1$}
      \put(18.5,15.5){\small $\bar{E}_1$}
      \put(51.5,81.5){\small $E_2$}
      \put(51.5,15.5){\small $\bar{E}_2$}
      \put(30.5,52.5){\small $\mu_1=\mu_2$}
\end{overpic}}
\hspace{.5cm}
 \subcaptionbox{$\xi_{E_1}^\new < \xi < \xi_{\merge}$}{
\begin{overpic}[width=.28\textwidth]{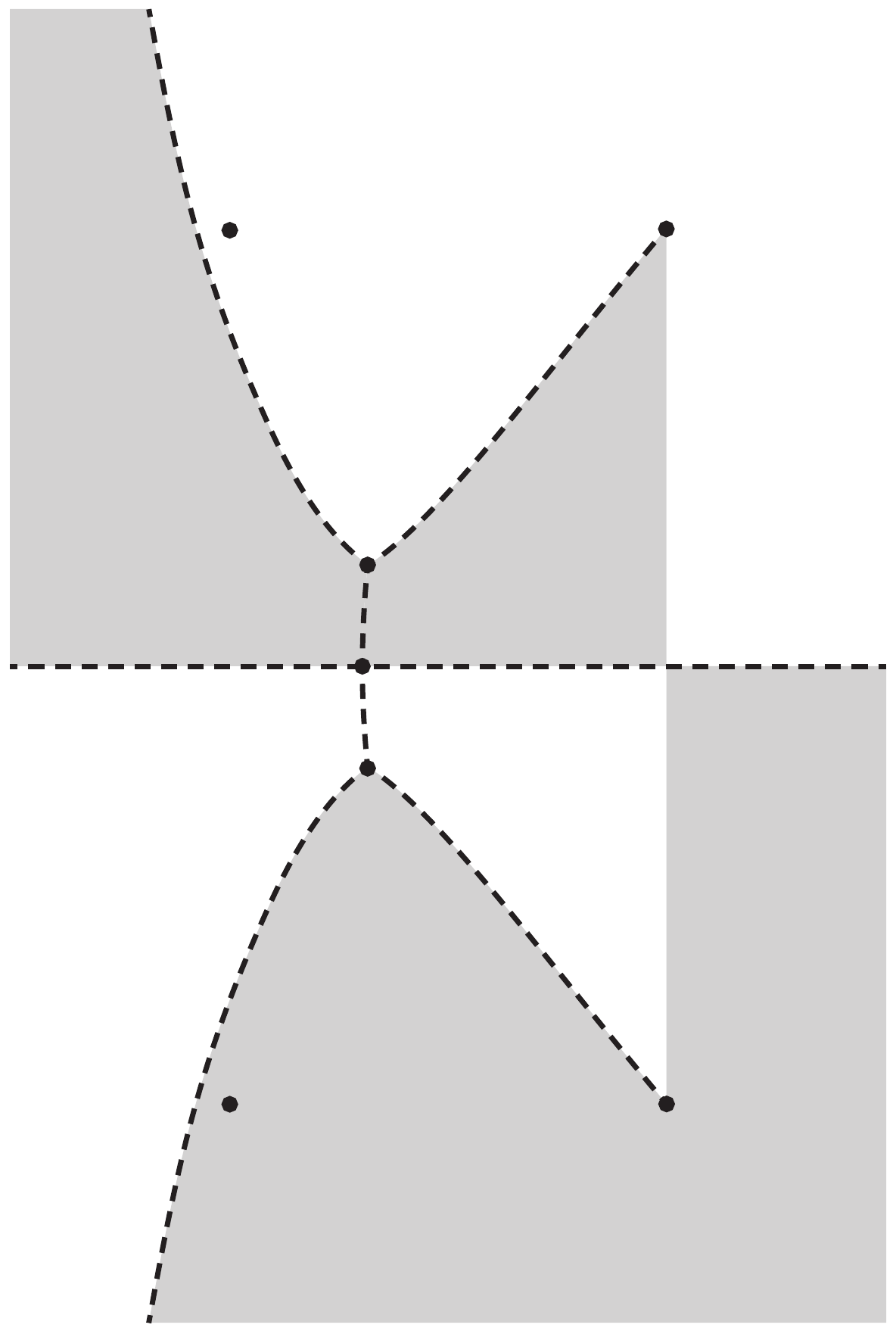}
       \put(30,90){\small $\Im g>0$}
       \put(30,7){\small $\Im g<0$}
     \put(18.5,81.5){\small $E_1$}
      \put(18.5,15.5){\small $\bar{E}_1$}
      \put(51.5,81.5){\small $E_2$}
      \put(51.5,15.5){\small $\bar{E}_2$}
      \put(28.5,52){\small $\mu$}
  \put(26.5,60.5){\small $\beta$}
      \put(26,36){\small $\bar{\beta}$}
 \end{overpic}}
	\\ \vspace{.2cm}
 \subcaptionbox{$\xi = \xi_{E_1}^\new$}{
\begin{overpic}[width=.28\textwidth]{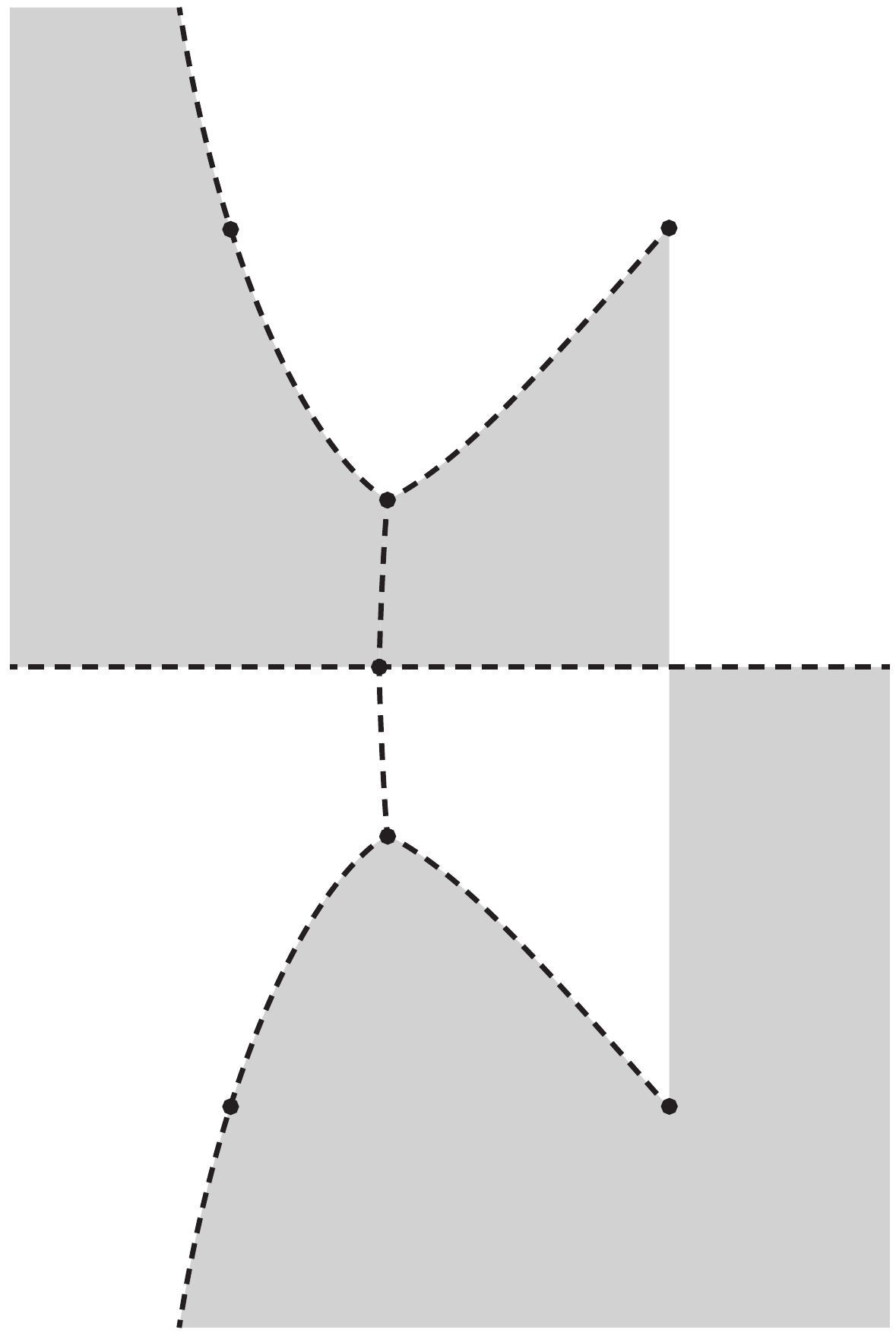}
       \put(30,90){\small $\Im g>0$}
       \put(30,7){\small $\Im g<0$}
     \put(10,81.5){\small $E_1$}
      \put(10,15.5){\small $\bar{E}_1$}
      \put(51.5,81.5){\small $E_2$}
      \put(51.5,15.5){\small $\bar{E}_2$}
      \put(30,52){\small $\mu$}
  \put(28,65.5){\small $\beta$}
      \put(27,31){\small $\bar{\beta}$}
\end{overpic}}
\hspace{.5cm}
 \subcaptionbox{$0 < \xi < \xi_{E_1}^\new$}{
\begin{overpic}[width=.28\textwidth]{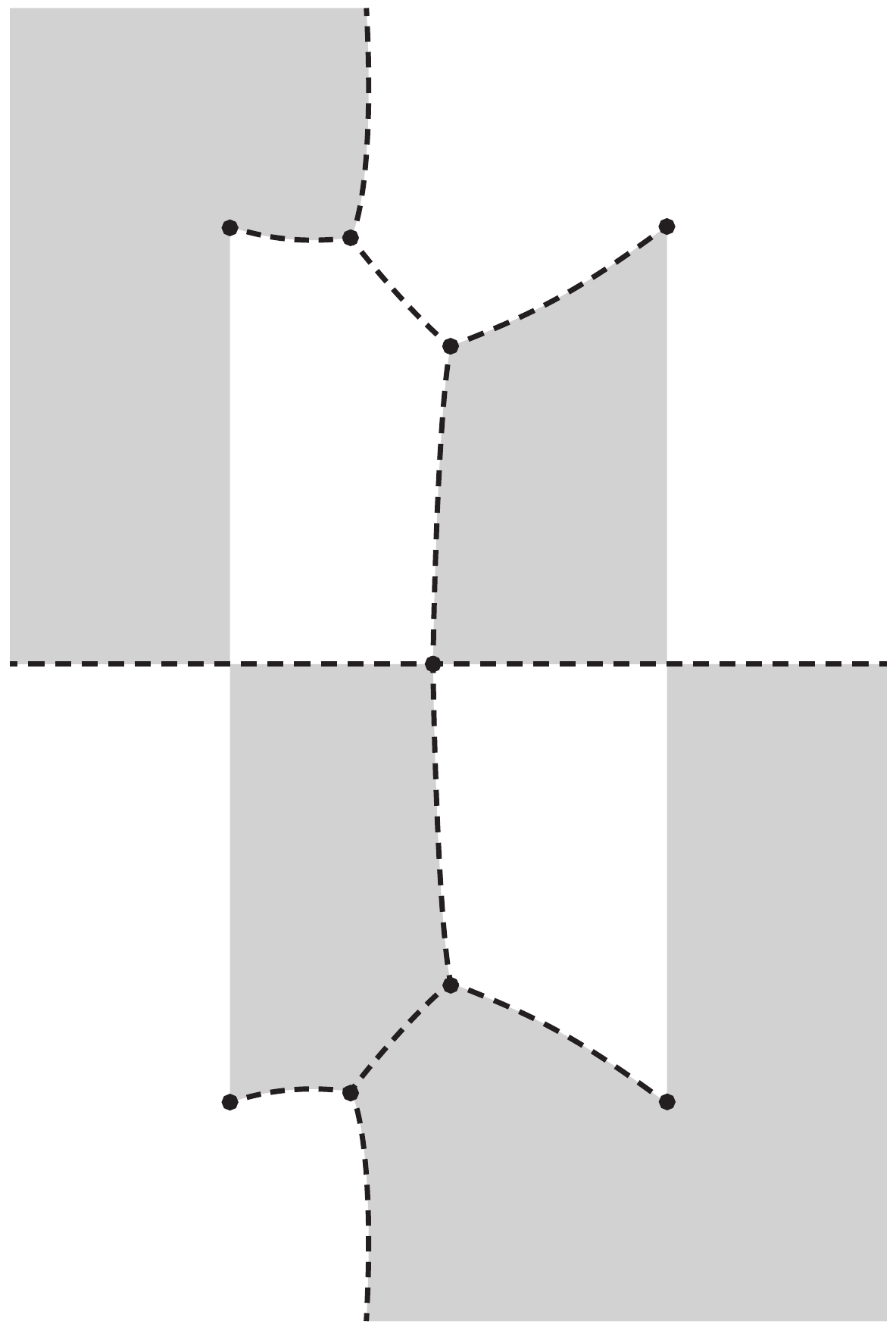}
      \put(40,90){\small $\Im g>0$}
       \put(40,7){\small $\Im g<0$}
     \put(10,81.5){\small $E_1$}
      \put(10,15.5){\small $\bar{E}_1$}
      \put(51.5,81.5){\small $E_2$}
      \put(51.5,15.5){\small $\bar{E}_2$}
      \put(34,52){\small $\mu$}
   \put(28,82){\small $\alpha$}
      \put(28,16){\small $\bar{\alpha}$}
   \put(35,70){\small $\beta$}
      \put(35,27){\small $\bar{\beta}$}
\end{overpic}}
\hspace{.5cm}
 \subcaptionbox{$\xi = 0$}{
\begin{overpic}[width=.28\textwidth]{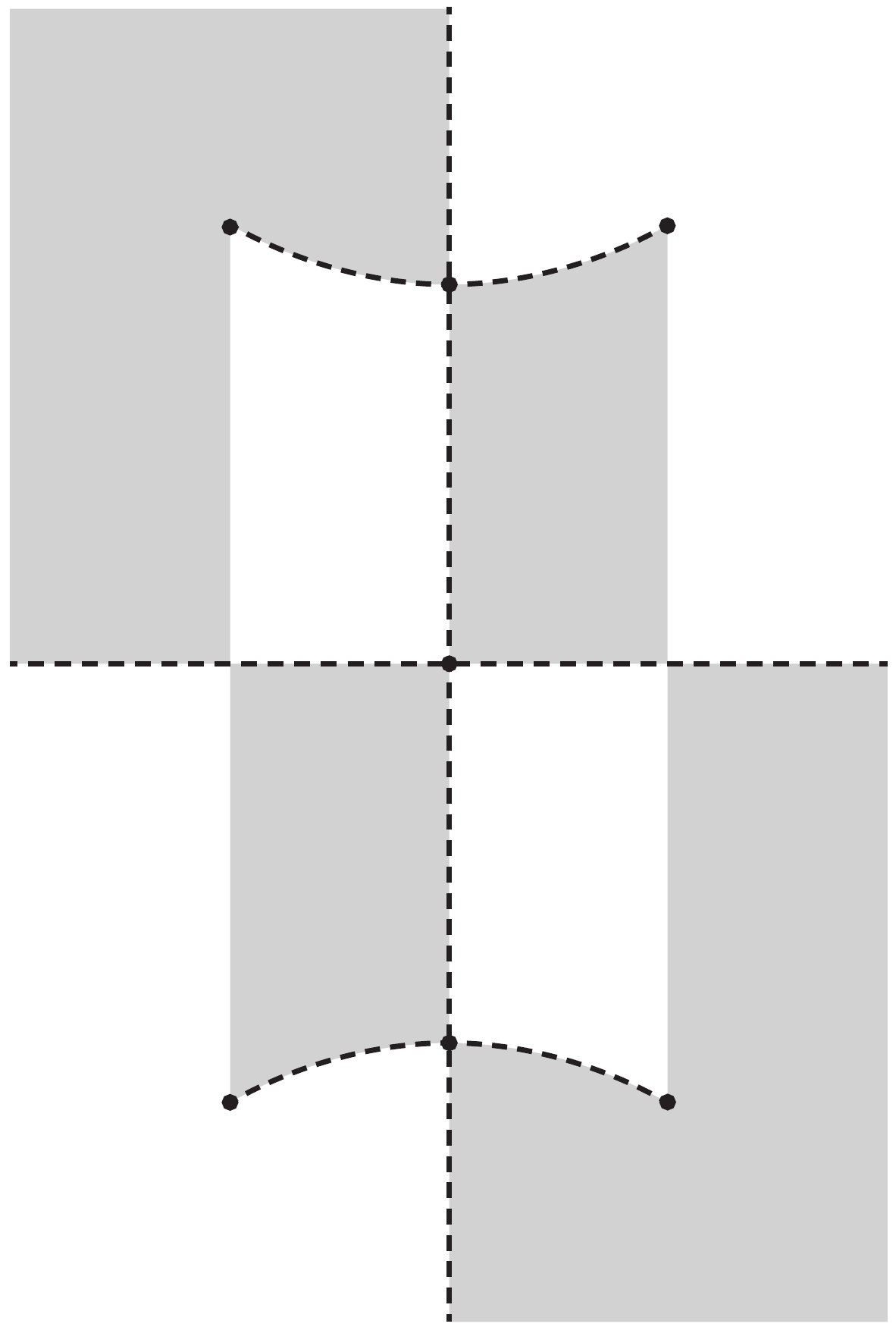}
       \put(42,90){\small $\Im g>0$}
       \put(42,7){\small $\Im g<0$}
     \put(10,81.5){\small $E_1$}
      \put(10,15.5){\small $\bar{E}_1$}
      \put(51.5,81.5){\small $E_2$}
      \put(51.5,15.5){\small $\bar{E}_2$}
      \put(35,52.5){\small $\mu$}
   \put(35,75){\small $\alpha = \beta$}
      \put(35,23){\small $\bar{\alpha}= \bar{\beta}$}
\end{overpic}}
\caption{Signature tables of $\Im g(\xi,k)$ corresponding to the six columns of Table \ref{5thscenariotable} of the 5th scenario. Each figure shows the zero level set $\Im g=0$ (dashed) and the regions where $\Im g<0$ (shaded) and $\Im g>0$ (white) in the complex $k$-plane for $\xi$ as indicated.}
\label{fig:5thscenarioimg}
\end{figure}
%---------------------------------------%

%---------------------------------------%
%:fig 5.15
%---------------------------------------%
\begin{figure}[ht]
\centering\includegraphics[scale=.55]{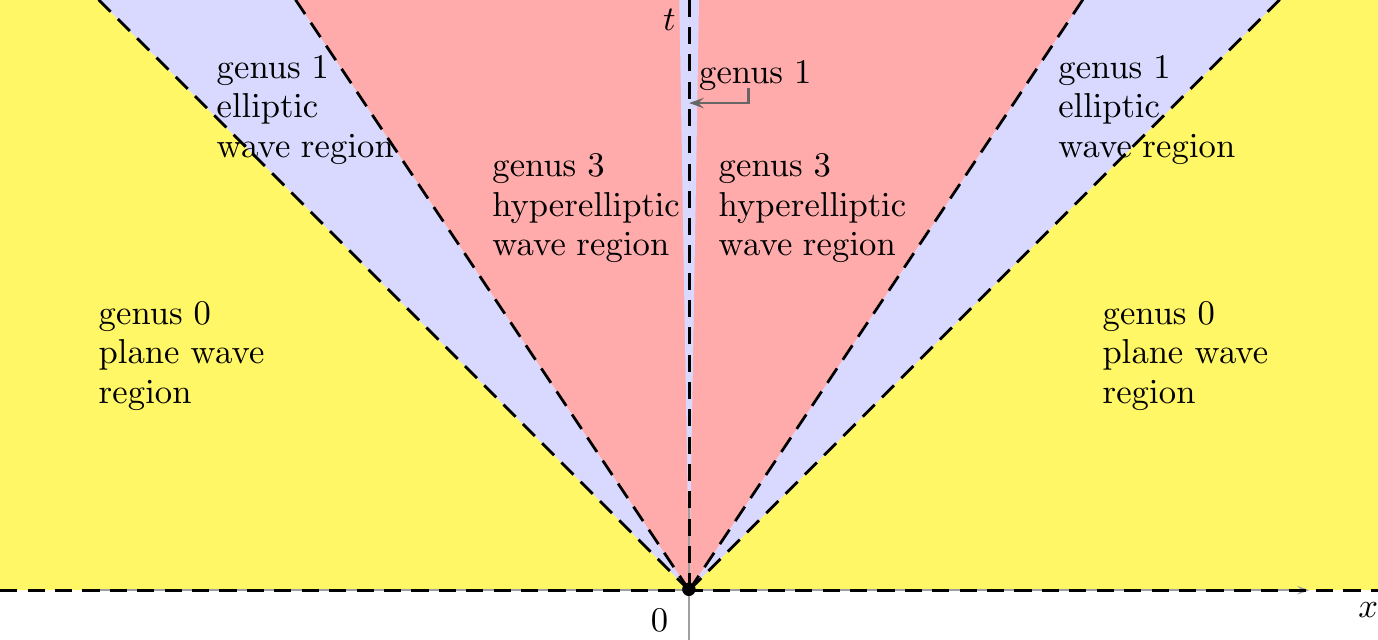}
\caption{5th scenario (symmetric shock case): $\frac{A}{B}>\frac{2}{7}(2+3\sqrt{2})$} 
\label{fig:shock-scenario-5}
\end{figure}
%---------------------------------------%

We are in Case 2. As $\xi$ goes down from $+\infty$, the $g$-function $g_2$ is appropriate until the two real zeros $\mu_1$ and $\mu_2$ of $g_2'$ (see \eqref{g2}) merge, that is, as long as $\xi>\xi_{\merge}$. Then, a new $g$-function $g\equiv g_2^{\new}$ is required whose derivative $g'$ has the same form as in the rarefaction case:
\begin{equation} \label{g-sce5}
g'(\xi,k)=4\frac{(k-\mu(\xi))(k-\beta(\xi))(k-\bar\beta(\xi))}{\sqrt{(k-E_2)(k-\bar E_2)(k-\beta(\xi))(k-\bar\beta(\xi))}},
\end{equation}
and thus the asymptotics is given in terms of elliptic functions, as in \cite{BKS11}. This new $g$-function remains appropriate until the infinite branch of $\Im g_2^{\new}=0$ hits $E_1$ and $\bar E_1$, which happens for $\xi=\xi_{E_1}^{\new}$. Finally, for $0<\xi<\xi_{E_1}^{\new}$, a third $g$-function is to be considered with derivative of the form \eqref{g-xi}:
\begin{equation}  \label{g-3}
g'(\xi,k)=4\frac{(k-\mu(\xi))(k-\alpha(\xi))(k-\bar\alpha(\xi))(k-\beta(\xi))(k-\bar\beta(\xi))}{w(\xi,k)}\,,
\end{equation}
where
\[
w^2=(k-E_1)(k-\bar E_1)(k-E_2)(k-\bar E_2)(k-\alpha(\xi))(k-\bar\alpha(\xi))(k-\beta(\xi))(k-\bar\beta(\xi)),
\]
and where $\alpha(\xi)$ emerges from $E_1$ at $\xi=\xi_{E_1}^{\new}$. As above, the  parameters $\mu(\xi)$, $\alpha(\xi)$, and $\beta(\xi)$ of this genus~$3$ sector are determined by the system of equations \eqref{dg-conditions}. The left end of the range characterized by \eqref{g-3} is $\xi=0$. As $\xi\to 0$, $\alpha(\xi)$ and $\beta(\xi)$ both approach a single point $\ii\alpha_0$ with $\alpha_0=\sqrt{A^2-B^2}$ whereas $\mu(\xi)\to 0$. At $\xi=0$ the $g$-function takes the genus $1$ form \eqref{g-0}:
\[
g'(0,k)=4\frac{k(k^2+\alpha_0^2)}{\sqrt{(k-E_1)(k-\bar E_1)(k-E_2)(k-\bar E_2)}}.
\]
%---------------------------------------------------------%
%:s.6 
%---------------------------------------------------------%
\section{Existence of a genus 2 sector}\label{implicitfunctiontheorem}

The first three scenarios in the symmetric shock case presented in the previous section include genus $2$ sectors. We arrived upon these sectors by studying the dependence of the $g$-function on $\xi$, and their existence is clearly confirmed by numerical computations. However, to actually prove that these sectors exist, it is necessary to show that the system of equations characterizing the parameters (see \eqref{dg-conditions-2}) has a solution. In this section, we show that these genus $2$ sectors actually exist by establishing solvability of this system. Even though we restrict attention to these particular sectors for definiteness, it seems clear that our approach can be used to show existence also of other similar higher-genus sectors. A key point in the approach is the introduction of an appropriate local diffeomorphism (see \eqref{varphidef}) which makes it possible to apply the implicit function theorem. 

Our approach can be compared with an approach of \cite{TV2010}, where a determinantal formula for the $g$-function is exploited to prove a similar result, and the approach developed in \cite{KMM2003}*{Section 7.2}, where a normal form method is used to show existence in a different way.

%---------------------------------------------------------%
%:s.6.1 
%---------------------------------------------------------%
\subsection{Genus 2 Riemann surface and associated $\BS{g}$-function}

We consider the Cauchy problem for NLS defined by \eqref{nlsic} and \eqref{qnot} for parameters satisfying \eqref{symmetriccase} and $\frac{A}{B}<\frac{2}{7}(2+3\sqrt{2})$. In particular, we have $E_1=-B+\ii A$ and $E_2=B+\ii A$ with $B>0$ and $A>0$. These assumptions correspond to the first three scenarios of the symmetric shock case. 

Let $\Sigma_\alpha$ be the genus $2$ hyperelliptic Riemann surface with branch points at $E_1$, $\bar{E}_1$, $E_2$, $\bar{E}_2$, $\alpha$, $\bar{\alpha}$ for some nonreal complex number $\alpha$ with $\Im\alpha>0$. Let $\C{C}\subset\D{C}$ be the union of the cuts $\croch{E_1,\bar E_1}$, $\croch{E_2,\bar E_2}$, and $\croch{\alpha,\bar\alpha}$ (see Figure~\ref{fig:genus2ab}):
\[
\C{C}\coloneqq\Sigma_1\cup\Sigma_2\cup\croch{\alpha,\bar\alpha}.
\]
Define the meromorphic differential $\dd g$ on $\Sigma_\alpha$ as follows:
\[
\dd g(k)\coloneqq\frac{4(k-\mu_1)(k-\mu_2)(k-\alpha)(k-\bar\alpha)}{w(k)}\dd k,
\]
where $\mu_1,\mu_2\in\D{R}$, $\mu_1<\mu_2$, and
\[
w(k)\coloneqq\sqrt{(k-E_1)(k-\bar E_1)(k-E_2)(k-\bar E_2)(k-\alpha)(k-\bar\alpha)}.
\]
We view $\Sigma_\alpha$ as a two-sheeted cover of the complex plane such that $w(k^+)\sim k^3$ as $k\to\infty$, where $k^\pm$ denote the points on the upper and lower sheets which project onto $k$.

The definition of $\dd g$ depends on the four real numbers $\mu_1$, $\mu_2$, $\alpha_1$, $\alpha_2$, where $\alpha_1$ and $\alpha_2$ denote the real and imaginary parts of $\alpha$:
\[
\alpha=\alpha_1+\ii\alpha_2,\quad\alpha_2>0.
\]
These four real numbers are determined by the four conditions 
\begin{subequations}   \label{dg-conditions-2}
\begin{align}  \label{dg-first-conditions-2}
&\int_{a_1}\dd g=\int_{a_2}\dd g=0,\\
\label{dg-last-conditions-2}
&\lim_{k\to\infty}\left(\frac{\dd g}{\dd k}-4k\right)=\xi,\quad       
\lim_{k\to\infty}k\left(\frac{\dd g}{\dd k}-4k-\xi\right)=0,
\end{align}
\end{subequations}
where we let $a_j$, $j=1,2$, be a counterclockwise loop on the upper sheet enclosing $[\bar{E}_j, E_j]$ and no other branch points, see Figure~\ref{fig:genus2ab}. We let $\zeta=(\zeta_1,\zeta_2)$ be the normalized basis of $H^1(\Sigma_\alpha)$ which is dual to the canonical homology basis $\accol{a_j,b_j}_1^2$ in the sense that $\accol{\zeta_j}_1^2$ are holomorphic differentials such that
\[
\int_{a_i}\zeta_j=\delta_{ij},\quad i,j=1,2.
\]
The basis $\accol{\zeta_j}_1^2$ is explicitly given by $\zeta_j=\sum_{l=1}^2\C{A}_{jl}\hat\zeta_l$, where
\begin{equation}\label{hatzetadef}
\hat\zeta_l=\frac{k^{l-1}}{w}\dd k
\end{equation}
and the invertible matrix $\C{A}$ is given by
\begin{equation}\label{calAdef}
\left(\C{A}^{-1}\right)_{jl}=\int_{a_j}\hat\zeta_l.
\end{equation}
Note that $\C{A}$, $\zeta$, and $\hat{\zeta}$ depend on $\alpha$.
%---------------------------------------%
%:fig 6.1
%---------------------------------------%
\begin{figure}[ht]
\begin{center}
\begin{overpic}[width=.5\textwidth]{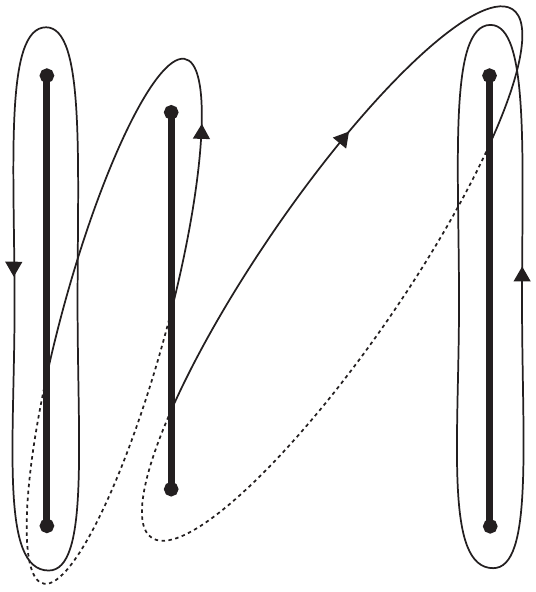}
      \put(-4,54){\small $a_1$}
      \put(91,53){\small $a_2$}
      \put(36,76){\small $b_1$}
      \put(60,74){\small $b_2$}
      \put(5.5,90){\small $E_1$}
      \put(5.5,5.5){\small $\bar{E}_1$}
     \put(27.5,83.5){\small $\alpha$}
      \put(27.5,12.5){\small $\bar{\alpha}$}
      \put(80,90){\small $E_2$}
      \put(80,5.5){\small $\bar{E}_2$}
\end{overpic}
\end{center}
\caption{The homology basis $\accol{a_j,b_j}_1^2$ on the genus $2$ Riemann surface $\Sigma_\alpha$.} 
\label{fig:genus2ab}
\end{figure}
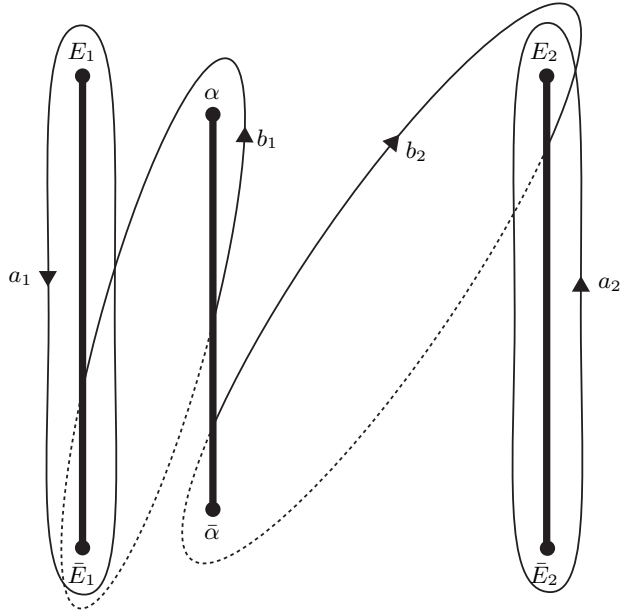
%---------------------------------------%

The conditions in \eqref{dg-last-conditions-2} can be formulated as
\begin{align}\label{dgdkasymptotics}
\frac{\dd g}{\dd k}(k^+) = 4k +\xi +\ord(k^{-2}),\qquad k \to\infty.
\end{align}
The solvability of the system of equations \eqref{dg-conditions-2} characterizes the genus $2$ sector. Since
\begin{equation}  \label{dg-symmetry}
\frac{\dd g}{\dd k}(k)=\overline{\frac{\dd g}{\dd k}(\bar k)},\quad k\in\D{C}\setminus\C{C},
\end{equation}
we have $\int_A^B\dd g=\overline{\int_{\bar A}^{\bar B}\dd g}$ where the contour in the second integral is the complex conjugate of the contour in the first integral. This implies that
\begin{equation} \label{intdg}
\int_{a_j}\dd g \in\ii\,\D{R},\quad j=1,2,
\end{equation}
so the conditions in \eqref{dg-first-conditions-2} are two real conditions. 

As $\xi$ decreases from $+\infty$, the infinite branch hits $E_1$ and $\bar E_1$ when $\xi =\xi_{E_1}$, where
\begin{align}\label{xiE1def}
\xi_{E_1}=2(B+|E_1|).
\end{align}
For $\xi>\xi_{E_1}$, we are in the genus $0$ sector and the $g$-function is given by (see \eqref{g2bis})
\[
\dd g =\frac{4(k-\mu_1)(k-\mu_2)}{\sqrt{(k-E_2)(k-\bar{E}_2)}}\dd k,
\]
where $\mu_1<\mu_2$ are given by \eqref{mu12bis}. For $\xi=\xi_{E_1}$, we have
\begin{equation}\label{mu1mu2atxiE1}
\begin{split}
\mu_1(\xi_{E_1}) =\frac{B-|E_1| - \sqrt{2 B \left(3 |E_1|+5 B\right)-7 A^2}}{4},\\  
\mu_2(\xi_{E_1}) =\frac{B-|E_1| +\sqrt{2 B \left(3 |E_1|+5 B\right)-7 A^2}}{4}.
\end{split}
\end{equation}

As $\xi$ decreases below $\xi_{E_1}$, we expect to see a genus $2$ sector. We will show that the system \eqref{dg-conditions-2} indeed has a unique solution for $\xi\in (\xi_{E_1}-\delta,\xi_{E_1})$ for some $\delta>0$ and that this solution can be extended until the qualitative structure of the $g$-function changes (see item (f) below). 

%-------------------%
%:thm 6.1
%-------------------%
\begin{theorem}[Existence of genus $2$ sector]\label{genus2existenceth}
Suppose $0<\frac{A}{B}<\frac{2}{7}(2+3\sqrt{2})$. Then there exists a $\xi_m < \xi_{E_1}$ and a smooth curve 
\[
\xi\mapsto(\alpha_1(\xi),\alpha_2(\xi),\mu_1(\xi),\mu_2(\xi))\in\D{R}^4
\]
defined for $\xi\in(\xi_m,\xi_{E_1})$ such that the following hold:
\begin{enumerate}[\rm(a)]
\item 
For each $\xi\in(\xi_m,\xi_{E_1})$, $(\xi,\alpha_1(\xi),\alpha_2(\xi),\mu_1(\xi),\mu_2(\xi))$ is a solution of the system of equations \eqref{dg-conditions-2}.
\item 
The curve $\xi\mapsto(\mu_1(\xi),\mu_2(\xi))$ is a smooth map $(\xi_m,\xi_{E_1})\to\D{R}^2$ such that
\[
\mu_1(\xi)<\mu_2(\xi)\quad\text{for}\quad\xi\in(\xi_m,\xi_{E_1}).
\]
\item 
The curve $\xi\mapsto\alpha(\xi)=\alpha_1(\xi)+\ii\alpha_2(\xi)$ is a smooth map $(\xi_m,\xi_{E_1})\to\D{C}^+\setminus\{E_1,E_2\}$.
\item 
As $\xi\uparrow\xi_{E_1}$, we have
\begin{align}\label{behavioratxiE1}
\alpha(\xi) \to E_1,\quad \mu_1(\xi) \to\mu_1(\xi_{E_1}),\quad \mu_2(\xi) \to\mu_2(\xi_{E_1}),
\end{align}
where $\mu_1(\xi_{E_1})$ and $\mu_2(\xi_{E_1})$ are given by \eqref{mu1mu2atxiE1}, i.e., there is a continuous transition from the genus $0$ sector $\xi>\xi_{E_1}$ to the genus $2$ sector at $\xi =\xi_{E_1}$. 
\item 
For all $\xi\in(\xi_m,\xi_{E_1})$ sufficiently close to $\xi_{E_1}$, we have $\alpha_1(\xi)>\Re E_1$ so that the branch cut $[\bar{\alpha},\alpha]$ lies to the right of the cut $[\bar{E}_1,E_1]$. In fact, as $\xi\uparrow\xi_{E_1}$,
\begin{align}\label{alphaexpansionE1}
\alpha(\xi) = E_1 + c_1 \frac{\xi_{E_1} - \xi}{|\ln(\xi_{E_1}- \xi)|} +\osmall\bigg(\frac{\xi_{E_1} - \xi}{|\ln(\xi_{E_1} - \xi)|}\bigg),
\end{align}
where
\[
c_1\coloneqq\frac{2BE_1}{A^2+4\ii AB-3B^2-B|E_2|}
\]
has strictly positive real and imaginary parts.
\item 
As $\xi\downarrow\xi_m$, at least one of the following occurs: 
\begin{enumerate}[\rm(i)]
\item 
the zeros $\mu_1$ and $\mu_2$ merge,
\item 
$\alpha(\xi)$ and $\overline{\alpha(\xi)}$ merge at a point on the real axis, i.e., $\alpha_2(\xi)\downarrow 0$,
\item 
$\alpha(\xi)$ approaches $E_1$ or $E_2$. 
\item 
$\xi_m=-\infty$.
\end{enumerate}
\item 
$\alpha(\xi)=\alpha_1(\xi)+\ii\alpha_2(\xi)$ satisfies the following nonlinear ODE for $\xi\in(\xi_m,\xi_{E_1})$:
\begin{align}\label{alphaODE}
\begin{pmatrix}\alpha_1'(\xi)\\ \alpha_2'(\xi)\end{pmatrix}
= - P^{-1}G - P^{-1} \C{A} 
\begin{pmatrix}\int_{a_1}\frac{k^2(k-\alpha_1)}{w(k)}\dd k\\
\int_{a_2}\frac{k^2(k-\alpha_1)}{w(k)}\dd k\end{pmatrix},
\end{align}
where 
\begin{itemize}
\item 
The matrix $P(\xi,\alpha_1,\alpha_2)$ and the vector $G(\alpha_1,\alpha_2)$ are defined by
\begin{align}\label{Pdef}
P =\begin{pmatrix} P_{11} & P_{12}\\ P_{21} & P_{22} \end{pmatrix},\qquad G =\begin{pmatrix}G_1 \\G_2 \end{pmatrix},
\end{align}
where the entries $\accol{P_{ij}(\xi,\alpha_1,\alpha_2)}_{i,j=1}^2$ and $\accol{G_j(\alpha_1,\alpha_2)}_1^2$ are polynomials given by
\begin{align}
& P_{11} = 12 \alpha_1^3+2 \alpha_1^2 \xi +6 \alpha_1 \alpha_2^2+4 \alpha_1 A^2+\alpha_2^2 \xi -4 \alpha_1  B^2,\notag\\
& P_{21} = -12 \alpha_1^2-2 \alpha_1 \xi -4 A^2+6 \alpha_2^2+4 B^2,\notag\\
& P_{12} =\alpha_2 \left(\alpha_1 \xi +4 A^2-6 \alpha_2^2-4 B^2\right),\notag\\ 
\label{Pijdef}
& P_{22} =\alpha_2 (12 \alpha_1+\xi ),
\end{align}
and
\begin{equation}\label{G1G2def}
G_1(\alpha_1,\alpha_2) =\alpha_1(\alpha_1^2 +\alpha_2^2),\qquad G_2(\alpha_1,\alpha_2) =\alpha_2^2 - \alpha_1^2.
\end{equation}
\item 
The zeros $\mu_j =\mu_j(\xi)$ are expressed in terms of $\xi$ and $\alpha_j =\alpha_j(\xi)$ by
\begin{subequations}\label{mu1mu2}
\begin{align}
\mu_1 =\frac{1}{8} \left(-4 \alpha_1 -\xi - \sqrt{-48 \alpha_1^2-8 \alpha_1 \xi -64 A^2+32 \alpha_2^2+64 B^2+\xi ^2} \right),
	\\
\mu_2 =\frac{1}{8} \left(-4 \alpha_1 -\xi +\sqrt{-48 \alpha_1^2-8 \alpha_1 \xi -64 A^2+32 \alpha_2^2+64 B^2+\xi ^2} \right).
\end{align}	
\end{subequations}
\end{itemize}
\end{enumerate}
\end{theorem}
%-------------------%

%-------------------%
%:rem 6.2
%-------------------%
\begin{remark}
Numerical simulations strongly suggest that as $\xi\downarrow\xi_m$ (see item (f)) 
\begin{itemize}
\item
case (i) (the zeros $\mu_1$ and $\mu_2$ merge) occurs if $1<\frac{A}{B}<\frac{2}{7}(2+3\sqrt{2})$, 
\item
case (ii) ($\alpha(\xi)$ and $\overline{\alpha(\xi)}$ merge at a point on the real axis) occurs if $0<\frac{A}{B}<1$,
\item
whereas we expect both (i) and (ii) to occur for $\frac{A}{B}=1$.
\end{itemize}
\end{remark}
%-------------------%

%---------------------------------------------------------%
%:s.6.2
%---------------------------------------------------------%
\subsection{Proof of Theorem \ref{genus2existenceth}}\label{proofsec}
The conditions in \eqref{dg-last-conditions-2} can be written more explicitly as
\begin{align*}
& 4(\alpha_1 +\mu_1 +\mu_2) = -\xi,\\
& 2 \mu_2 (\alpha_1+\mu_1)+2 \alpha_1 \mu_1-2 A^2+\alpha_2^2+2 B^2 = 0.
\end{align*}
Solving these two equations for $\mu_1$ and $\mu_2$, we find \eqref{mu1mu2}.

We write $\alpha=\alpha_1+\ii\alpha_2$ and let $\mathbf{x}=(\xi,\alpha_1,\alpha_2)\in\D{R}^3$ denote the vector with coordinates $(\xi,\alpha_1,\alpha_2)$. Let $\C{W}$ denote the open subset of $\D{R}^3$ consisting of all points $\mathbf{x}=(\xi,\alpha_1,\alpha_2) \in\D{R}^3$ such that $\alpha_2 > 0$, $\alpha \notin \{E_1,E_2\}$, and the expression under the square roots in \eqref{mu1mu2} is strictly positive. If we want to emphasize the dependence on $\mathbf{x}=(\xi,\alpha_1,\alpha_2)$, we will write $\dd g\equiv\dd g(k;\mathbf{x})$, $\mu_1\equiv\mu_1(\mathbf{x})$, and $\mu_2\equiv\mu_2(\mathbf{x})$, where $\dd g(k;\mathbf{x})$ is evaluated with $\mu_1,\mu_2$ given by \eqref{mu1mu2}.

We define the map $F\colon\C{W}\to\D{R}^2$ by (see \eqref{intdg})
\[
F(\mathbf{x}) =\frac{1}{\ii} \begin{pmatrix} \int_{a_1} \dd g(k;\mathbf{x}) \\\int_{a_2} \dd g(k;\mathbf{x}) \end{pmatrix}
\]
and let $D_\alpha F$ denote the Jacobian matrix
\begin{align}\label{D2Fdef}
D_\alpha F(\mathbf{x}) 
=\begin{pmatrix}\partial_{\alpha_1}F_1&\partial_{\alpha_2}F_1\\\partial_{\alpha_1}F_2&\partial_{\alpha_2}F_2\end{pmatrix} 
=\frac{1}{\ii} \begin{pmatrix} \int_{a_1} \partial_{\alpha_1}\dd g & \int_{a_1} \partial_{\alpha_2}\dd g \\
\int_{a_2} \partial_{\alpha_1}\dd g &\int_{a_2} \partial_{\alpha_2}\dd g \end{pmatrix},
\end{align}
where $\partial_{\alpha_j}\coloneqq\frac{\partial}{\partial\alpha_j}$, $j=1,2$.

%-------------------%
%:rem
%-------------------%
\begin{remark*}
The function $F$ is in general multivalued on $\C{W}$, because of a monodromy as $\alpha$ encircles $E_1$ or $E_2$. Strictly speaking, we should therefore define $F\colon\hat{\C{W}}\to\D{R}^2$, where $\hat{\C{W}}$ denotes the universal cover of $\C{W}$. However, it can be proved using \eqref{dgdkasymptotics} that $F\mapsto\mathcal{M}F$ for some matrix $\mathcal{M}$ under such a monodromy transformation. In particular, the zero locus of $F$ is a well-defined subset of $\C{W}$. Thus, this distinction is of no consequence for us and will be suppressed from the notation. 
\end{remark*}
%-------------------%

%-------------------%
%:lem 6.3
%-------------------%
\begin{lemma}\label{Flemma}
$F\colon\C{W}\to\D{R}^2$ is a smooth map such that $\det D_\alpha F\neq 0$ at each point of $\C{W}$.
\end{lemma}
%-------------------%

%-------------------%
\begin{proof}
Smoothness follows directly from the definitions. We will prove that $\det D_\alpha F\neq 0$.
For each $\mathbf{x}\in\C{W}$, $\dd g(k;\mathbf{x})$ is a meromorphic differential on $\Sigma_\alpha$ whose only poles lie at $\infty^\pm$ and whose singular behavior at $\infty^\pm$ (which is prescribed by \eqref{dgdkasymptotics}) is independent of $\alpha_1$ and $\alpha_2$. It follows that $\partial_{\alpha_1}\dd g$ and $\partial_{\alpha_2}\dd g$ are holomorphic differentials on $\Sigma_\alpha$. More precisely, a direct computation gives
\begin{equation}\label{ddgdalpha}
\partial_{\alpha_1}\dd g =\frac{P_{11} + k P_{21}}{w(k)}\dd k,\qquad\partial_{\alpha_2}\dd g =\frac{P_{12} + k P_{22}}{w(k)}\dd k,
\end{equation}
where $\accol{P_{ij}(\xi,\alpha_1,\alpha_2)}_{i,j=1}^2$ are the polynomials defined in \eqref{Pijdef}.

In terms of $P_{ij}$, $i,j=1,2$ and $\hat\zeta_l$, $l=1,2$ (defined in \eqref{hatzetadef}), we can write \eqref{ddgdalpha} as
\begin{align*}
\begin{pmatrix}\partial_{\alpha_1}\dd g \\ \partial_{\alpha_2}\dd g \end{pmatrix}
 =\begin{pmatrix} P_{11} \hat{\zeta}_1 + P_{21} \hat{\zeta}_2 \\ P_{12} \hat{\zeta}_1 + P_{22} \hat{\zeta}_2 \end{pmatrix}.
\end{align*}
Substitution into \eqref{D2Fdef} yields
\begin{align}
D_\alpha F(\mathbf{x}) 
& = 
\frac{1}{\ii}\begin{pmatrix} P_{11} \int_{a_1}  \hat{\zeta}_1 + P_{21} \int_{a_1} \hat{\zeta}_2 & P_{12} \int_{a_1} \hat{\zeta}_1 + P_{22} \int_{a_1} \hat{\zeta}_2\\
 P_{11} \int_{a_2} \hat{\zeta}_1 + P_{21} \int_{a_2} \hat{\zeta}_2 &P_{12} \int_{a_2}  \hat{\zeta}_1 + P_{22} \int_{a_2} \hat{\zeta}_2 \end{pmatrix}\notag\\
& =\frac{1}{\ii} \begin{pmatrix} P_{11} \left(\C{A}^{-1}\right)_{11} + P_{21} \left(\C{A}^{-1}\right)_{12} & P_{12} \left(\C{A}^{-1}\right)_{11} + P_{22} \left(\C{A}^{-1}\right)_{12}\notag\\
 P_{11} \left(\C{A}^{-1}\right)_{21} + P_{21} \left(\C{A}^{-1}\right)_{22} &P_{12} \left(\C{A}^{-1}\right)_{21} + P_{22} \left(\C{A}^{-1}\right)_{22} \end{pmatrix}\\ 
\label{D2FAinvP}
& = -\ii\C{A}^{-1}P.
\end{align}
We conclude that $D_\alpha F$ is invertible if and only if the matrix $P$ is invertible. A straightforward computation using \eqref{Pijdef} gives
\begin{align*}
\det P 
&=\alpha_2 \Big[16 A^4+16 A^2 \left(\alpha_1 (6 \alpha_1+\xi)-3 \alpha_2^2-2 B^2\right)+36 \left(4 \alpha_1^4+\alpha_2^4\right)+48 \alpha_1^3 \xi\\
&\qquad+\xi ^2 \left(4 \alpha_1^2+\alpha_2^2\right)+16 B^4-16 B^2 \left(\alpha_1 (6 \alpha_1+\xi )-3 \alpha_2^2\right)\Big].
\end{align*}
Recalling the expressions \eqref{mu1mu2} for $\mu_1,\mu_2$, this can be rewritten more concisely as
\begin{align*}
\det P & = 16 \alpha_2 \left((\alpha_1-\mu_1)^2+\alpha_2^2\right) \left((\alpha_1-\mu_2)^2+\alpha_2^2\right)\\
& = 16 \alpha_2 |\alpha - \mu_1|^2 |\alpha - \mu_2|^2.
\end{align*}
In particular, $\det P>0$ on $\C{W}$ (on which $\alpha_2>0$). 
\end{proof}
%-------------------%

If $\mathbf{x}=(\xi,\alpha_1,\alpha_2) \in\C{W}$ is a solution of $F(\mathbf{x}) = 0$, then Lemma \ref{Flemma} and the implicit function theorem implies that the level set $F = 0$ locally near $\mathbf{x}$ can be parametrized by a smooth curve $\xi \mapsto (\xi,\alpha_1(\xi),\alpha_2(\xi))$ such that
\begin{align}\label{dalphadxi}
\begin{pmatrix} \alpha_1'(\xi) \\ \alpha_2'(\xi) \end{pmatrix}
= - D_\alpha F(\xi,\alpha_1(\xi),\alpha_2(\xi))^{-1} \begin{pmatrix} \partial_{\xi}F_1 \\ \partial_{\xi}F_2 \end{pmatrix}\bigg|_{(\xi,\alpha_1(\xi),\alpha_2(\xi))},
\end{align}
where $\partial_{\xi}\coloneqq\frac{\partial}{\partial\xi}$. A computation shows that
\[
\partial_{\xi}\dd g =\frac{(k+\alpha_1)((k-\alpha_1)^2 +\alpha_2^2)}{w(k)} = G_1 \hat{\zeta}_1 + G_2 \hat{\zeta}_2 +\frac{k^2(k-\alpha_1)}{w(k)}\dd k,
\]
where the polynomials $\accol{G_j(\alpha_1,\alpha_2)}_1^2$ are given by \eqref{G1G2def}.
Thus
\begin{align}
\begin{pmatrix} \partial_{\xi}F_1 \\ \partial_{\xi}F_2 \end{pmatrix}
& =\frac{1}{\ii} \begin{pmatrix} \int_{a_1} \partial_{\xi} \dd g \\\int_{a_2} \partial_{\xi}  \dd g \end{pmatrix}
=\frac{1}{\ii} \begin{pmatrix} G_1 \int_{a_1} \hat{\zeta}_1 + G_2 \int_{a_1} \hat{\zeta}_2 +\int_{a_1} \frac{k^2(k-\alpha_1)}{w(k)}\dd k
\\G_1 \int_{a_2} \hat{\zeta}_1 + G_2 \int_{a_2} \hat{\zeta}_2 \dd g +\int_{a_2} \frac{k^2(k-\alpha_1)}{w(k)}\dd k \end{pmatrix}\notag\\
& =\frac{1}{\ii} \begin{pmatrix} G_1 (\C{A}^{-1})_{11} + G_2 (\C{A}^{-1})_{12} +\int_{a_1} \frac{k^2(k-\alpha_1)}{w(k)}\dd k
\\G_1 (\C{A}^{-1})_{21} + G_2 (\C{A}^{-1})_{22} +\int_{a_2} \frac{k^2(k-\alpha_1)}{w(k)}\dd k \end{pmatrix}\notag\\
\label{dFdxi}
& = -\ii\C{A}^{-1} G  - \ii\begin{pmatrix} \int_{a_1} \frac{k^2(k-\alpha_1)}{w(k)}\dd k \\
\int_{a_2} \frac{k^2(k-\alpha_1)}{w(k)}\dd k \end{pmatrix},
\end{align}
where
\[
G=\begin{pmatrix}G_1 \\G_2 \end{pmatrix}.
\]
Note that $\frac{k^2(k-\alpha_1)}{w(k)}\dd k$ is a meromorphic differential on $\Sigma_\alpha$ of the second kind (i.e., all residues are zero) which is holomorphic except for two double poles at $\infty^\pm$ such that
\[
\frac{k^2(k-\alpha_1)}{w(k^\pm)} =\pm1 +\ord(k^{-2}),\qquad k \to\infty.
\]
Substituting \eqref{D2FAinvP} and \eqref{dFdxi} into \eqref{dalphadxi}, we find
\begin{align*}
\begin{pmatrix} \alpha_1'(\xi) \\ \alpha_2'(\xi) \end{pmatrix}
& = - \ii P^{-1} \C{A}\left(-\ii\C{A}^{-1} G  - \ii\begin{pmatrix} \int_{a_1} \frac{k^2(k-\alpha_1)}{w(k)}\dd k \\
 \int_{a_2} \frac{k^2(k-\alpha_1)}{w(k)}\dd k \end{pmatrix}\right)\\
& = - P^{-1} G - P^{-1} \C{A} \begin{pmatrix} \int_{a_1} \frac{k^2(k-\alpha_1)}{w(k)}\dd k \\
\int_{a_2} \frac{k^2(k-\alpha_1)}{w(k)}\dd k \end{pmatrix},
\end{align*}
which is the ODE in \eqref{alphaODE}. 

We have shown that the nonlinear ODE \eqref{alphaODE} describes the solution curves of $F = 0$ whenever such curves exist. By Lemma \ref{Flemma}, each solution curve can be continued as long as it stays in $\C{W}$ and the zeros $\{\mu_j\}_1^2$ remain bounded. We will show in the next lemma that $\alpha,\mu_1,\mu_2$ remain bounded on the zero set of $F$ unless $|\xi|\to\infty$. Therefore, the solution curve can either be extended indefinitely to all $\xi \in (-\infty,\xi_{E_1})$ or it ends at a point $\xi =\xi_m$ where at least one of the following must occur: 
\begin{enumerate}[(i)]
\item 
the zeros $\mu_1$ and $\mu_2$ merge,
\item
$\alpha_2\downarrow 0$ (i.e., $\alpha$ and $\bar{\alpha}$ merge), 
\item
$\alpha$ hits one of the branch points $E_1$ or $E_2$.
\end{enumerate}

%-------------------%
%:lem 6.4
%-------------------%
\begin{lemma}
As $\alpha=\alpha_1+\ii\alpha_2\to\infty$, the function $F$ satisfies
\[
|F(\xi,\alpha_1,\alpha_2)|\to\infty,
\] 
uniformly for $\xi$ in bounded subsets of $\D{R}$ and $\arg\alpha\in [0,\pi]$. In particular, if $F(\xi,\alpha_1(\xi),\alpha_2(\xi))=0$, then $\alpha(\xi)$, $\mu_1(\xi)$, and $\mu_2(\xi)$ remain bounded whenever $\xi$ does.
\end{lemma}
%-------------------%

%-------------------%
\begin{proof}
Let $w_1(k) = \sqrt{(k-E_1)(k-\bar{E}_1)(k-E_2)(k-\bar{E}_2)}$ with branch cuts along $[E_1, \bar{E}_1]$ and $[E_2, \bar{E}_2]$ and the branch fixed by the condition that $w_1(k) \sim k^2$ as $k \to \infty$. As $\alpha \to \infty$ along the ray $\alpha_1 = q \alpha_2$, $q \in \D{R}$, we have
\[
\dd g(k; \mathbf{x}) = \bigg[2(2q^2 -1)\sqrt{1 + q^2} \alpha_2^3  + \bigg(\frac{6qk}{\sqrt{1 + q^2}} + \xi q \sqrt{1 + q^2}\bigg)\alpha_2^2\bigg]\frac{\dd k}{w_1(k)}+\ord(\alpha_2),
\]
uniformly for $q$ and $\xi$ in bounded subsets of $\D{R}$ and for $k$ in compact subsets of $\D{C}\setminus\{E_1,\bar{E}_1,E_2,\bar{E}_2\}$.
Letting
\[
J_j\coloneqq\frac{1}{\ii} \int_{a_j} \frac{\dd k}{w_1(k)}, \quad
K_j\coloneqq\frac{1}{\ii} \int_{a_j} \frac{k\,\dd k}{w_1(k)}, \qquad j = 1,2,
\]
we find, for $j = 1,2$,
\[
F_j(\mathbf{x}) = 2(2q^2 -1)\sqrt{1 + q^2} \alpha_2^3 J_j  
+ \bigg(\frac{6qK_j}{\sqrt{1 + q^2}} + \xi q \sqrt{1 + q^2}J_j\bigg)\alpha_2^2 + \ord(\alpha_2),
\]
uniformly for $q$ and $\xi$ in bounded subsets of $\D{R}$. Using that
\[
J_1 = -J_2 \neq 0, \qquad K_1 = K_2 \neq 0,
\]
we infer that
\begin{subequations}
\begin{align}\label{F1F2plusminusa}
&F_1(\mathbf{x}) - F_2(\mathbf{x}) = 4(2q^2 -1)\sqrt{1 + q^2} \alpha_2^3 J_1 + 2\xi q \sqrt{1 + q^2}J_1 \alpha_2^2 +\ord(\alpha_2),\\
\label{F1F2plusminusb}
&F_1(\mathbf{x}) + F_2(\mathbf{x}) = \frac{12qK_1}{\sqrt{1 + q^2}} \alpha_2^2 +\ord(\alpha_2),
\end{align}
\end{subequations}
uniformly for $q$ and $\xi$ in bounded subsets of $\D{R}$.
Equation \eqref{F1F2plusminusa} implies that $|F(\xi, \alpha_1, \alpha_2)| \to \infty$ as $\alpha \to \infty$, uniformly for $q$ in compact subsets of $\D{R} \setminus \{\pm 1/\sqrt{2}\}$ and $\xi$ in bounded subsets of $\D{R}$. Equation \eqref{F1F2plusminusb} implies that $|F(\xi, \alpha_1, \alpha_2)| \to \infty$ as $\alpha \to \infty$, uniformly for $q$ in compact subsets of $\D{R} \setminus \{0\}$ and $\xi$ in bounded subsets of $\D{R}$. Combining these two conclusions, we find that $|F(\xi, \alpha_1, \alpha_2)| \to \infty$ as $\alpha \to \infty$ for $\arg \alpha \in [\epsilon, \pi - \epsilon]$. 

To show that $|F(\xi, \alpha_1, \alpha_2)| \to \infty$ as $\alpha \to \infty$ also for $\arg \alpha \in [0, \epsilon] \cup [\pi - \epsilon, \pi]$, we instead use the fact that, as $\alpha \to \infty$ along the ray $\alpha_2 = \tilde{q} \alpha_1$, $\tilde{q} \in \D{R}$, we have
\[
\dd g(k; \mathbf{x}) = 2(2-\tilde{q}^2)\sqrt{1 + \tilde{q}^2} |\alpha_1|^3 \frac{\dd k}{w_1(k)}  + \ord(\alpha_1^2),
\]
uniformly for $\tilde{q}$ and $\xi$ in bounded subsets of $\D{R}$ and $k$ in compact subsets of $\D{C}\setminus\{E_1,\bar{E}_1,E_2,\bar{E}_2\}$.

We conclude that $F(\xi, \alpha_1, \alpha_2) \to \infty$ as $\alpha\to \infty$, uniformly for $\xi$ in bounded subsets of $\D{R}$ and $\arg \alpha \in [0, \pi]$. The second statement follows because, by \eqref{mu1mu2}, $\mu_1$ and $\mu_2$ remain bounded whenever $\alpha$  and $\xi$ stay bounded.
\end{proof}
%-------------------%

It remains to show that the zero set $F = 0$ contains a curve which satisfies \eqref{behavioratxiE1} and \eqref{alphaexpansionE1} as $\xi \uparrow \xi_{E_1}$. The limits $\lim_{\xi \uparrow \xi_{E_1}}\mu_j(\xi) =\mu_j(\xi_{E_1})$, $j = 1,2$, are a consequence of \eqref{mu1mu2} if we can show that the zero set of $F$ contains a smooth curve $(\xi,\alpha_1(\xi),\alpha_2(\xi))$ which approaches the point 
\[
\mathbf{x}_0 \coloneqq (\xi_{E_1},\Re E_1,\Im E_1) \in\partial \C{W}
\] 
as $\xi\uparrow \xi_{E_1}$. To prove this, we will first show that $F$ has a continuous extension to $\mathbf{x}_0$ such that $F(\mathbf{x}_0) = 0$ and then apply a boundary version of the implicit function theorem at the point $\mathbf{x}_0$. The proof is complicated by the fact that the Riemann surface $\Sigma_\alpha$ degenerates to a genus zero surface as $\alpha$ approaches $E_1$. This implies that the partial derivatives $\partial_{\alpha_1}F_1$ and $\partial_{\alpha_2}F_1$ blow up like $\ln|\alpha-E_1|$ in this limit. Therefore, we cannot apply the implicit function theorem at the point $\mathbf{x}_0\in\partial\C{W}$ directly to $F$; instead we will introduce a function $\tilde{F}$, which is a modified version of $F$, and apply the implicit function theorem to this modified function.

We begin by establishing the behavior of $F$ and its first order partial derivatives as $\alpha\to E_1$. The analysis of the second component $F_2$ is easier than the analysis of $F_1$, because $F_2$ is nonsingular at $\alpha = E_1$. We therefore begin with $F_2$.

Let $B_R \subset \D{R}^3$ denote the open ball of radius $R > 0$ centered at $\mathbf{x}_0$. Let $L\subset\D{R}^3$ denote the line on which $\alpha = E_1$:
\[
L=\{(\xi,\Re E_1,\Im E_1)\mid\xi\in\D{R}\}.
\]
Let $\mathbf{x}_L = (\xi,\Re E_1,\Im E_1)$ denote the orthogonal projection of $\mathbf{x}=(\xi,\alpha_1,\alpha_2)$ onto $L$. Note that $\dist(\mathbf{x}, L) = |\alpha-E_1|$. By choosing $R>0$ sufficiently small, we may assume that $\bar{B}_R\setminus L\subset\C{W}$ and, say, $R < \min\{A,B,1\}/2$.

Let $\Sigma_0$ denote the genus $0$ Riemann surface with a single cut from $\bar{E}_2$ to $E_2$ defined by 
\[
w_0^2=(k-E_2)(k-\bar{E}_2).
\]
We view this as a two-sheeted cover of the complex plane such that $w_0(k)=\sqrt{(k-E_2)(k-\bar{E}_2)}\sim k$ as $k\to\infty$ on the upper sheet. 

%-------------------%
%:lem 6.5
%-------------------%
\begin{lemma}[Behavior of $F_2(\mathbf{x})$ as $\alpha\to E_1$]\label{F2lemma}
The function $F_2\colon B_R \setminus L \to\D{R}$ extends to a smooth function $B_R \to\D{R}$. Moreover, the following estimates hold uniformly for $\mathbf{x}\in B_R$:
\begin{subequations}\label{F2expansions}
\begin{align}\label{F2expansion}
F_2(\mathbf{x}) &=\ord(|\alpha-E_1|)
	\\ \label{dF2dxiexpansion}
\partial_{\xi}F_2(\mathbf{x}) &=\ord(|\alpha-E_1|),
	\\\label{dF2dalphajexpansion}	
\partial_{\alpha_j}F_2(\mathbf{x}) &= q_j(\xi) +\ord(|\alpha-E_1|),\qquad j = 1,2,
\end{align}
\end{subequations}
where $\accol{q_j(\xi)}_1^2$ are linear functions of $\xi\in\D{R}$ given by
\begin{align*}
& q_1(\xi) = -\pi\,\Im\bigg\{\frac{\C{Q}(\xi)}{\sqrt{B} \sqrt{E_2}} \bigg\},\qquad
q_2(\xi) =\pi\,\Re\bigg\{\frac{\C{Q}(\xi)}{\sqrt{B} \sqrt{E_2}} \bigg\},
\end{align*}
with
\begin{align}\label{X3def}
\C{Q}(\xi) \coloneqq -2\ii A^2+A (\xi -12 B)+2\ii B (4 B-\xi)
\end{align}
and the principal branch is used for $\sqrt{E_2}$. For $\xi =\xi_{E_1}$, it holds that $q_1(\xi_{E_1}) \neq 0$ and $q_2(\xi_{E_1}) \neq 0$.
\end{lemma}
%-------------------%

%-------------------%
\begin{proof}
In the limit as $\mathbf{x}\in B_R\setminus L$ approaches $L$, we have $\alpha\to E_1$ and $\bar{\alpha} \to\bar{E}_1$, so that the Riemann surface $\Sigma_\alpha$ degenerates to the genus zero surface $\Sigma_0$. With appropriate choices of the branches, we have
\[
F_2(\mathbf{x}) =\frac{1}{\ii}\int_{a_2} \frac{4(k-\mu_1(\mathbf{x}))(k-\mu_2(\mathbf{x}))\sqrt{(k-\alpha)(k-\bar\alpha)}}{\sqrt{(k-E_1)(k-\bar E_1)(k-E_2)(k-\bar E_2)}}\dd k.
\]
We see that the integrand is smooth as a function of $\mathbf{x} \in B_R$ and analytic as a function of $k$ for $k$ in a neighborhood of the contour $a_2$. This shows that $F_2\colon B_R\setminus L\to\D{R}$ extends to a smooth function $B_R\to\D{R}$. 

To prove \eqref{F2expansion}, we note that a Taylor expansion gives
\begin{subequations}\label{dgdkexpansion}
\begin{equation}\label{dgdkexpansion1}
\frac{\dd g}{\dd k}(k;\mathbf{x})=\frac{4(k-\mu_1(\mathbf{x}_L))(k-\mu_2(\mathbf{x}_L))}{w_0(k)}\left(1+\ord(|\alpha-E_1|)\right),
\end{equation}
uniformly for $\mathbf{x}\in B_R$ and $k$ on $a_2$. Similarly we also have
\begin{equation}\label{dgdkexpansion2}
\frac{\dd g}{\dd k}(k;\mathbf{x}) 
=\frac{4(k-\mu_1(\mathbf{x}))(k - \mu_2(\mathbf{x}))}{w_0(k)}	
\biggl\{1-\frac{\alpha-E_1}{2(k-E_1)}-\frac{\bar{\alpha} - \bar{E}_1}{2(k-\bar{E}_1)}+\ord(|\alpha-E_1|^2)\biggr\},
\end{equation}
\end{subequations}
uniformly for $\mathbf{x}\in B_R$ and $k$ on $a_2$. It follows from \eqref{dgdkexpansion1} that 
\[
F_2(\mathbf{x})=\frac{1}{\ii}\int_{a_2}\dd g(k;\mathbf{x})=\frac{1}{\ii}\int_{a_2}\frac{4(k-\mu_1(\mathbf{x}_L))(k-\mu_2(\mathbf{x}_L))}{w_0(k)}\dd k+\ord(|\alpha-E_1|).
\]
Deforming the contour to infinity and using that
\begin{align*}
\frac{4(k-\mu_1(\mathbf{x}_L))(k-\mu_2(\mathbf{x}_L))}{w_0(k)} 
&=4k - 4(\mu_1(\mathbf{x}_L) +\mu_2(\mathbf{x}_L) -B)\\
&\quad+\frac{4 (\mu_1(\mathbf{x}_L)-B) (\mu_2(\mathbf{x}_L)-B)-2 A^2}{k} +\ord(k^{-2})\\
&=4k +\xi +\ord(k^{-2}),\qquad k \to\infty,
\end{align*}
we find that the integral over $a_2$ vanishes. This proves \eqref{F2expansion}.

To derive the expansions of the first-order partial derivatives, we use \eqref{dgdkexpansion} to compute
\[
\partial_{\xi}\frac{\dd g}{\dd k}(k;\mathbf{x})=\partial_{\xi}  \frac{4(k-\mu_1(\mathbf{x}))(k-\mu_2(\mathbf{x}))(k-\alpha)(k-\bar{\alpha})}{w(k;\mathbf{x})}=X_0+\ord(|\alpha-E_1|)
\]
and, similarly,
\[
\partial_{\alpha_j}\frac{\dd g}{\dd k}(k;\mathbf{x})= X_j +\ord(|\alpha-E_1|),\qquad j = 1,2,
\]
where the error terms are uniform with respect to $k \in a_2$ and $\{X_j\}_0^2$ are short-hand notations for the expressions
\begin{align*}
& X_0 \coloneqq \frac{k-B}{w_0(k)},
	\\
& X_1 \coloneqq \frac{A^2 (-10 B+2 k+\xi )+2 B (B+k) (\xi -4 B)}{(k-E_1)(k-\bar{E}_1)w_0(k)},
	\\
& X_2 \coloneqq -\frac{A \left(2 A^2+4 B^2+B (12 k+\xi )-k \xi \right)}{(k-E_1)(k-\bar{E}_1)w_0(k)}.
\end{align*}
Consequently, deforming the contour to infinity and noting that the residue of $X_j$ at $k =\infty$ vanishes for each $j$, we obtain
\begin{align*}
\partial_{\xi}F_2(\mathbf{x})
& =\frac{1}{\ii} \int_{a_2} X_0 \dd k +\ord(|\alpha-E_1|) =\ord(|\alpha-E_1|),\\
\partial_{\alpha_1}F_2(\mathbf{x})
& =\frac{1}{\ii} \int_{a_2} X_1 \dd k +\ord(|\alpha-E_1|)
= -2\pi \bigg(\Res_{k = E_1} +\Res_{k =\bar{E}_1}\bigg) X_1+\ord(|\alpha-E_1|)\\
&=q_1(\xi) +\ord(|\alpha-E_1|),\\
\partial_{\alpha_2}F_2(\mathbf{x})
& =\frac{1}{\ii} \int_{a_2} X_2 \dd k +\ord(|\alpha-E_1|)
 = -2\pi \bigg(\Res_{k = E_1} +\Res_{k =\bar{E}_1}\bigg) X_2
+\ord(|\alpha-E_1|),\\
&=q_2(\xi) +\ord(|\alpha-E_1|),
\end{align*}
uniformly for $\mathbf{x} \in\bar{B}_R$. This proves \eqref{dF2dxiexpansion} and \eqref{dF2dalphajexpansion}.

In order to prove that $q_1(\xi_{E_1}) \neq 0$ and $q_2(\xi_{E_1}) \neq 0$, it is sufficient to verify that $\frac{\C{Q}^2}{B E_2}\notin\D{R}$. But evaluation at $\xi =\xi_{E_1}$ gives
\[
\C{Q}(\xi_{E_1}) \coloneqq 2 A \left(|E_2|-5 B\right)-4\ii B \left(|E_2|-B\right)-2\ii A^2
\]
and then a computation yields
\[
\Im\{\C{Q}^2 \bar{E}_2\} = 16 \big[4 A B^3 (|E_2|-B)+A^3 B (3 |E_2|-5 B)\big].
\]
The right-hand side is strictly positive for $A,B>0$. This proves that $q_j(\xi_{E_1}) \neq 0$ for $j = 1,2$ and completes the proof of the lemma.
\end{proof}
%-------------------%

We next consider the first component $F_1(\mathbf{x})$ for $(\xi,\alpha)$ near $(\xi_{E_1}, E_1)$. Since it is enough for our purposes, we will for simplicity restrict attention to $\alpha$ such that $\alpha_1 \geq \Re E_1$;  this will simplify the specification of some branches of square roots. As above, we let $R> 0$ be small. We recall that $\mathbf{x}_0 = (\xi_{E_1},\Re E_1,\Im E_1) \in L$ and let $S_R \subset \D{R}^3$ denote the open half-ball 
\[
S_R = B_R \cap \{\alpha_1 > \Re E_1\}.
\]
Square roots and logarithms are defined using the principal branch unless specified otherwise.

%-------------------%
%:lem 6.6
%-------------------%
\begin{lemma}[Behavior of $F_1(\mathbf{x})$ as $\alpha\to E_1$]\label{F1lemma}
As $\mathbf{x} \in\bar{S}_R \setminus L$ approaches the line $L$ (in other words, as $\alpha\to E_1$), $F_1(\mathbf{x})$ admits an asymptotic expansion to all orders of the form
\begin{align}\label{F1allorders}
F_1(\mathbf{x}) \sim \Im\bigg\{\sum_{n,m = 0}^\infty \big[c_{nm}(\xi) + d_{nm}(\xi) (\alpha-E_1) \ln(\alpha-E_1)\big]
(\alpha-E_1)^n(\bar{\alpha} - \bar{E}_1)^m\bigg\},
\end{align}
where $\{c_{nm}(\xi), d_{nm}(\xi)\}_{n,m=0}^\infty$ are smooth complex-valued functions of $\xi$. Moreover, the expansion \eqref{F1allorders} can be differentiated termwise with respect to $\alpha_1$, $\alpha_2$, and $\xi$. In particular, the following estimates are valid uniformly for $\mathbf{x}=(\xi,\alpha_1,\alpha_2) \in\bar{S}_R \setminus L$:
\begin{subequations}\label{F1expansions}
\begin{align}\label{F1expansion}
F_1(\mathbf{x}) &= f_0(\xi) +\ord\big(|\alpha-E_1|(1 + |\ln|\alpha-E_1||)\big),\\ 
\label{dF1dxiexpansion}
\partial_{\xi}F_1(\mathbf{x}) &= f_0'(\xi) +\ord\big(|\alpha-E_1|(1 + |\ln|\alpha-E_1||)\big),\\
\label{dF1dalpha1expansion}	
\partial_{\alpha_1}F_1(\mathbf{x})&=\Im\{d_{00}(\xi) \ln(\alpha-E_1)\} + f_1(\xi) +\ord\big(|\alpha-E_1|(1 + |\ln|\alpha-E_1||)\big),\\
\label{dF1dalpha2expansion}	
\partial_{\alpha_2}F_1(\mathbf{x})&=\Im\{\ii d_{00}(\xi) \ln(\alpha-E_1)\} + f_2(\xi) +\ord\big(|\alpha-E_1|(1 + |\ln|\alpha-E_1||)\big),
\end{align}
\end{subequations}
where 
\begin{itemize}
\item 
$f_0(\xi)$ is the linear real-valued function defined by
\begin{align}\label{f0def}
f_0(\xi) =  - 8 \sqrt{B} \big(\Im\sqrt{E_2}\big) (\xi - \xi_{E_1}).
\end{align}
\item 
$d_{00}(\xi)$ is the linear function of $\xi \in\D{R}$ given by
\begin{align}\label{d00def}
 d_{00}(\xi) =\frac{-\ii \overline{\C{Q}(\xi)}}{\sqrt{B} \sqrt{\bar{E}_2}}
\end{align}
with $\C{Q}(\xi)$ defined in \eqref{X3def}.
\item 
$\{f_j(\xi)\}_1^2$ are smooth real-valued functions of $\xi \in\D{R}$.
\end{itemize}
\end{lemma}
%-------------------%

%-------------------%
\begin{proof}
In order to derive \eqref{F1allorders}, we fix a large negative number $p< 0$. For $z_0, z_1 \in\D{C}$, we let $[z_0,z_1]$ denote the straight line segment from $z_0$ to $z_1$, and we let $[z_0,z_1]^+$ denotes its preimage in the upper sheet under the natural projection $\Sigma_\alpha\to\D{C}$. Deforming the contour and using the symmetry $\dd g(k) =\overline{\dd g(\bar{k})}$, we see that, for $\mathbf{x} \in\bar{S}_R \setminus L$,
\begin{equation}\label{F1pE1}
F_1(\mathbf{x}) =\frac{1}{\ii} \int_{a_1} \dd g
= -\frac{2}{\ii}\bigg(\int_{[p, E_1]^+} \dd g +\int_{[\bar{E}_1,p]^+} \dd g\bigg)
=\Im\bigg\{-4 \int_{[p, E_1]^+} \dd g\bigg\}.
\end{equation}	
Defining the function $h(k;\mathbf{x})$ for $k$ in a neighborhood of $[p, E_1]$ by
\[
h(k;\mathbf{x}) = - \frac{4(k-\mu_1(\mathbf{x}))(k-\mu_2(\mathbf{x})) \sqrt{\bar{\alpha} - k}}{\sqrt{\bar{E}_1 - k} \sqrt{(E_2-k)(\bar{E}_2-k)}},
\]
we have
\[
h(k;\mathbf{x}) =\frac{\sqrt{E_1 - k}}{\sqrt{\alpha - k}}\,\frac{\dd g}{\dd k}(k^+;\mathbf{x}) \quad \text{for $k \in [p, E_1]$}.
\]
Here and elsewhere in the proof, the principal branch is adopted for all square roots and logarithms. The function $h$ depends smoothly on $\mathbf{x} \in\bar{S}_R$ and is analytic for $k$ in a neighborhood of $[p, E_1]$. Defining $I_l(\alpha)$ by
\[
I_l(\alpha) \coloneqq \int_{[p_1, E_1]} (E_1 - k)^{l - \frac{1}{2}} \sqrt{\alpha - k}\,\dd k,\qquad l = 0,1,\dots,
\]
and employing the expansion 
\[
h(k;\mathbf{x}) \sim \sum_{n,m,l \geq 0} h_{nml}(\xi) (\alpha-E_1)^{n}(\bar{\alpha} - \bar{E}_1)^{m}(E_1 - k)^l,
\]
where $h_{nml}(\xi)$ are smooth functions, we infer that if $p_1 \in [p, E_1]$ is a point sufficiently close to $E_1$, then we have the expansion
\begin{align}         
\int_{[p_1, E_1]^+}\dd g(k;\mathbf{x})
&=\int_{[p_1, E_1]}h(k;\mathbf{x}) \frac{\sqrt{\alpha - k}}{\sqrt{E_1 - k}}\,\dd k\notag\\
\label{intdghnml}
&\sim\sum_{n, m, l \geq 0} h_{nml}(\xi) (\alpha-E_1)^n(\bar{\alpha} - \bar{E}_1)^mI_l(\alpha)
\end{align}
and this expansion can be differentiated termwise with respect to $\alpha_1$, $\alpha_2$, and $\xi$.

We claim that there exist complex coefficients $\{q_l\}_{l\geq 0}$ and $\{r_{lj}\}_{l,j\geq 0}$ such that
\begin{equation}\label{Ilalphaexpansion}
I_l(\alpha) \sim q_l (\alpha-E_1)^{l+1}\ln(\alpha-E_1) +\sum_{j=0}^\infty r_{lj} (\alpha-E_1)^j
\end{equation}
for each integer $l \geq 0$ as $\alpha\to E_1$. Indeed, the statement is true for $l = 0$ by direct computation. Moreover, an integration by parts gives, for $l \geq 1$, 
\begin{align*}
I_l(\alpha)
&= - \frac{2}{3}(E_1 - p_1)^{l-\frac{1}{2}}(\alpha - p_1)^{\frac{3}{2}}
 - \frac{2(l - \frac{1}{2})}{3} \int_{[p_1, E_1]} (E_1 - k)^{l - \frac{3}{2}} (\alpha -k)^{\frac{3}{2}}\dd k\\
&= - \frac{2}{3}(E_1 - p_1)^{l-\frac{1}{2}}(\alpha - p_1)^{\frac{3}{2}}
- \frac{2(l - \frac{1}{2})}{3} \big\{(\alpha -E_1)I_{l-1}(\alpha) + I_l(\alpha)\big\}.
\end{align*}
Solving for $I_l(\alpha)$, we obtain
\[
I_l(\alpha)
=\frac{1}{1+\frac{2}{3}(l-\frac{1}{2})}\bigg\{-\frac{2}{3}(E_1 - p_1)^{l-\frac{1}{2}}(\alpha- p_1)^{\frac{3}{2}}-\frac{2}{3}(l-\frac{1}{2})(\alpha -E_1)I_{l-1}(\alpha)\bigg\},
\]
and hence \eqref{Ilalphaexpansion} follows for all integers $l \geq 0$ by induction.

Equations \eqref{intdghnml} and \eqref{Ilalphaexpansion} imply that, as $\alpha\to E_1$,
\begin{equation} \label{minus4intdg}
-4\int_{[p, E_1]^+} \dd g(k;\mathbf{x})
\sim\sum_{n,m \geq 0} \big[c_{nm}(\xi) + d_{nm}(\xi) (\alpha-E_1) \ln(\alpha-E_1)\big](\alpha-E_1)^n(\bar{\alpha} - \bar{E}_1)^m,
\end{equation}
where $c_{nm}(\xi), d_{nm}(\xi)$ are smooth complex-valued functions of $\xi$ which are independent of $\alpha_1$ and $\alpha_2$, and the expansion can be differentiated termwise with respect to $\alpha_1$, $\alpha_2$, and $\xi$. The existence of the expansion \eqref{F1allorders} now follows from \eqref{F1pE1}.

The rest of the lemma follows from \eqref{F1allorders} if we can verify the expressions \eqref{f0def} and \eqref{d00def} for $f_0$ and $d_{00}$. To derive the expression \eqref{f0def} for $f_0$, we note that by \eqref{F1pE1} (see also \eqref{dgdkexpansion1}) 
\begin{align*}
f_0(\xi) 
&=\lim_{\alpha\to E_1} F_1(\mathbf{x})
= 2\ii\lim_{\alpha\to E_1} \bigg(\int_{[p, E_1]^+} +\int_{[\bar{E}_1, p]^+}\bigg) \dd g(k;\mathbf{x})\\ 
& = 2\ii\bigg(\int_{[p,E_1]^+}+\int_{[\bar{E}_1, p]^+}\bigg) \dd g(k;\mathbf{x}_L) 
= -2\ii\int_{[\bar{E}_1,E_1]} \frac{4(k-\mu_1(\mathbf{x}_L))(k-\mu_2(\mathbf{x}_L))}{\sqrt{(E_2-k)(\bar{E}_2-k)}} \dd k.
\end{align*}
Substituting in the expressions for $\mu_1(\mathbf{x}_L)$ and $\mu_2(\mathbf{x}_L)$ and integrating, we find
\begin{align*}
f_0(\xi) & = - 2\ii\int_{[\bar{E}_1,E_1]} \frac{2A^2 - (B-k)(4k +\xi)}{\sqrt{A^2 + (B-k)^2}} \dd k
	\\
& = - 2\ii\Big[(2B + 2k +\xi) \sqrt{A^2 + (B-k)^2}\Big]_{k=\bar{E}_1}^{E_1}
	\\
& = 16 A \sqrt{B} \Re\sqrt{E_2} - 8 \xi \sqrt{B} \Im\sqrt{E_2}.
\end{align*}
Observing that the definition \eqref{xiE1def} of $\xi_{E_1}$ can be rewritten as
\[
\xi_{E_1} =\frac{2A}{\tan(\frac{1}{2} \arctan\frac{A}{B})}= 2 A\frac{\Re\sqrt{E_2}}{\Im\sqrt{E_2}},
\]
the expression for $f_0$ in \eqref{f0def} follows.

We finally derive the expression \eqref{d00def} for $d_{00}(\xi)$. Using \eqref{minus4intdg} and then \eqref{ddgdalpha}, we see that
\begin{align*}
d_{00}(\xi) 
&=\lim_{\alpha\to E_1} \frac{-4\int_{[p, E_1]^+} \partial_{\alpha_1}\dd g(k;\mathbf{x})}{\ln(\alpha-E_1)}
=\lim_{\alpha\to E_1} \frac{-4\int_{[p, E_1]^+} \frac{P_{11}(\mathbf{x}) + k P_{21}(\mathbf{x})}{w(k)}\dd k}{\ln(\alpha-E_1)}.
\end{align*}
Consequently,	
\begin{align*}
d_{00}(\xi)
&=\lim_{\alpha\to E_1} \frac{4\int_{[p, E_1]} \frac{P_{11}(\mathbf{x}_L) + k P_{21}(\mathbf{x}_L)}{\sqrt{E_1 -k}\sqrt{\bar{E}_1 -k} \sqrt{E_2-k}\sqrt{\bar{E}_2-k} \sqrt{\alpha - k}\sqrt{\bar{E}_1 - k}}\dd k}{\ln(\alpha-E_1)}\\
&=\lim_{\alpha\to E_1} \frac{4\frac{P_{11}(\mathbf{x}_L) + E_1 P_{21}(\mathbf{x}_L)}{\sqrt{\bar{E}_1 - E_1} \sqrt{E_2-E_1}\sqrt{\bar{E}_2-E_1} \sqrt{\bar{E}_1 - E_1}}
\int_{[p, E_1]} \frac{\dd k}{\sqrt{E_1 -k}\sqrt{\alpha - k}}}{\ln(\alpha-E_1)}.	
\end{align*}
Using that
\begin{align*}
\int_{[p, E_1]} \frac{\dd k}{\sqrt{E_1 -k}\sqrt{\alpha - k}}
& =  -2 \ln\big(\sqrt{E_1 - k} +\sqrt{\alpha - k}\big)\bigg|_{k = p}^{E_1}\\
& =  -2 \ln\bigg(\frac{\sqrt{\alpha-E_1}}{\sqrt{E_1 - p} +\sqrt{\alpha - p}}\bigg),
\end{align*}
we find
\[
d_{00}(\xi) = -4\frac{P_{11}(\mathbf{x}_L) + E_1 P_{21}(\mathbf{x}_L)}{(\bar{E}_1 - E_1) \sqrt{E_2-E_1}\sqrt{\bar{E}_2-E_1}}
=\frac{2A^2 +\ii  A(12 B - \xi) + 2B(\xi - 4B)}{\sqrt{B} \sqrt{\bar{E}_2}},
\]
which proves \eqref{d00def}.
\end{proof}
%-------------------%

Lemmas \ref{F2lemma} and \ref{F1lemma} show that the smooth map $F\colon\bar{S}_R\setminus L \to\D{R}^2$ extends continuously to a map $\bar{S}_R \to\D{R}^2$ (i.e., $F$ can be continuously extended to the set where $\alpha = E_1$) and that on the line $L$ where $\alpha = E_1$ this extension is given by 
\[
F(\xi,\Re E_1,\Im E_1) =\begin{pmatrix}
- 8 \sqrt{B} \big(\Im\sqrt{E_2}\big) (\xi - \xi_{E_1}) \\ 0 \end{pmatrix}.
\]
In particular, $F(\xi,\Re E_1,\Im E_1)$ vanishes if and only if $\xi =\xi_{E_1}$. This suggests that the zero set of $F$ indeed contains a curve starting at the point $\mathbf{x}_0 = (\xi_{E_1},\Re E_1,\Im E_1)$. However, Lemma \ref{F1lemma} also implies that the extension of $F$ to $\bar{S}_R$ is not $C^1$, because the partial derivatives $\partial_{\alpha_j}F_1$, $j =1,2$, are singular as $\alpha\to E_1$. Thus, in order to apply the implicit function theorem, we will define a modification $\tilde{F}$ of $F$. The singular behavior of $\partial_{\alpha_j}F_1$ stems from the existence of a term proportional to $(\alpha-E_1)\ln(\alpha-E_1)$ in the expansion \eqref{F1allorders} of $F_1$. As motivation for the definition of $\tilde{F}$, we therefore consider the following simple example. 

%-------------------%
%:exa 6.7
%-------------------%
\begin{example}
Consider the function $f\colon(0,1)\to\D{R}$ defined by $f(x) = x\ln x$. Although $f(x)$ has a continuous extension to $x = 0$, the derivative $f'(x) = 1 +\ln x$ is singular at $x = 0$.
However, the modified function $\tilde{f}\colon(0,1) \to\D{R}$ defined by
\[
\tilde{f}(x) = f\bigg(\frac{x}{|\ln x|}\bigg) = -x +\frac{x\ln(|\ln x|)}{\ln x}
\]
is such that both $\tilde{f}(x)$ and its derivative $\tilde{f}'(x) = -1 +\frac{1 + (\ln{x} -1)\ln(|\ln x|)}{(\ln x)^2}$ extend continuously to $x = 0$.
\end{example}
%-------------------%

Employing the standard identification of $\D{C}$ with $\D{R}^2$, we can write $F(\xi,\alpha) \equiv F(\xi,\alpha_1,\alpha_2)$. Let $R > 0$ be small. We define the modified function $\tilde{F}\colon\bar{S}_R \setminus L \to\D{R}^2$ by
\begin{equation}\label{tildeFdef}
\tilde{F}(\xi,\alpha)=\begin{pmatrix} F_1(\xi,\varphi(\alpha)) \\ F_2(\xi,\varphi(\alpha))|\ln|\alpha-E_1||\end{pmatrix},
\end{equation}
where 
\begin{equation}\label{varphidef}
\varphi(\alpha)=E_1+\frac{\alpha-E_1}{|\ln|\alpha-E_1||}.
\end{equation}
There is an $r\in(0,R)$ such that $(\xi,\alpha) \mapsto (\xi,\varphi(\alpha))$ is a diffeomorphism from $\bar{S}_r \setminus L$ onto a subset of $\bar{S}_R \setminus L$. Then, since $F\colon\bar{S}_R\setminus L\to\D{R}^2$ is smooth, $\tilde{F}\colon\bar{S}_r\setminus L\to\D{R}^2$ is also smooth. The next lemma shows that $\tilde{F}$ extends to a $C^1$ map $\bar{S}_r\to\D{R}^2$.

%-------------------%
%:rem 6.8
%-------------------%
\begin{remark}
In addition to incorporating the dilation defined by $\varphi$, the definition of $\tilde{F}$ also includes a factor of $|\ln|\alpha-E_1||$ in the second component. This factor has been included in order to make the partial derivative $\partial \tilde{F}/\partial \alpha_2$ nonzero at $\mathbf{x}_0$ (so that we later can apply the implicit function theorem at $\mathbf{x}_0$). 
\end{remark}
%-------------------%

%-------------------%
%:lem 6.9
%-------------------%
\begin{lemma}\label{tildeFlemma}
The map $\tilde{F}\colon\bar{S}_r\setminus L\to\D{R}^2$ and its Jacobian matrix of first order partial derivatives
\[
D\tilde{F}(\mathbf{x}) =\begin{pmatrix}\partial_{\xi}\tilde{F}_1&\partial_{\alpha_1}\tilde{F}_1&\partial_{\alpha_2}\tilde{F}_1\\
\partial_{\xi}\tilde{F}_2&\partial_{\alpha_1}\tilde{F}_2&\partial_{\alpha_2}\tilde{F}_2\end{pmatrix}
\]
can be extended to continuous maps on $\bar{S}_r$. Moreover, this extension satisfies
\[
\tilde{F}(\mathbf{x}_0) = 0,\qquad
D\tilde{F}(\mathbf{x}_0) =\begin{pmatrix} f_0'(\xi_{E_1})  & -\Im d_{00}(\xi_{E_1}) & -\Re d_{00}(\xi_{E_1})
	\\
0 & q_1(\xi_{E_1}) & q_2(\xi_{E_1})
\end{pmatrix}.
\]
\end{lemma}
%-------------------%

%-------------------%
\begin{proof}
The proof consists of long but straightforward computations using the Taylor expansions of Lemma \ref{F2lemma} and Lemma \ref{F1lemma}. 
Since
\[
\varphi(\alpha)-E_1 =\frac{\alpha-E_1}{|\ln|\alpha-E_1||},
\]
we find from the Taylor expansions \eqref{F2expansions} and \eqref{F1expansions} that
\begin{align*}
&\tilde{F}_1(\mathbf{x})=f_0(\xi) +\ord\bigg(\bigg|\frac{\alpha-E_1}{\ln|\alpha-E_1|}\bigg|\bigg(1 +\ln \bigg|\frac{\alpha-E_1}{\ln|\alpha-E_1|}\bigg|\bigg)\bigg),\\
&\partial_{\xi}\tilde{F}_1(\mathbf{x}) = f_0'(\xi) +\ord\bigg(\bigg|\frac{\alpha-E_1}{\ln|\alpha-E_1|}\bigg|\bigg(1 +\ln \bigg|\frac{\alpha-E_1}{\ln|\alpha-E_1|}\bigg|\bigg)\bigg),
	\\
& \tilde{F}_2(\mathbf{x}) =\ord(\alpha-E_1),
\qquad\partial_{\xi}\tilde{F}_2(\mathbf{x}) =\ord(\alpha-E_1),
\end{align*}
which shows that these functions have continuous extensions to $\bar{S}_r$. Write $\varphi(\alpha) =\varphi_1(\alpha) +\ii \varphi_2(\alpha)$. Using that
\[
\partial_{\alpha_1}\ln|\alpha-E_1| =\frac{\alpha_1 - \Re E_1}{|\alpha-E_1|^2},\qquad\partial_{\alpha_2}\ln|\alpha-E_1| =\frac{\alpha_2 - \Im E_1}{|\alpha-E_1|^2},
\]
we find
\begin{align*}
\partial_{\alpha_1}\varphi_1(\alpha) 
&=\frac{1}{|\ln|\alpha-E_1||} +\frac{(\alpha_1 - \Re E_1)^2}{|\ln|\alpha-E_1||^2|\alpha-E_1|^2}\\
&=\frac{1}{|\ln|\alpha-E_1||} +\ord\bigg(\frac{1}{|\ln|\alpha-E_1||^2}\bigg),\\
\partial_{\alpha_2}\varphi_2(\alpha) 
&=\frac{1}{|\ln|\alpha-E_1||} +\frac{(\alpha_2 - \Im E_1)^2}{|\ln|\alpha-E_1||^2|\alpha-E_1|^2}	\\
&=\frac{1}{|\ln|\alpha-E_1||} +\ord\bigg(\frac{1}{|\ln|\alpha-E_1||^2}\bigg),\\
\partial_{\alpha_2}\varphi_1(\alpha) 
&=\partial_{\alpha_1}\varphi_2(\alpha) =\frac{(\alpha_1 - \Re E_1)(\alpha_2 - \Im E_1)}{|\ln|\alpha-E_1||^2|\alpha-E_1|^2}
=\ord\bigg(\frac{1}{|\ln|\alpha-E_1||^2}\bigg).
\end{align*}
Hence, by \eqref{F1expansions},
\begin{align*}
\partial_{\alpha_1}\tilde{F}_1(\xi,\alpha) 
&=\partial_{\alpha_1}F_1(\xi,\varphi(\alpha))\partial_{\alpha_1}\varphi_1(\alpha)+\partial_{\alpha_2}F_1(\xi,\varphi(\alpha))\partial_{\alpha_1}\varphi_2(\alpha)\\
&=\Big[\Im\{d_{00}(\xi) \ln(\varphi(\alpha)-E_1)\} +\ord(1)\Big]\bigg[\frac{1}{|\ln|\alpha-E_1||} +\ord\bigg(\frac{1}{|\ln|\alpha-E_1||^2}\bigg)\bigg]\\
&\quad+\Big[\Im\{\ii d_{00}(\xi)\ln(\varphi(\alpha)-E_1)\} +\ord(1)\Big]\ord\bigg(\frac{1}{|\ln|\alpha-E_1||^2}\bigg)\\
&=\frac{\Im\{d_{00}(\xi) \ln(\alpha-E_1)\}}{|\ln|\alpha-E_1||}
+\ord\bigg(\frac{\ln|\ln|\alpha-E_1||}{\ln|\alpha-E_1|}\bigg)\\
&=-\Im d_{00}(\xi)+\ord\bigg(\frac{\ln|\ln|\alpha-E_1||}{\ln|\alpha-E_1|}\bigg)
\end{align*}
and
\begin{align*}
\partial_{\alpha_2}\tilde{F}_1(\xi,\alpha) 
&=\partial_{\alpha_1}F_1(\xi,\varphi(\alpha))\partial_{\alpha_2}\varphi_1(\alpha)+\partial_{\alpha_2}F_1(\xi,\varphi(\alpha))\partial_{\alpha_2}\varphi_2(\alpha)\\
&=\Big[\Im\accol{d_{00}(\xi)\ln(\varphi(\alpha)-E_1)}+\ord(1)\Big]\ord\bigg(\frac{1}{|\ln|\alpha-E_1||^2}\bigg)	\\
&\quad+\Big[\Im\{\ii d_{00}(\xi)\ln(\varphi(\alpha)-E_1)\} +\ord(1)\Big]\bigg[\frac{1}{|\ln|\alpha-E_1||} +\ord\bigg(\frac{1}{|\ln|\alpha-E_1||^2}\bigg)\bigg]\\
&=\frac{\Im\accol{\ii d_{00}(\xi)\ln(\alpha-E_1)}}{|\ln|\alpha-E_1||}+\ord\bigg(\frac{\ln|\ln|\alpha-E_1||}{\ln|\alpha-E_1|}\bigg)\\
&= -\Re d_{00}(\xi)+\ord\bigg(\frac{\ln|\ln|\alpha-E_1||}{\ln|\alpha-E_1|}\bigg).
\end{align*}
Similarly, by \eqref{F2expansions},
\begin{align*}
\partial_{\alpha_1}\tilde{F}_2(\xi,\alpha) 
&=\partial_{\alpha_1}F_2(\xi,\varphi(\alpha))\partial_{\alpha_1}\varphi_1(\alpha)|\ln|\alpha-E_1||+\partial_{\alpha_2}F_2(\xi,\varphi(\alpha))\partial_{\alpha_1}\varphi_2(\alpha)|\ln|\alpha-E_1||\\
&\quad+F_2(\xi,\varphi(\alpha))\bigg(-\frac{\alpha_1-\Re E_1}{|\alpha-E_1|^2}\bigg)\\
&=\bigg[q_1(\xi)+\ord\bigg(\frac{|\alpha-E_1|}{|\ln|\alpha-E_1||}\bigg)\bigg]\bigg[1+\ord\bigg(\frac{1}{|\ln|\alpha-E_1||}\bigg)\bigg]\\
&\quad+\bigg[q_2(\xi)+\ord\bigg(\frac{|\alpha-E_1|}{|\ln|\alpha-E_1||}\bigg)\bigg]\ord\bigg(\frac{1}{|\ln|\alpha-E_1||}\bigg)\\
&\quad+\ord\bigg(\frac{|\alpha-E_1|}{|\ln|\alpha-E_1||}\bigg)\bigg(-\frac{\alpha_1-\Re E_1}{|\alpha-E_1|^2}\bigg)\\
&=q_1(\xi)+\ord\bigg(\frac{1}{|\ln|\alpha-E_1||}\bigg)
\end{align*}
and
\begin{align*}
\partial_{\alpha_2}\tilde{F}_2(\xi,\alpha) 
&=\partial_{\alpha_1}F_2(\xi,\varphi(\alpha))\partial_{\alpha_2}\varphi_1(\alpha)\ln|\alpha-E_1|+\partial_{\alpha_2}F_2(\xi,\varphi(\alpha))\partial_{\alpha_2}\varphi_2(\alpha)\ln|\alpha-E_1|\\
&\quad+F_2(\xi,\varphi(\alpha))\bigg(-\frac{\alpha_2-\Im E_1}{|\alpha-E_1|^2}\bigg)\\
&=\bigg[q_1(\xi)+\ord\bigg(\frac{|\alpha-E_1|}{|\ln|\alpha-E_1||}\bigg)\bigg]\ord\bigg(\frac{1}{|\ln|\alpha-E_1||}\bigg)\\
&\quad+\bigg[q_2(\xi) +\ord\bigg(\frac{|\alpha-E_1|}{|\ln|\alpha-E_1||}\bigg)\bigg]\bigg[1+\ord\bigg(\frac{1}{|\ln|\alpha-E_1||}\bigg)\bigg]\\
&\quad+\ord\bigg(\frac{|\alpha-E_1|}{\ln|\alpha-E_1|}\bigg)\bigg(- \frac{\alpha_2-\Im E_1}{|\alpha-E_1|^2}\bigg)\\
&=q_2(\xi)+\ord\bigg(\frac{1}{|\ln|\alpha-E_1||}\bigg).
\end{align*}
The statements of the lemma follow from the above expansions.
\end{proof}
%-------------------%

Lemma \ref{tildeFlemma} implies that $\tilde{F}\colon\bar{S}_r\to\D{R}^2$ is a $C^1$ map such that $\tilde{F}(\mathbf{x}_0) = 0$ and 
\[
\det\begin{pmatrix}
\partial_{\xi}\tilde{F}_1&\partial_{\alpha_2}\tilde{F}_1\\
\partial_{\xi}\tilde{F}_2&\partial_{\alpha_2}\tilde{F}_2
\end{pmatrix}
=\det\begin{pmatrix} 
f_0'(\xi_{E_1})&-\Re d_{00}(\xi_{E_1})\\
0 & q_2(\xi_{E_1})
\end{pmatrix}
=-8\sqrt{B}\big(\Im\sqrt{E_2}\big)q_2(\xi_{E_1})\neq 0,
\]
where we have used the fact that $q_2(\xi_{E_1})\neq 0$ (see Lemma \ref{F2lemma}) in the last step. Hence we can apply the implicit function theorem to conclude that there exists a $\delta > 0$ and a $C^1$-curve 
\begin{align*}
&\gamma\colon\croch{\Re E_1,\Re E_1 +\delta}\to\bar{S}_r\\
&\alpha_1\mapsto\gamma(\alpha_1)=(\xi(\alpha_1),\alpha_1,\alpha_2(\alpha_1))
\end{align*}
such that $\gamma(\Re E_1)=\mathbf{x}_0$, the function $\tilde{F}$ vanishes identically on the image of $\gamma$, and
\[
\begin{pmatrix}\xi'(\alpha_1)\\ \alpha_2'(\alpha_1)\end{pmatrix}= -\begin{pmatrix}\partial_{\xi}\tilde{F}_1&\partial_{\alpha_2}\tilde{F}_1\\\partial_{\xi}\tilde{F}_2&\partial_{\alpha_2}\tilde{F}_2\end{pmatrix}^{-1}\begin{pmatrix}\partial_{\alpha_1}\tilde{F}_1\\ \partial_{\alpha_1}\tilde{F}_2\end{pmatrix}.
\]
The technical complication that $\mathbf{x}_0$ lies on the boundary of $\bar{S}_r$ can be overcome either by appealing to a boundary version of the implicit function theorem (see \cite{D1913}*{Theorem 5}) or by first constructing a $C^1$ extension of $\tilde{F}$ to an open neighborhood of $\mathbf{x}_0$ (the existence of such an extension follows, for example, from the Whitney extension theorem) and then applying the standard implicit function theorem. 

It follows from the definition \eqref{tildeFdef} of $\tilde{F}$ that $F$ vanishes on the image of the curve $\Phi \circ \gamma$, where $\Phi$ denotes the map $(\xi,\alpha) \mapsto (\xi,\varphi(\alpha))$ which is a bijection from $\bar{S}_r$ to a subset of $\bar{S}_R$. At the endpoint $\mathbf{x}_0$, a computation gives
\begin{align*}
\begin{pmatrix}\xi'(\Re E_1)\\ \alpha_2'(\Re E_1)\end{pmatrix}
&=-\begin{pmatrix} f_0'(\xi_{E_1})&-\Re d_{00}(\xi_{E_1})\\0 & q_2(\xi_{E_1})\end{pmatrix}^{-1}\begin{pmatrix}-\Im d_{00}(\xi_{E1})\\q_1(\xi_{E_1})\end{pmatrix}\\
&=\begin{pmatrix}\frac{-3 A^4+7A^2B\left(|E_2|-5 B\right)-8 B^3\left(|E_2|+3B\right)}{2B^2\left(9 A^2+8 B^2\right)}\\[1mm]
\frac{A (|E_2|-B)}{B (B+3 |E_2|)}
\end{pmatrix}.
\end{align*}
In particular, $\xi'(\Re E_1)<0$ and $\alpha_2'(\Re E_1)>0$.

We finally show \eqref{alphaexpansionE1}. Let $t \mapsto \gamma(t)$ be a parametrization of $\gamma$ such that $\gamma(0) =\mathbf{x}_0$. Since $\gamma$ is $C^1$, we have
\[
\gamma(t) =\mathbf{x}_0 + (at, bt, ct) +\osmall(t)\qquad t\downarrow 0,
\]
where $\gamma'(0) = (a,b,c)$ with $b>0$ is proportional to $(\xi'(\Re E_1), 1,\alpha_2'(\Re E_1))$; in particular, $a < 0$ and $c > 0$. Letting $w \coloneqq b +\ii c$, we find
\[
\Phi(\gamma(t)) 
=\mathbf{x}_0 +\bigg(at +\osmall(t),\frac{w t +\osmall(t)}{|\ln|wt +\osmall(t)||}\bigg)
=\mathbf{x}_0 +\bigg(at +\osmall(t),\frac{wt}{|\ln t|} +\osmall\bigg(\frac{t}{|\ln t|}\bigg)\bigg).
\]
Introducing a new parameter $s$ by $s=-at+\osmall(t)$, this becomes
\[
\Phi(\gamma(t)) 
=\mathbf{x}_0 +\bigg(-s,\frac{c_1s}{|\ln s|} +\osmall\bigg(\frac{s}{|\ln s|}\bigg)\bigg),\qquad s \downarrow 0,
\]
where
\[
c_1 \coloneqq -\frac{w}{a} = -\frac{1 +\ii  \alpha_2'(\Re E_1)}{\xi'(\Re E_1)}=\frac{2BE_1}{A^2 + 4\ii AB - 3B^2 - B|E_2|}
\]
satisfies $\Re c_1 >0$ and $\Im c_1 > 0$. In terms of the curve $\alpha(\xi)$ in \eqref{behavioratxiE1}, this can be expressed as (let $s =\xi_{E_1}-\xi$)
\[
\alpha(\xi) = E_1 + c_1 \frac{\xi_{E_1} - \xi}{|\ln(\xi_{E_1}- \xi)|} +\osmall\bigg(\frac{\xi_{E_1} - \xi}{|\ln(\xi_{E_1} - \xi)|}\bigg),\qquad \xi \uparrow \xi_{E_1},
\]
which proves \eqref{alphaexpansionE1}. This completes the proof of Theorem \ref{genus2existenceth}.
%-------------------%
\begin{acknowledgements*}
The authors are grateful to the two referees whose comments and suggestions have improved the manuscript. J.~Lenells acknowledges support from the G\"oran Gustafsson Foundation, the Ruth and Nils-Erik Stenb\"ack Foundation, the Swedish Research Council, Grant No.~2015-05430, and the European Research Council, Grant Agreement No.~682537.
\end{acknowledgements*}
%---------------------------------------------------------%
%:bib
%---------------------------------------------------------%

%---------------------------------------------------------%
\end{document}